\documentclass[reqno,10pt,centertags]{amsart}
\usepackage{amsmath,amsthm,amscd,amssymb,latexsym,esint,upref,stmaryrd,
enumerate,color,verbatim,yfonts}
\usepackage{hyperref}

\newcommand*{\mailto}[1]{\href{mailto:#1}{\nolinkurl{#1}}}

\newcommand{\C}{{\mathbb C}}

\newcommand{\bbC}{{\mathbb{C}}}

\newcommand{\bbN}{{\mathbb{N}}}

\newcommand{\bbR}{{\mathbb{R}}}

\newcommand{\bsA}{{\boldsymbol{A}}}
\newcommand{\bsB}{{\boldsymbol{B}}}

\newcommand{\bsD}{{\boldsymbol{D}}}

\newcommand{\bsH}{{\boldsymbol{H}}}
\newcommand{\bsI}{{\boldsymbol{I}}}
\newcommand{\bsJ}{{\boldsymbol{J}}}
\newcommand{\bsK}{{\boldsymbol{K}}}
\newcommand{\bsL}{{\boldsymbol{L}}}

\newcommand{\bsR}{{\boldsymbol{R}}}
\newcommand{\bsS}{{\boldsymbol{S}}}
\newcommand{\bsT}{{\boldsymbol{T}}}
\newcommand{\bsU}{{\boldsymbol{U}}}
\newcommand{\bsV}{{\boldsymbol{V}}}

\newcommand{\bsu}{{\boldsymbol{u}}}
\newcommand{\bsv}{{\boldsymbol{v}}}

\newcommand{\cB}{{\mathcal B}}
\newcommand{\cC}{{\mathcal C}}
\newcommand{\cD}{{\mathcal D}}

\newcommand{\cF}{{\mathcal F}}

\newcommand{\cH}{{\mathcal H}}

\newcommand{\cK}{{\mathcal K}}

\newcommand{\e}{\varepsilon}

\DeclareMathOperator{\supp}{supp}

\DeclareMathOperator{\ran}{ran}
\DeclareMathOperator{\dom}{dom}

\DeclareMathOperator{\tr}{tr}
\DeclareMathOperator{\spr}{spr}

\DeclareMathOperator*{\nlim}{n-lim}
\DeclareMathOperator*{\slim}{s-lim}

\renewcommand{\Im}{\text{\rm Im}}
\renewcommand{\ln}{\text{\rm ln}}

\newcommand{\loc}{\text{\rm{loc}}}
\newcommand{\ind}{\operatorname{ind}}
\newcommand{\no}{\notag}
\newcommand{\lb}{\label}
\newcommand{\f}{\frac}

\newcommand{\ol}{\overline}
\newcommand{\bs}{\backslash}

\newcommand{\wti}{\widetilde}
\newcommand{\Oh}{O}
\newcommand{\oh}{o}
\newcommand{\hatt}{\widehat} 
\newcommand{\bi}{\bibitem}

\let\geq\geqslant
\let\leq\leqslant

\makeatletter
\def\theequation{\@arabic\c@equation}
\allowdisplaybreaks 
\numberwithin{equation}{section}

\newtheorem{theorem}{Theorem}[section]

\newtheorem{lemma}[theorem]{Lemma}
\newtheorem{corollary}[theorem]{Corollary}
\newtheorem{definition}[theorem]{Definition}
\newtheorem{hypothesis}[theorem]{Hypothesis}

\theoremstyle{remark}
\newtheorem{remark}[theorem]{Remark}

\begin{document}

\numberwithin{equation}{section}
\allowdisplaybreaks

\title[Reduction of Fredholm Determinants]{A Jost--Pais-Type Reduction of Fredholm Determinants and Some Applications}

\author[A.\ Carey]{Alan Carey}  
\address{Mathematical Sciences Institute, Australian National University, Kingsley St., 
Canberra, ACT 0200, Australia}  
\email{\mailto{acarey@maths.anu.edu.au}}
%\email{acarey@maths.anu.edu.au}  
\urladdr{\url{http://maths.anu.edu.au/~acarey/}}
%\urladdr{http://maths.anu.edu.au/~acarey/}
  
\author[F.\ Gesztesy]{Fritz Gesztesy}  
\address{Department of Mathematics,
University of Missouri, Columbia, MO 65211, USA}
\email{\mailto{gesztesyf@missouri.edu}}
%\email{gesztesyf@missouri.edu}
\urladdr{\url{http://www.math.missouri.edu/personnel/faculty/gesztesyf.html}}
%\urladdr{http://www.math.missouri.edu/personnel/faculty/gesztesyf.html}

\author[D.\ Potapov]{Denis Potapov}
\address{School of Mathematics and Statistics, UNSW, Kensington, NSW 2052,
Australia} 
\email{\mailto{d.potapov@unsw.edu.au}}
%\email{d.potapov@unsw.edu.au}

\author[F.\ Sukochev]{Fedor Sukochev}
\address{School of Mathematics and Statistics, UNSW, Kensington, NSW 2052,
Australia} 
\email{\mailto{f.sukochev@unsw.edu.au}}
%\email{f.sukochev@unsw.edu.au}

\author[Y.\ Tomilov]{Yuri Tomilov}
\address{Faculty of Mathematics and Computer Science, Nicholas 
Copernicus University, ul.\ Chopina 12/18, 87-100 Torun, Poland, and Institute of Mathematics, Polish Academy of Sciences. \'Sniadeckich str. 8, 00-956 Warsaw, Poland}
\email{\mailto{tomilov@mat.uni.torun.pl}}
%\email{tomilov@mat.uni.torun.pl}

\thanks{{\it Integral Eq. Operator Theory} (to appear).}

\date{\today}
\subjclass[2010]{Primary: 47B10, 47G10, Secondary: 34B27, 34L40.}
\keywords{Fredholm determinants, semi-separable kernels, Jost functions,
perturbation determinants.}

%%%%%%%%%%%
%%%%%%%%%%% 
\begin{abstract} 
We study the analog of semi-separable integral kernels in $\cH$ of the type 
\begin{equation*}
K(x,x')=\begin{cases} F_1(x)G_1(x'), & a<x'< x< b, \\ 
F_2(x)G_2(x'), & a<x<x'<b,  
\end{cases}   
\end{equation*}
where $-\infty\leq a<b\leq \infty$, and for a.e.\ $x \in (a,b)$, $F_j (x) \in \cB_2(\cH_j,\cH)$ 
and $G_j(x) \in \cB_2(\cH,\cH_j)$ 
such that $F_j(\cdot)$ and $G_j(\cdot)$ are uniformly measurable, and 
\begin{equation*}  
\|F_j( \cdot)\|_{\cB_2(\cH_j,\cH)} \in L^2((a,b)), \; 
\|G_j (\cdot)\|_{\cB_2(\cH,\cH_j)} \in L^2((a,b)), \quad j=1,2,   
\end{equation*}
with $\cH$ and $\cH_j$, $j=1,2$, complex, separable Hilbert spaces. Assuming 
that $K(\cdot, \cdot)$ generates a trace class operator $\bsK$ in 
$L^2((a,b);\cH)$, we derive the analog of the Jost--Pais reduction theory that 
succeeds in proving that the Fredholm determinant 
${\det}_{L^2((a,b);\cH)}(\bsI - \alpha \bsK)$, $\alpha \in \bbC$, naturally reduces 
to appropriate Fredholm determinants in the Hilbert spaces $\cH$ (and 
$\cH_1 \oplus \cH_2$). 
   
Explicit applications of this reduction theory are made to Schr\"odinger operators 
with suitable bounded operator-valued potentials. In addition, we provide an 
alternative approach to a fundamental trace formula first established by Pushnitski 
which leads to a Fredholm index computation of a certain model operator. 
\end{abstract}
%%%%%%%%%%
%%%%%%%%%%

\maketitle

%\newpage 

{\scriptsize{\tableofcontents}}
%\normalsize

%%%%%%%%%%%%%%%%%%%%%%%%%%%%%%
%%%%%%%%%%%%%%%%%%%%%%%%%%%%%%
\section{Introduction}  \lb{s1}
%%%%%%%%%%%%%%%%%%%%%%%%%%%%%%
%%%%%%%%%%%%%%%%%%%%%%%%%%%%%%

The principal topic in this paper concerns semi-separable integral operators and 
their associated Fredholm determinants. In a nutshell, suppose that $\cH$ 
and $\cH_j$, $j=1,2$, are complex, separable Hilbert spaces, that 
$-\infty\leq a<b\leq \infty$, and introduce the semi-separable integral kernel
in $\cH$, 
\begin{equation}
K(x,x')=\begin{cases} F_1(x)G_1(x'), & a<x'< x< b, \\ 
F_2(x)G_2(x'), & a<x<x'<b,  \lb{1.1}
\end{cases}   
\end{equation}
where for a.e.\ $x \in (a,b)$, $F_j (x) \in \cB(\cH_j,\cH)$ and $G_j(x) \in \cB(\cH,\cH_j)$ 
such that $F_j(\cdot)$ and $G_j(\cdot)$ are uniformly measurable (i.e., measurable with 
respect to the uniform operator topology), and 
\begin{equation}  
\|F_j( \cdot)\|_{\cB_2(\cH_j,\cH)} \in L^2((a,b)), \; 
\|G_j (\cdot)\|_{\cB_2(\cH,\cH_j)} \in L^2((a,b)), \quad j=1,2.    
\end{equation}
Assuming that $K(\cdot, \cdot)$ generates a trace class operator $\bsK$ in 
$L^2((a,b);\cH)$, we derive the analog of the Jost--Pais reduction theory that 
naturally reduces the Fredholm determinant 
${\det}_{L^2((a,b);\cH)}(\bsI - \alpha \bsK)$, $\alpha \in \bbC$, to appropriate 
Fredholm determinants in the Hilbert spaces $\cH$ and $\cH_1 \oplus \cH_2$ as 
described in detail in Theorem \ref{tA.13}. For instance, we will prove the 
remarkable Jost--Pais-type reduction of Fredholm determinants \cite{JP51} (see also \cite{Ge86}, \cite{Ne80}), 
\begin{align}
\begin{split} 
{\det}_{L^2((a,b);\cH)}(\bsI-\alpha \bsK)
& = {\det}_{\cH_1}\bigg(I_{\cH_1}-\alpha \int_a^b dx\,
G_1(x)\widehat F_1(x; \alpha )\bigg) \lb{1.101a} \\
& ={\det}_{\cH_2}\bigg(I_{\cH_2}-\alpha \int_a^b dx\,
G_2(x)\widehat F_2(x; \alpha )\bigg),  
\end{split} 
\end{align}
where $\widehat F_1(\cdot; \alpha )$ and $\widehat F_2(\cdot; \alpha )$
are defined via the Volterra integral equations
\begin{align}
\widehat F_1(x; \alpha )&=F_1(x)-\alpha \int_x^b dx'\, H(x,x')\widehat F_1(x'; \alpha ), \lb{1.35} \\ 
\widehat F_2(x; \alpha )&=F_2(x)+\alpha \int_a^x dx'\, H(x,x')\widehat F_2(x'; \alpha ) \lb{1.36}  
\end{align}  
(cf.\ \eqref{A.6} for the definition of $H(\cdot  , \cdot )$). 

To illustrate the ubiquity of semi-separable integral operators it suffices to consider the special 
finite-dimensional case and note that the integral kernel of the resolvent of any ordinary differential 
and finite difference operator with matrix-valued coefficients, on arbitrary intervals on the real line, 
yields a Green's matrix of the type \eqref{1.1}, cf.\ \cite[Sect.\ XIV.3]{GGK90}. (The same applies 
to certain classes of convolution operators, cf.\ \cite[Sect.\ XIII.10]{GGK90}.) In particular, 
Schr\"odinger, Dirac, Jacobi, and CMV operators of great relevance to mathematical physics, are 
prime candidates to which this circle of ideas applies. In these cases the determinant reduction formulas  
\eqref{1.101a}  lead to natural extensions of well-known results due to Jost--Pais 
\cite{JP51}. We also note that Jost functions of the type \eqref{1.101a}  
are intimately related to Evans functions, a fundamental 
tool in linear stability theory associated with classes of non linear evolution equations. In the latter 
context we note the frequent necessity to consider non-self-adjoint operators as the result of a 
linearization process and stress that (infinite) determinants are ideally suited to analyze certain spectral 
properties of non-self-adjoint operators. Moreover, as shown in \cite{GM03}, suitable $2$-modified 
Fredholm determinant extensions of this approach also apply to convolution integral operators, whose
kernel is associated with a symbol given by a rational function. The corresponding determinant 
formula then represents a Wiener--Hopf analog of Day's formula for the determinant associated
with finite Toeplitz matrices generated by the Laurent expansion of a rational function. In addition, 
we note that this circle of ideas applies to Floquet theory and relates the Fredholm determinant of 
a particular Birman--Schwinger-type operator (modeled by $\bsK$ in \eqref{1.101a}) to the Floquet 
discriminant, a standard device in 
the theory of periodic differential and difference equations. In the instance where matrix-valued coefficients 
are replaced by operator-valued coefficients, one can now apply the results developed in this paper in 
the infinite-dimensional extensions reflected in \eqref{1.101a}.  

In Section \ref{s2}, the reduction theory leading to \eqref{1.101a} is presented in 
detail, culminating in Theorem \ref{tA.13}. In order to discuss concrete applications of this reduction 
theory for Fredholm determinants, Appendix \ref{sA} recalls the basic 
Weyl--Titchmarsh theory for Schr\"odinger operators with bounded 
operator-valued potentials on a half-line and on $\bbR$. Section \ref{s3}  
provides the application of Section \ref{s2} to Schr\"odinger operators with 
operator-valued potentials, culminating in Theorem \ref{tB.3}, which is of interest in its own right. 
Section \ref{s4} recalls the computation of the Fredholm index of the model operator 
$\bsD_\bsA^{}$ in $L^2(\bbR;\cH)$ (cf.\ \cite{GLMST11}, \cite{Pu08}, 
\cite{Ra04}, \cite{RS95} and the references therein) in terms of appropriate spectral 
shift functions. Here 
\begin{equation}
\bsD_\bsA^{} = \f{d}{dt} + \bsA,
\quad \dom(\bsD_\bsA^{})= \dom(d/dt).   \lb{1.9}
\end{equation}
with the operator $d/dt$ in $L^2(\bbR;\cH)$  is defined by
\begin{align} 
& \bigg(\f{d}{dt}f\bigg)(t) = f'(t) \, \text{ for a.e.\ $t\in\bbR$,} 
\lb{1.10}  \\
& \, f \in \dom(d/dt) = \big\{g \in L^2(\bbR;\cH) \, \big|\,
g \in AC_{\loc}(\bbR; \cH); \, g' \in L^2(\bbR;\cH)\big\} = W^{2,1}(\bbR;\cH),   \no
\end{align}
and the operator $\bsA \in \cB\big(L^2(\bbR;\cH)\big)$ is associated with the family of 
bounded self-adjoint operators $A(t) \in \cB(\cH)$, $t\in\bbR$, in $\cH$ by
\begin{equation}
(\bsA f)(t) = A(t) f(t) \, \text{ for a.e.\ $t\in\bbR$,}   \quad f \in L^2(\bbR;\cH).    
\lb{1.8}
\end{equation}
(For the precise hypotheses on $A(t)$, $t \in \bbR$ we refer to \eqref{2.3} and 
Hypothesis \ref{h4.1}.) 

Our hypotheses on $A(\cdot)$ guarantee the existence of bounded, self-adjoint asymptotes 
$A_{\pm} = \nlim_{t  \to \pm \infty} A(t) \in \cB(\cH)$. In addition, we prove that $\bsD_\bsA^{}$ is a 
Fredholm operator in $L^2(\bbR; \cH)$ if and only if $0 \in \rho(A_+)\cap\rho(A_-)$. 

We emphasize that we 
do not make the assumption that the operators $A_\pm$ have discrete spectrum. 
We then prove the following equalities for the Fredholm index of $\bsD_\bsA^{}$: 
\begin{align}
\ind (\bsD_\bsA^{}) & = \xi(0_+; \bsH_2, \bsH_1)    \lb{1.38} \\  
& = \xi(0;A_+,A_-),     \lb{1.39} 
\end{align}
where $\xi(\,\cdot\,;S_2,S_1)$ denotes the spectral shift function for the pair of 
self-adjoint operators $(S_2, S_1)$, and where 
\begin{align} \lb{1.14}
\bsH_1=\bsD_\bsA^* \bsD_\bsA^{} = - \f{d^2}{dt^2} \dot + \bsV_1, \quad 
\dom\big(\bsH_1^{1/2}\big) = W^{2,1}(\bbR;\cH), 
\quad \bsV_1 = \bsA^2 - \bsA',    \\
\bsH_2=\bsD_\bsA^{} \bsD_\bsA^* = - \f{d^2}{dt^2} \dot + \bsV_2, \quad 
\dom\big(\bsH_2^{1/2}\big) = W^{2,1}(\bbR;\cH), 
\quad \bsV_2 = \bsA^2 + \bsA',
\end{align} 
and the operator $\bsA^{\prime}$ is defined in terms of 
the family $\{A'(t)\}_{t\in\bbR}$ in $\cH$ as in \eqref{2.8a}. Thus, 
\begin{equation} 
\bsH_2=\bsH_1 \dot + 2 \bsA'.   \lb{1.9a} 
\end{equation} 
Here the symbol $\dot +$ abbreviates the form sum (and we note that 
$\bsA^2 \in \cB\big(L^2(\bbR;\cH)\big)$). 

Equation \eqref{1.38} follows from the fundamental trace formula 
\begin{align}
\begin{split}
   \lb{1.19}
   \tr_{L^2(\bbR;\cH)}\big((\bsH_2 - z \, \bsI)^{-1}-(\bsH_1 - z \, 
\bsI)^{-1}\big) = \frac{1}{2z} \tr_\cH \big(g_z(A_+)-g_z(A_-)\big),&  \\  
z\in\bbC\backslash [0,\infty).&  
\end{split}
\end{align} 
originally due to Pushnitski \cite{Pu08} (extended in \cite{GLMST11}), where 
 \begin{equation} \lb{1.16} 
g_z(x) = x(x^2-z)^{-1/2}, \quad z\in\C\backslash [0,\infty), \; x\in\bbR.   
\end{equation}
Our final Section \ref{s5} then provides a new and Fredholm determinant based 
proof of the trace formula \eqref{1.19}. In addition, introducing the symmetrized 
perturbation determinant $\wti D_{\bsH_2/\bsH_1} (z)$ by 
\begin{align}
\begin{split} 
& \wti D_{\bsH_2/\bsH_1} (z) = {\det}_{L^2(\bbR;\cH)} 
\big((\bsH_1 - z \bsI)^{-1/2} (\bsH_2 - z \bsI) (\bsH_1 - z \bsI)^{-1/2}\big)    \\
& \quad = {\det}_{L^2(\bbR;\cH)} 
\big(\bsI + 2 (\bsH_1 - z \bsI)^{-1/2} \bsA' (\bsH_1 - z \bsI)^{-1/2}\big), 
\quad z \in \rho(\bsH_1),     \lb{1.16a} 
\end{split} 
\end{align}
we prove the new result 
\begin{align}
& \wti D_{\bsH_2/\bsH_1} (z) = {\det}_{L^2(\bbR;\cH)} (\bsI - \bsK(z))     \lb{1.54} \\
& \quad = - {\det}_{\cH} \Big(z^{-1} \big[(A_+^2 - z I_{\cH})^{1/2} + A_+\big] 
\big[-(A_-^2 - z I_{\cH})^{1/2} + A_-\big]\Big), \quad z \in \rho(\bsH_1),      \no 
\end{align}
where the Birman--Schwinger-type integral operator 
$\bsK(z) \in \cB_1\big(L^2(\bbR; \cH)\big)$ is defined by (cf.\ \eqref{1.9a})
\begin{equation}
\bsK(z) = -2 \bsU_{\bsA'} |\bsA'|^{1/2} (\bsH_1 - z \bsI)^{-1} |\bsA'|^{1/2}, 
\quad z \in \rho(\bsH_1),    \lb{1.41} 
\end{equation}
recalling the polar decomposition for $A'(\cdot)$,  
\begin{equation} 
A'(t) = U_{A'(t)} |A'(t)|,  \quad t\in\bbR,    \lb{1.38a} 
\end{equation}   
with $U_{A'(t)}$ a partial isometry in $\cH$, and hence 
\begin{equation}
(\bsU_{\bsA'} f)(t) = U_{A'(t)} f(t) \, \text{ for a.e.\ $t\in\bbR$,}  
\quad f \in L^2(\bbR;\cH).    \lb{1.8A}
\end{equation}

In particular, $\bsK(z)$ has the semi-separable integral kernel $K(z,\cdot , \cdot)$ defined by 
\begin{equation}
K(z,t,t') = - 2 U_{A'(t)} |A'(t)|^{1/2} G_1(z,t,t') |A'(t')|^{1/2}, 
\quad z \in \rho(\bsH_1), \; t, t' \in \bbR,      \lb{1.42} 
\end{equation}
with the semi-separable integral kernel $G_1(z, \cdot, \cdot)$, the integral kernel of 
$(\bsH_1 - z \bsI)^{-1}$, defined in \eqref{3.37}.

Finally, we briefly summarize some of the notation used in this paper: Typically, 
$\cH$ (resp., $\cK$) will be a separable complex Hilbert space, $(\cdot,\cdot)_{\cH}$ 
denotes the scalar product in $\cH$ (linear in the second argument), and $I_{\cH}$ is 
the identity operator in $\cH$.

Next, if $T$ is a linear operator mapping (a subspace of) a Hilbert space into 
another, then $\dom(T)$ and $\ker(T)$ denote the domain and kernel (i.e., 
null space) of $T$. 
The closure of a closable operator $S$ is denoted by $\ol S$. 
The spectrum, essential spectrum, discrete spectrum, point spectrum, and 
resolvent set of a closed linear operator in a Hilbert space will be denoted by 
$\sigma(\cdot)$, $\sigma_{\rm ess}(\cdot)$, $\sigma_{\rm d}(\cdot)$, 
$\sigma_{\rm p}(\cdot)$, and $\rho(\cdot)$, respectively. 

The convergence in the strong operator topology (i.e., pointwise limits) will 
be denoted by $\slim$. Similarly, limits in the norm topology are abbreviated 
by $\nlim$. 

If $T$ is a Fredholm operator, its Fredholm index is denoted by $\ind(T)$.  

The Banach spaces of bounded and compact linear operators between complex, separable 
Hilbert spaces $\cH$ and $\cK$ are denoted by $\cB(\cH, \cK)$ and $\cB_\infty(\cH, \cK)$, 
respectively; the corresponding $\ell^p$-based trace ideals will be denoted 
by $\cB_p (\cH, \cK)$, $p>0$. (In the special case $\cH = \cK$, we will use the notation 
$\cB(\cH)$, $\cB_{\infty}(\cH)$, $\cB_p(\cH)$, $p>0$.) 
Moreover, ${\det}_{\cH}(I_\cH-A)$, and $\tr_{\cH}(A)$ denote the standard 
Fredholm determinant and the corresponding trace 
of a trace class operator $A\in\cB_1(\cH)$. Modified Fredholm determinants are denoted by 
${\det}_{k, \cH}(I_\cH-A)$, $A\in\cB_k(\cH)$, $k \in \bbN$, $k \geq 2$.

We will use the abbreviation $\bbC_+ = \{z\in\bbC\,|\, \Im(z) > 0\}$ for the 
open complex upper half-plane. 

Additional notational conventions  in the context of semi separable integral operators 
and Schr\"odinger operators in $L^2((a,b); dx; \cH)$ are introduced near the beginning 
of Appendix \ref{sA}.

%%%%%%%%%%%%%%%%%%%%%%%%%%%%%%%%%%%%%%
%%%%%%%%%%%%%%%%%%%%%%%%%%%%%%%%%%%%%%
\section{Semiseparable Operators and Reduction Theory for Fredholm 
Determinants} \lb{s2}
%%%%%%%%%%%%%%%%%%%%%%%%%%%%%%%%%%%%%%
%%%%%%%%%%%%%%%%%%%%%%%%%%%%%%%%%%%%%%

In this section we describe one of the basic tools in this paper: a reduction 
theory for Fredholm determinants that permits one to reduce Fredholm 
determinants in the Hilbert space $L^2((a,b);\cH)$ to those in the Hilbert 
space $\cH$, as described in detail in Theorem \ref{tA.13}. More precisely, 
we focus on a particular set of trace class 
operators $\bsK$ in $L^2((a,b);\cH)$ with $\cB(\cH)$-valued semi-separable 
integral kernels (with $\cH$ a complex, separable Hilbert space, generally of 
infinite dimension) and show how naturally to reduce the Fredholm 
determinant ${\det}_{L^2((a,b);\cH)}(\bsI - \alpha \bsK)$, $\alpha \in \bbC$, 
to appropriate Fredholm determinants in Hilbert spaces $\cH$ and 
$\cH \oplus \cH$ (in fact, we will describe a slightly more general framework 
below).   

In our treatment we closely follow the approaches presented in Gohberg,
Goldberg, and Kaashoek \cite[Ch.\ IX]{GGK90} and Gohberg, Goldberg, and
Krupnik \cite[Ch.\ XIII]{GGK00} (see also \cite{GK84}), and especially, in 
\cite{GM03}, where the particular case $\dim(\cH) < \infty$ was treated in 
detail. Our treatment of the case $\dim(\cH) = \infty$ in this section appears 
to be new and we hope it will prove to be of independent interest.  

Before setting up the basic formalism for this section, we state the following 
elementary result: 

%%%%%%%%%%%%
\begin{lemma} \lb{lA.1} 
Let $\cH$ and $\cH'$ be complex, separable Hilbert spaces and $-\infty\leq a<b\leq \infty$. 
Suppose that for a.e.\ $x \in (a,b)$, $F (x) \in \cB(\cH',\cH)$ and $G(x) \in \cB(\cH,\cH')$ 
with $F(\cdot)$ and $G(\cdot)$ uniformly measurable, and  
\begin{equation}  
\|F( \cdot)\|_{\cB(\cH',\cH)} \in L^2((a,b)), \; 
\|G (\cdot)\|_{\cB(\cH,\cH')} \in L^2((a,b)). \lb{A.5}
\end{equation}
Consider the integral operator $\bsS$ in $L^2((a,b);\cH)$ with $\cB(\cH)$-valued 
separable integral kernel of the type 
\begin{equation}
S(x,x') = F(x) G(x') \, \text{ for a.e.\ $x, x' \in (a,b)$.}      \lb{A.6a} 
\end{equation}
Then 
\begin{equation}
\bsS \in \cB\big(L^2((a,b);\cH)\big).   \lb{A.7a}
\end{equation}
\end{lemma}
%%%%%%%%%%%%
\begin{proof}
Let $f\in L^2((a,b);\cH)$, then for a.e.\ $x \in (a,b)$, and any integral 
operator $\bsT$ in $L^2((a,b);\cH)$ with $\cB(\cH)$-valued integral kernel $T(\cdot \, , \cdot)$, 
one obtains  
\begin{align}
\|(\bsT f)(x)\|_{\cH} & \leq \int_a^b dx' \, \|T(x,x')\|_{\cB(\cH)}\|f(x')\|_{\cH}   \no \\
& \leq \bigg(\int_a^b dx' \, \|T(x,x')\|_{\cB(\cH)}^2\bigg)^{1/2} 
\bigg(\int_a^b dx'' \, \|f(x'')\|_{\cH}^2\bigg)^{1/2},      \lb{A.8a}
\end{align}
and hence 
\begin{equation}
\int_a^b dx \, \|(\bsT f)(x)\|_{\cH}^2 \leq \bigg[\int_a^b dx \int_a^b dx' \, \|T(x,x')\|_{\cB(\cH)}^2\bigg]   
\int_a^b dx'' \, \|f(x'')\|_{\cH}^2,       \lb{A.9aa} 
\end{equation}
yields $\bsT \in \cB(L^2((a,b);\cH))$ whenever 
$\Big[\int_a^b dx \int_a^b dx' \, \|T(x,x')\|_{\cB(\cH)}^2\Big] < \infty$, implying
\begin{equation}
\|\bsT\|_{\cB(L^2((a,b);\cH))} \leq 
\bigg(\int_a^b dx \int_a^b dx' \, \|T(x,x')\|_{\cB(\cH)}^2\bigg)^{1/2}.    \lb{A.10a} 
\end{equation}
Thus, using the special form \eqref{A.6a} of $\bsS$ implies 
\begin{align}
\|\bsS\|_{\cB(L^2((a,b);\cH))}^2 &\leq \int_a^b dx \int_a^b dx' \, \|S(x,x')\|_{\cB(\cH)}^2   \no \\
& = \int_a^b dx \int_a^b dx' \, \|F(x) G(x')\|_{\cB(\cH)}^2    \no \\
& \leq \int_a^b dx \, \|F(x)\|_{\cB(\cH',\cH)}^2 \int_a^b dx' \, \|G(x')\|_{\cB(\cH,\cH')}^2 < \infty. 
\lb{A.10A} 
\end{align}
\end{proof}
%%%%%%%%%%%%

At this point we now make the following initial set of assumptions: 
 
%%%%%%%%%%%
\begin{hypothesis} \lb{hA.2}  
Let $\cH$ and $\cH_j$, $j=1,2$, be complex, separable Hilbert spaces and $-\infty\leq a<b\leq \infty$. 
Suppose that for a.e.\ $x \in (a,b)$, $F_j (x) \in \cB(\cH_j,\cH)$ and $G_j(x) \in \cB(\cH,\cH_j)$ 
such that $F_j(\cdot)$ and $G_j(\cdot)$ are uniformly measurable, and 
\begin{equation}  
\|F_j( \cdot)\|_{\cB(\cH_j,\cH)} \in L^2((a,b)), \; 
\|G_j (\cdot)\|_{\cB(\cH,\cH_j)} \in L^2((a,b)), \quad j=1,2. \lb{A.1}
\end{equation}
\end{hypothesis}
%%%%%%%%%%%

Given Hypothesis \ref{hA.2}, we introduce in $L^2((a,b);\cH)$ the operator 
\begin{equation}
(\bsK f)(x)=\int_a^b dx'\, K(x,x')f(x') \, \text{ for a.e.\ $x \in (a,b)$, $f\in L^2((a,b);\cH)$,} 
\end{equation}
with $\cB(\cH)$-valued semi-separable integral kernel
$K(\cdot,\cdot)$ defined by
\begin{equation}
K(x,x')=\begin{cases} F_1(x)G_1(x'), & a<x'< x< b, \\ 
F_2(x)G_2(x'), & a<x<x'<b. \end{cases}  \lb{A.3}
\end{equation}

The operator $\bsK$ is bounded,  
\begin{equation}
\bsK \in \cB\big(L^2((a,b);\cH)\big),  
\end{equation}
since, using \eqref{A.10a} and \eqref{A.3}, one readily verifies
\begin{align}
\begin{split} 
& \int_a^bdx\int_a^bdx' \, \|K(x,x')\|_{\cB(\cH)}^2 
= \int_a^b dx \bigg(\int_a^x + \int_x^b \bigg)dx' \, \|K(x,x')\|_{\cB(\cH)}^2   \\
&\quad \leq \sum_{j=1}^2 \int_a^b dx \, \|F_j(x)\|_{\cB(\cH_j,\cH)}^2 
\int_a^b dx' \, \|G_j(x')\|_{\cB(\cH,\cH_j)}^2 < \infty. 
\end{split}
\end{align}

Associated with $\bsK$ we also introduce the bounded Volterra operators $\bsH_a$ and
$\bsH_b$ in $L^2((a,b);\cH)$ defined by 
\begin{align}
(\bsH_af)(x)&=\int_a^x dx'\, H(x,x')f(x'), \lb{A.4} \\
(\bsH_bf)(x)&=-\int_x^b dx'\, H(x,x')f(x'); \quad f\in
L^2((a,b);\cH), \lb{A.12a} 
\end{align}
with $\cB(\cH)$-valued (triangular) integral kernel
\begin{equation}
H(x,x')=F_1(x)G_1(x')-F_2(x)G_2(x').  \lb{A.6}
\end{equation} 
Moreover, introducing the bounded operator block matrices\footnote{$M^\top$  
denotes the transpose of the operator matrix $M$.}  
\begin{align}
C(x)&=(F_1(x) \;\; F_2(x)), \lb{A.7} \\
B(x)&=(G_1(x) \;\; -G_2(x))^\top, \lb{A.8} 
\end{align}
one verifies
\begin{equation}
H(x,x')=C(x)B(x'), \, \text{ where } \begin{cases} 
a<x'<x<b & \text{for $\bsH_a$,} \\ 
a<x<x'<b & \text {for $\bsH_b$} \end{cases} \lb{A.9}
\end{equation}
and 
\begin{equation}
K(x,x')=\begin{cases} C(x)(I_{\cH_1 \oplus \cH_2}-P_0)B(x'), & a<x'<x<b, \\
-C(x)P_0B(x'), & a<x<x'<b,  \end{cases} \lb{A.9a}
\end{equation}
with 
\begin{equation}
 P_0=\begin{pmatrix} 0 & 0 \\ 0 & I_{\cH_2} \end{pmatrix}. \lb{A.9b} 
\end{equation}

The next result proves that, as expected, $\bsH_a$ and $\bsH_b$ are quasi-nilpotent (i.e., have 
vanishing spectral radius) in $L^2((a,b);\cH)$:

%%%%%%%%%%%
\begin{lemma} \lb{lA.2a}
Assume Hypothesis \ref{hA.2}. Then $\bsH_a$ and $\bsH_b$ are quasi-nilpotent in 
$L^2((a,b);\cH)$, equivalently,
\begin{equation}
\sigma (\bsH_a) = \sigma (\bsH_b) = \{0\}.     \lb{A.20a} 
\end{equation}
\end{lemma}
%%%%%%%%%%%
\begin{proof}
It suffices to discuss $\bsH_a$. Then estimating the norm of $H_a^n (x,x')$, $n \in \bbN$, (i.e., the integral kernel for $\bsH_a^n$) in a straightforward 
manner (cf.\ \eqref{A.4}, \eqref{A.6}) yields for a.e.\ $x,x' \in (a,b)$, 
\begin{align}
& \big\|H_a^n (x,x')\big\|_{\cB(\cH)} \leq 2^n 
\max_{j = 1,2}\big(\|F_j(x)\|_{\cB(\cH_j,\cH)}\big) \max_{k = 1,2}\big(\|G_k(x')\|_{\cB(\cH,\cH_k)}\big) 
\no \\
& \quad \times \f{1}{(n-1)!} \bigg[\int_a^x dx'' \, 
\max_{1 \leq \ell, m \leq 2}\big(\|G_{\ell}(x'')\|_{\cB(\cH,\cH_{\ell})} 
\|F_m (x'')\|_{\cB(\cH_m,\cH)}\big)\bigg]^{(n-1)},     \no \\
& \hspace*{10cm}  n \in \bbN.    
\end{align}
Thus, applying \eqref{A.10a}, one verifies
\begin{align}
& \big\|\bsH_a^n\big\|_{\cB(L^2((a,b);\cH))} \leq 
\bigg(\int_a^b dx \int_a^b dx' \, \|H_a^n(x,x')\|_{\cB(\cH)}^2\bigg)^{1/2}     \no \\
& \quad \leq \max_{j=1,2} \bigg(\int_a^b dx \, \|F_j(x)\|_{\cB(\cH_j,\cH)}\bigg)^{1/2} 
\max_{k=1,2} \bigg(\int_a^b dx' \, \|G_k(x')\|_{\cB(\cH,\cH_k)}\bigg)^{1/2}    \no \\
& \qquad \times \f{2^n}{(n-1)!} \max_{1 \leq \ell,m \leq 2} \bigg(\int_a^b dx'' \, 
\|G_{\ell}(x'')\|_{\cB(\cH,\cH_{\ell})} \|F_m(x'')\|_{\cB(\cH_m,\cH)}\bigg)^{(n-1)},     \no \\
& \hspace*{10cm} n \in\bbN,    
\end{align}
and hence
\begin{equation}
\spr(\bsH_a) = \lim_{n\to\infty} \big\|\bsH_a^n\big\|_{\cB(L^2((a,b);\cH))}^{1/n} = 0 
\end{equation}
(here $\spr(\cdot)$ abbreviates the spectral radius). 
Thus, $\bsH_a$ and $\bsH_b$ are quasi-nilpotent in $L^2((a,b);\cH)$ which in turn is equivalent to 
\eqref{A.20a}.
\end{proof}
%%%%%%%%%%%

Next, introducing the linear maps
\begin{align}
&Q\colon \cH_2\mapsto L^2((a,b);\cH), \quad (Q w)(x)=F_2(x) w,
\quad w \in\cH_2, \lb{A.10} \\
&R\colon L^2((a,b);\cH) \mapsto \cH_2, \quad (Rf)=\int_a^b
dx'\,G_2(x')f(x'), \quad f\in L^2((a,b);\cH), \lb{A.11} \\
&S\colon \cH_1 \mapsto L^2((a,b);\cH), \quad (S v)(x)=F_1(x) v,
\quad v\in\cH_1, \lb{A.12} \\
&T\colon L^2((a,b);\cH) \mapsto \cH_1, \quad (Tf)=\int_a^b 
dx'\, G_1(x')f(x'), \quad f\in L^2((a,b);\cH), \lb{A.13} 
\end{align}
one easily verifies the following elementary result (cf.\ 
\cite[Sect.\ IX.2]{GGK90}, \cite[Sect.\ XIII.6]{GGK00} in the case $\dim(\cH)<\infty$):

%%%%%%%%%%%
\begin{lemma} \lb{lA.3} Assume Hypothesis \ref{hA.2}. Then
\begin{align}
\bsK &=\bsH_a + QR \lb{A.14} \\
 &=\bsH_b + ST. \lb{A.15}
\end{align}
\end{lemma}
%%%%%%%%%%%

To describe the inverse of $\bsI-\alpha \bsH_a$ and $\bsI-\alpha \bsH_b$,
$\alpha\in\bbC$, one introduces the block operator matrix $A(\cdot)$ in $\cH_1 \oplus \cH_2$
\begin{align}
A(x)&=\begin{pmatrix} G_1(x)F_1(x) & G_1(x)F_2(x) \\ 
-G_2(x)F_1(x) & -G_2(x)F_2(x) \end{pmatrix} \lb{A.16} \\[1mm]
& = B(x)C(x)\, \text{ for a.e.\ $x\in (a,b)$} \lb{A.17}
\end{align}
and considers the linear evolution equation in $\cH_1 \oplus \cH_2$, 
\begin{equation}
\begin{cases} u'(x) = \alpha A(x) u(x), \quad \alpha \in \bbC, \, \text{ for a.e.\ $x \in (a,b)$,}  \\
u(x_0) = u_0 \in \cH_1 \oplus \cH_2  \end{cases}    \lb{A.32a} 
\end{equation}
for some $x_0 \in (a,b)$. Since $A(x) \in \cB(\cH_1 \oplus \cH_2)$ for a.e. $x\in (a,b)$, 
$A(\cdot)$ is uniformly measurable, and $\|A(\cdot)\|_{\cB(\cH_1 \oplus \cH_2)} \in L^1((a,b))$, 
Theorems\ 1.1 and 1.4 in \cite{OY11} (see also \cite{HO00}, which includes a discussion of a 
nonlinear extension of \eqref{A.32a}) apply and yield the existence of a unique propagator 
$U(\,\cdot \, ,\,\cdot \, ; \alpha)$ on $(a,b) \times (a,b)$ satisfying the following conditions: 
\begin{align}
& U(\cdot \, ,\cdot \, ; \alpha):(a,b) \times (a,b) \to \cB(\cH_1 \oplus \cH_2) \, \text{ is uniformly (i.e., norm) continuous.}  
\lb{A.33A} \\
& \text{There exists $C_{\alpha} > 0$ such that for all $x, x' \in (a,b)$, } \, 
\|U(x,x'; \alpha)\|_{\cB(\cH)} \leq C_{\alpha}.    \lb{A.33B} \\
& \text{For all $x, x', x'' \in (a,b)$, } \, U(x,x'; \alpha) U(x',x''; \alpha) = U(x,x''; \alpha),   \lb{A.33C} \\ 
& \hspace*{6.21cm} U(x,x; \alpha) = I_{\cH_1 \oplus \cH_2}.    \no \\
&\text{For all $u \in \cH_1 \oplus \cH_2$, $\alpha \in \bbC$, }\no\\
&\quad U(x,\cdot\, ;\alpha) u, U(\cdot \, ,x;\alpha) u \in W^{1,1} ((a,b); \cH_1 \oplus \cH_2),\quad x\in (a,b), 
\lb{A.33D} \\ 
& \text{and}    \no \\
& \text{for a.e.\ $x \in (a,b)$, } \, (\partial/\partial x) U(x,x'; \alpha) u = \alpha A(x) U(x,x'; \alpha) u, 
\quad x' \in (a,b),   \lb{A.33E} \\
& \text{for a.e.\ $x' \in (a,b)$, } \, (\partial/\partial x') U(x,x'; \alpha) u = - \alpha U(x,x'; \alpha) A(x') u, 
\quad x \in (a,b).   \lb{A.33F}
\end{align} 

Hence, $u(\, \cdot \, ; \alpha)$ defined by
\begin{equation}
u(x; \alpha) = U(x,x_0; \alpha) u_0, \quad x \in (a,b),      \lb{A.34a} 
\end{equation}
is the unique solution of \eqref{A.32a}, satisfying
\begin{equation}
u(\,\cdot\,; \alpha) \in W^{1,1} ((a,b); \cH_1 \oplus \cH_2).     \lb{A.35a} 
\end{equation}

In fact, an explicit construction (including the proof of uniqueness and that of the properties of 
\eqref{A.33A}--\eqref{A.33F}) of $U(\cdot\, , \cdot \, ; \alpha)$ can simply be obtained by a 
norm-convergent iteration of 
\begin{equation}
U(x,x'; \alpha) = I_{\cH_1 \oplus \cH_2} + \alpha \int_{x'}^x dx'' \, A(x'') U(x'', x'; \alpha), 
\quad x, x' \in (a,b).      \lb{A.36a} 
\end{equation}
Moreover, because of the integrability assumptions made in Hypothesis \ref{hA.2}, 
\eqref{A.32a}-\eqref{A.36a} extend to $x, x' \in [a,b)$ (resp., $x, x' \in (a,b]$) if $a > - \infty$ 
(resp., $b < \infty$) and permit taking norm limits of $U(x,x'; \alpha)$ as $x, x'$ to $-\infty$ if 
$a=-\infty$ (resp., $+\infty$ if $b= +\infty$), see also Remark \ref{rA.5}. 

The next result appeared in \cite[Sect.\ IX.2]{GGK90}, \cite[Sects.\ XIII.5, XIII.6]{GGK00} in the 
special case $\dim(\cH)<\infty$. Here we extend the results to the case where $\cH$ is 
infinite-dimensional. While this extension is straightforward, it appears to be new.  

%%%%%%%%%%%
\begin{theorem} \lb{tA.4} 
Assume Hypothesis \ref{hA.2}. Then, \\
$(i)$ $\bsI-\alpha \bsH_a$ and $\bsI-\alpha \bsH_b$ are boundedly invertible for all
$\alpha\in\bbC$ and 
\begin{align}
(\bsI-\alpha \bsH_a)^{-1}&= \bsI+\alpha \bsJ_a(\alpha), \lb{A.19} \\
(\bsI-\alpha \bsH_b)^{-1}&= \bsI+\alpha \bsJ_b(\alpha), \lb{A.20} \\
(\bsJ_a(\alpha) f)(x)&=\int_a^x dx'\, J(x,x'; \alpha )f(x'), \lb{A.21} \\ 
(\bsJ_b(\alpha) f)(x)&=-\int_x^b dx'\, J(x,x'; \alpha )f(x'); \quad f \in
L^2((a,b);\cH), \lb{A.22} \\  
J(x,x'; \alpha )&=C(x) U(x, x'; \alpha) B(x'), \, \text{ where }
\begin{cases}  a<x'<x<b & \text{for $\bsJ_a(\alpha)$,} \\ 
a<x<x'<b & \text {for $\bsJ_b(\alpha)$.} \end{cases}  \lb{A.23}  
\end{align}
$(ii)$ Let $\alpha\in\bbC$. Then $\bsI-\alpha \bsK$ is boundedly invertible if and only if 
$I_{\cH_2}-\alpha R(\bsI-\alpha \bsH_a)^{-1}Q$ is. Similarly,
$\bsI-\alpha \bsK$ is boundedly invertible if and only if 
$I_{\cH_1}-\alpha T(\bsI-\alpha \bsH_b)^{-1}S$ is. In particular,
\begin{align}
& (\bsI-\alpha \bsK)^{-1}=(\bsI-\alpha \bsH_a)^{-1}
+\alpha (\bsI-\alpha \bsH_a)^{-1}QR(\bsI-\alpha \bsK)^{-1} \lb{A.24} \\
& \quad = (\bsI-\alpha \bsH_a)^{-1} \no \\
& \qquad +\alpha (\bsI-\alpha \bsH_a)^{-1}Q\big[I_{\cH_2}
-\alpha R(\bsI-\alpha \bsH_a)^{-1}Q\big]^{-1}R(\bsI-\alpha \bsH_a)^{-1}
\lb{A.25} \\
& \quad =(\bsI-\alpha \bsH_b)^{-1}+\alpha (\bsI-\alpha \bsH_b)^{-1}
ST(\bsI-\alpha \bsK)^{-1} \lb{A.26} \\
& \quad =(\bsI-\alpha \bsH_b)^{-1} \no \\
& \qquad +\alpha (\bsI-\alpha \bsH_b)^{-1}S\big[I_{\cH_1}-\alpha T(\bsI-\alpha \bsH_b)^{-1}S\big]^{-1}
T(\bsI-\alpha \bsH_b)^{-1}.
\lb{A.27}
\end{align}
\end{theorem}
%%%%%%%%%%%
\begin{proof} 
To prove the results \eqref{A.19}--\eqref{A.23} it suffices to focus on $\bsH_a$. Let 
$f \in L^2((a,b); \cH)$. Then using $H(x,x') = C(x) B(x')$ and 
$A(x) = B(x) C(x)$  (cf.\ \eqref{A.9} 
and \eqref{A.17}) one computes (for some $x_0 \in (a,b)$) with the help of \eqref{A.33E}, 
\begin{align}
& \big((\bsI - \alpha \bsH_a)(\bsI + \alpha \bsJ_a(\alpha)) f\big)(x) 
= f(x) - \alpha \int_a^x dx' \, C(x) B(x') f(x') \no \\
& \qquad + \alpha \int_a^x dx' \, C(x) U(x,x'; \alpha) B(x') f(x')    \no \\ 
& \qquad - \alpha^2 \int_a^x dx' \, C(x) B(x') \int_a^{x'} dx'' \, C(x') U(x',x''; \alpha) B(x'') f(x'')  \no \\
& \quad = f(x) - \alpha \int_a^x dx' \, C(x) B(x') f(x') 
+ \alpha \int_a^x dx' \, C(x) U(x,x'; \alpha) B(x') f(x')    \no \\ 
& \qquad - \alpha^2 \int_a^x dx' \, C(x)B(x')C(x')U(x',x_0;\alpha) 
\int_a^{x'} dx'' \, U(x_0,x''; \alpha) B(x'') f(x'')   \no \\
& \quad = f(x) - \alpha \int_a^x dx' \, C(x) B(x') f(x') 
+ \alpha \int_a^x dx' \, C(x) U(x,x'; \alpha) B(x') f(x')    \no \\ 
& \qquad - \alpha \int_a^x dx' \, C(x) [(\partial/\partial x') U(x',x_0; \alpha)]
\int_a^{x'} dx'' \, U(x_0,x''; \alpha) B(x'') f(x'')   \no \\
& \quad = f(x) - \alpha \int_a^x dx' \, C(x) B(x') f(x') 
+ \alpha \int_a^x dx' \, C(x) U(x,x'; \alpha) B(x') f(x')    \no \\ 
& \qquad - \alpha C(x) \bigg[U(x',x_0; \alpha) \int_a^{x'} dx'' \, U(x_0,x''; \alpha) B(x'') f(x'') \bigg|_{x'=a}^x 
\no \\[1mm] 
& \hspace*{2.3cm} - \int_a^x dx' \, U(x',x_0; \alpha) U(x_0,x'; \alpha) B(x') f(x')\bigg]    \no \\
& \quad = f(x) \, \text{ for a.e.\ $x \in (a,b)$.} 
\end{align}
In the same manner one proves 
\begin{equation}
\big((\bsI + \alpha \bsJ_a(\alpha))(\bsI - \alpha \bsH_a) f\big)(x) = f(x)  \, \text{ for a.e.\ $x \in (a,b)$.} 
\end{equation} 

By \eqref{A.14} and \eqref{A.15}, $\bsK - \bsH_a$ and $\bsK - \bsH_b$ factor into $QR$ and $ST$, respectively.  Consequently, \eqref{A.24} and \eqref{A.26} follow from the second resolvent identity, while \eqref{A.25} and \eqref{A.27} are direct applications of Kato's resolvent equation for factored perturbations (cf. \cite[Sect.~2]{GLMZ05}).
%
%Relations \eqref{A.24}--\eqref{A.27} are clear from
%\eqref{A.14} and \eqref{A.15}, a standard resolvent identity, and the
%fact that $\bsK - \bsH_a$ and $\bsK - \bsH_b$ factor into $QR$ and $ST$, respectively. 
\end{proof}
%%%%%%%%%%%

%%%%%%%%%%%
\begin{remark} \lb{rA.5}  
Even though this will not be used in this paper, we mention for completeness that if 
$(\bsI - \alpha \bsK)^{-1} \in \cB\big(L^2((a,b);\cH)\big)$, and if $U(\cdot\, ,a ; \alpha)$ 
is defined by 
\begin{equation}
U(x,a; \alpha) = I_{\cH_1 \oplus \cH_2} + \alpha \int_a^x dx' \, A(x') U(x',a; \alpha), \quad x \in (a,b), 
\lb{A.48a}
\end{equation}
and partitioned with respect to $\cH_1 \oplus \cH_2$ as
\begin{equation}
U(x,a; \alpha) = \begin{pmatrix} U_{1,1}(x,a; \alpha) & U_{1,2}(x,a;\alpha) \\
U_{2,1}(x,a; \alpha) & U_{2,2}(x,a;\alpha) \end{pmatrix}, \quad x \in (a,b), 
\end{equation}
then 
\begin{align}
(\bsI - \alpha \bsK)^{-1} &= \bsI+\alpha \bsL(\alpha), \lb{A.28} \\
(\bsL(\alpha) f)(x) &= \int_a^b dx'\, L(x,x'; \alpha )f(x'), \lb{A.29} \\  
L(x,x'; \alpha )  &= \begin{cases} C(x)U(x,a; \alpha) (I-P(\alpha))
U(x',a; \alpha)^{-1}B(x'), & a<x'<x<b, \\  
-C(x)U(x,a; \alpha) P(\alpha) U(x',a; \alpha)^{-1}B(x'), & a<x<x'<b, \end{cases} 
\lb{A.30}  
\end{align}
where 
\begin{equation}
P(\alpha)=\begin{pmatrix} 0 & 0 \\ U_{2,2}(b,a; \alpha )^{-1}
U_{2,1}(b,a; \alpha ) & I_{\cH_2} \end{pmatrix},     \lb{A.33}
\end{equation}
with $U(b,a; \alpha) = \nlim_{x \uparrow b} U(x,a; \alpha)$. 
These results can be shown as in the finite-dimensional case treated in \cite[Ch.\ IX]{GGK90}. 
$\Diamond$
\end{remark}
%%%%%%%%%%%

%%%%%%%%%%%
\begin{lemma} \lb{lA.6}
Assume Hypothesis \ref{hA.2} and introduce, for $\alpha\in\bbC$ and a.e.\ $x \in (a,b)$, the Volterra 
integral equations
\begin{align}
\widehat F_1(x; \alpha )&=F_1(x)-\alpha \int_x^b dx'\, H(x,x')\widehat F_1(x'; \alpha ), \lb{A.35} \\ 
\widehat F_2(x; \alpha )&=F_2(x)+\alpha \int_a^x dx'\, H(x,x')\widehat F_2(x'; \alpha ). \lb{A.36}  
\end{align}  
Then there exist unique a.e.\ solutions on $(a,b)$, $\widehat F_j (\cdot\, ; \alpha) \in \cB(\cH_j,\cH)$,  
of \eqref{A.35}, \eqref{A.36} such that $\widehat F_j(\cdot\, ; \alpha)$ are uniformly measurable, and 
\begin{equation}  
\big\|\widehat F_j( \cdot \, ; \alpha)\big\|_{\cB(\cH_j,\cH)} \in L^2((a,b)), \quad j=1,2. \lb{A.45}
\end{equation}
\end{lemma}
%%%%%%%%%%%
\begin{proof}
Introducing, 
\begin{align}
\widehat F_{1,0} (x; \alpha) &= F_1(x),    \no \\
\widehat F_{1,n} (x; \alpha) &= - \alpha \int_x^b dx' \, H(x,x') \widehat F_{1,n-1} (x'; \alpha), \quad n \in\bbN,    \\
\widehat F_{2,0} (x; \alpha) &= F_2(x),    \no \\
\widehat F_{2,n} (x; \alpha) &= \alpha \int_a^x dx' \, H(x,x') \widehat F_{2,n-1} (x'; \alpha), \quad n \in\bbN,
\end{align}
for a.e.\ $x \in (a,b)$, the familiar iteration procedure (in the scalar or matrix-valued context) yields 
for fixed $x \in (a,b)$ except for a set of Lebesgue measure zero, 
\begin{align}
& \big\|\widehat F_{1,n} (x; \alpha)\big\|_{\cB(\cH_1,\cH)} \leq (2 |\alpha|)^n 
\max_{j=1,2} \big(\|F_j(x)\|_{\cB(\cH_j,\cH)}\big)     \\ 
& \quad \times \f{1}{n!} 
\bigg[\int_x^b dx' \, \max_{1 \leq k, \ell \leq 2} \big(\|G_k (x')\|_{\cB(\cH,\cH_k)} 
\|F_{\ell} (x')\|_{\cB(\cH_{\ell},\cH)}\big)\bigg]^n, \quad n \in \bbN,     \no \\
& \big\|\widehat F_{2,n} (x; \alpha)\big\|_{\cB(\cH_2,\cH)} \leq (2 |\alpha|)^n 
\max_{j=1,2} \big(\|F_j (x)\|_{\cB(\cH_j,\cH)}\big)      \\ 
& \quad \times \f{1}{n!} 
\bigg[\int_a^x dx' \, \max_{1 \leq k, \ell \leq 2} \big(\|G_k (x')\|_{\cB(\cH,\cH_k)} 
\|F_{\ell} (x')\|_{\cB(\cH_{\ell},\cH)}\big)\bigg]^n, \quad n \in \bbN.    \no 
\end{align}
Thus, the norm convergent expansions 
\begin{align}
\widehat F_j (x; \alpha) = \sum_{n=0}^{\infty} \widehat F_{j,n}(x; \alpha), \quad j=1,2, \, 
\text{ for a.e.\ $x \in (a,b)$,}  
\lb{A.50}
\end{align}
yield the bounds
\begin{align}
\big\|\widehat F_j (x; \alpha)\big\|_{\cB(\cH_j,\cH)} &\leq \max_{k=1,2} \big(\|F_k(x)\|_{\cB(\cH_k,\cH)}\big)   \\
& \;\;\, \times \max_{1 \leq \ell, m \leq 2}
\exp\bigg(2 |\alpha| \int_a^b dx' \, \|G_{\ell} (x')\|_{\cB(\cH,\cH_{\ell})} 
\|F_m (x')\|_{\cB(\cH_m,\cH)}\bigg)      \no 
\end{align}
for a.e.\ $x \in (a,b)$. As in the scalar case (resp., as in the proof of Theorem \ref{tA.4})  
one shows that \eqref{A.50} uniquely satisfies \eqref{A.35}, \eqref{A.36} 
\end{proof}
%%%%%%%%%%%

%%%%%%%%%%%
\begin{lemma} \lb{lA.7}
Assume Hypothesis \ref{hA.2}, let $\alpha\in\bbC$, and introduce 
\begin{align}
& U(x; \alpha)=\begin{pmatrix} I_{\cH_1}-\alpha \int_x^b dx'\, G_1(x') 
\widehat F_1(x'; \alpha ) & \alpha\int_a^x dx'\, G_1(x')\widehat F_2(x'; \alpha ) \\ 
\alpha\int_x^b dx'\, G_2(x')\widehat F_1(x'; \alpha ) & I_{\cH_2}-\alpha \int_a^x
dx'\, G_2(x')\widehat F_2(x'; \alpha ) \end{pmatrix},  \no \\
& \hspace*{9.4cm} x\in (a,b). \lb{A.37}  
\end{align}
If 
\begin{align}
\bigg[I_{\cH_1}-\alpha \int_a^b dx\, G_1(x) \widehat F_1(x; \alpha )\bigg]^{-1} 
\in \cB(\cH_1),    \lb{A.38} \\
\intertext{or equivalently,} 
\bigg[I_{\cH_2}-\alpha \int_a^b dx\, G_2(x)\widehat F_2(x; \alpha )\bigg]^{-1} \in \cB(\cH_2), \lb{A.39} 
\end{align}
then 
\begin{equation}
U(a; \alpha) , \, U(b; \alpha), \, U(x; \alpha), \; x\in (a,b), 
\end{equation} 
are boundedly invertible in $\cH_1 \oplus \cH_2$. In particular,  
\begin{equation}
U(x,x'; \alpha) = U(x; \alpha) U(x'; \alpha)^{-1}, \quad x, x' \in (a,b),   \lb{A.72a}
\end{equation}   
is the propagator for the evolution equation \eqref{A.32a} satisfying \eqref{A.33A}--\eqref{A.36a}, and 
\eqref{A.72a} extends by norm continuity to $x,x' \in \{a,b\}$. 
\end{lemma}
%%%%%%%%%%%
\begin{proof}
Since 
\begin{equation}
U(a; \alpha)=\begin{pmatrix} I_{\cH_1}-\alpha \int_a^b dx'\, G_1(x') 
\widehat F_1(x'; \alpha ) & 0 \\ 
\alpha\int_a^b dx'\, G_2(x')\widehat F_1(x'; \alpha ) & I_{\cH_2} \end{pmatrix},     \lb{A.68a}  
\end{equation}
$U(a; \alpha)$ is boundedly invertible in $\cH_1 \oplus \cH_2$ if and only if 
$\Big[I_{\cH_1}-\alpha \int_a^b dx'\, G_1(x') \widehat F_1(x'; \alpha )\Big]$ is in $\cH_1$. (One 
recalls that a bounded $2 \times 2$ block operator 
$D = \left(\begin{smallmatrix} D_{1,1} & 0 \\ D_{2,1} & I_{\cH_2} \end{smallmatrix}\right)$ 
in $\cH_1 \oplus \cH_2$ is boundedly invertible if and only if $D_{1,1}$ is boundedly invertible 
in $\cH_1$, with 
$D^{-1} = \left(\begin{smallmatrix} D_{1,1}^{-1} & 0 \\ - D_{2,1} D_{1,1}^{-1}  & I_{\cH_2} 
\end{smallmatrix}\right)$ if $D$ is boundedly invertible.) Similarly, 
\begin{equation}
U(b; \alpha)=\begin{pmatrix} I_{\cH_1} & \alpha\int_a^b dx'\, G_1(x')\widehat F_2(x'; \alpha ) \\ 
0 & I_{\cH_2}-\alpha \int_a^b dx'\, G_2(x')\widehat F_2(x'; \alpha ) \end{pmatrix}     \lb{A.69a}  
\end{equation}
is boundedly invertible in $\cH_1 \oplus \cH_2$ if and only if 
$\Big[I_{\cH_2}-\alpha \int_a^b dx'\, G_2(x') \widehat F_2(x'; \alpha )\Big]$ is 
boundedly invertible in $\cH_2$. (Again, one 
recalls that a bounded $2 \times 2$ block operator 
$E = \left(\begin{smallmatrix} I_{\cH_1} & E_{1,2} \\ 0 & E_{2,2} \end{smallmatrix}\right)$ 
in $\cH_1 \oplus \cH_2$ is boundedly invertible if and only if $E_{2,2}$ is boundedly invertible 
in $\cH_2$, with 
$E^{-1} = \left(\begin{smallmatrix} I_{\cH_1} & - E_{1,2} E_{2,2}^{-1} \\ 0 & E_{2,2}^{-1} 
\end{smallmatrix}\right)$ if $E$ is boundedly invertible.)

The equivalence of \eqref{A.38} and \eqref{A.39} has been settled in Theorem\ \ref{tA.4}\,$(ii)$. 

Next, differentiating the entries on the right-hand side of \eqref{A.37} with respect to $x$ and 
using the Volterra integral equations \eqref{A.35}, \eqref{A.36} yields 
\begin{equation}
(d/dx) U(x; \alpha) u = \alpha A(x) U(x; \alpha) u \, \text{ for a.e.\ $x \in (a,b)$}.  
\end{equation}
Thus, by uniqueness of the propagator $U(\cdot \, , \cdot \, ; \alpha)$, extended by norm 
continuity to $x=a$ (cf.\ Remark \ref{rA.5}), one obtains that 
\begin{equation}
U(x,a; \alpha) = U(x; \alpha) U(a; \alpha)^{-1}, \quad x \in (a,b). 
\end{equation}
Thus, $U(x; \alpha) = U(x,a; \alpha) U(a; \alpha)$ is boundedly invertible for all $x \in (a,b)$ since 
$U(x,a; \alpha)$, $x \in (a,b)$ is by construction (using norm continuity and the transitivity property 
in \eqref{A.33C}), and $U(a; \alpha)$ is boundedly invertible by hypothesis.  
Consequently, once more by uniqueness of the propagator $U(\cdot \, , \cdot \, ; \alpha)$, one 
obtains that  
\begin{equation}
U(x,x'; \alpha) = U(x; \alpha) U(x'; \alpha)^{-1}, \quad x, x' \in (a,b).   \lb{A.77a}
\end{equation}
Again by norm continuity, \eqref{A.77a} extends to $x, x' \in \{a, b\}$. 
\end{proof}
%%%%%%%%%%%

In the special case where $\cH$ and $\cH_j$, $j=1,2$, are finite-dimensional, the Volterra integral 
equations \eqref{A.35}, \eqref{A.36} and the operator $U$ in \eqref{A.37} were introduced 
in \cite{GM03}. 

%%%%%%%%%%%
\begin{lemma} \lb{lA.9} 
Let $\cH$ and $\cH'$ be complex, separable Hilbert spaces and $-\infty\leq a<b\leq \infty$. 
Suppose that for a.e.\ $x \in (a,b)$, $F (x) \in \cB_2(\cH',\cH)$ and $G(x) \in \cB_2(\cH,\cH')$ 
with $F(\cdot)$ and $G(\cdot)$ weakly measurable, and  
\begin{equation}  
\|F( \cdot)\|_{\cB_2(\cH',\cH)} \in L^2((a,b)), \; 
\|G (\cdot)\|_{\cB_2(\cH,\cH')} \in L^2((a,b)). \lb{A.79a}
\end{equation}
Consider the integral operator $\bsS$ in $L^2((a,b);\cH)$ with $\cB_1(\cH)$-valued 
separable integral kernel of the type 
\begin{equation}
S(x,x') = F(x) G(x') \, \text{ for a.e.\ $x, x' \in (a,b)$.}      \lb{A.80a} 
\end{equation}
Then 
\begin{equation}
\bsS \in \cB_1\big(L^2((a,b);\cH)\big).   \lb{A.81a}
\end{equation}
\end{lemma}
%%%%%%%%%%%
\begin{proof}
Since the Hilbert space of Hilbert--Schmidt operators, $\cB_2(\cH',\cH)$, is separable, weak 
measurability of $F(\cdot)$ implies $\cB_2(\cH',\cH)$-measurability 
by Pettis' theorem (cf., e.g., \cite[Theorem\ 1.1.1]{ABHN01}, \cite[Theorem\ II.1.2]{DU77}, 
\cite[3.5.3]{HP85}), and analogously for $G(\cdot)$. 

Next, one introduces (in analogy to \eqref{A.10}--\eqref{A.13}) the linear operators
\begin{align}
& Q_F \colon \cH' \mapsto L^2((a,b);\cH), \quad (Q_F w)(x)=F(x) w,
\quad w \in\cH', \lb{A.81A} \\
& R_G \colon L^2((a,b);\cH) \mapsto \cH', \quad (R_G f)=\int_a^b
dx'\,G(x')f(x'), \quad f\in L^2((a,b);\cH), \lb{A.81B}
\end{align} 
such that
\begin{equation}
\bsS = Q_F R_G.    \lb{A.81C} 
\end{equation}
Thus, with $\{v_n\}_{n\in\bbN}$ a complete orthonormal system in $\cH'$, using the monotone 
convergence theorem, one concludes that 
\begin{align}
& \|Q_F\|_{\cB_2(\cH', L^2((a,b);\cH))}^2 = \sum_{n\in\bbN} \|Q_F v_n\|_{L^2((a,b);\cH)}^2  \no \\
& \quad = \sum_{n\in\bbN} \int_a^b dx \, \|F(x) v_n\|_{\cH}^2 
= \int_a^b dx \, \sum_{n\in\bbN} (v_n, F(x)^* F(x) v_n)_{\cH'}    \no \\
& \quad = \int_a^b dx \, {\tr}_{\cH'} \big(F(x)^* F(x)\big) = \int_a^b dx \, \big\|F(x)^* F(x)\big\|_{\cB_1(\cH')} 
\no \\ 
& \quad = \int_a^b dx \, \|F(x)\|_{\cB_2(\cH', \cH)}^2 < \infty.     \lb{A.83a} 
\end{align}
The same argument applied to $R_G^*$ (which is of the form $Q_{G^*}$, i.e., given by 
\eqref{A.81A} with $F(\cdot)$ replaced by $G(\cdot)^*$) then proves 
$R_G^* \in \cB_2(\cH', L^2((a,b);\cH))$. Hence,  
\begin{equation}
Q_F \in \cB_2(\cH', L^2((a,b);\cH)), \quad R_G \in \cB_2(L^2((a,b);\cH), \cH'),  
\end{equation}
together with the factorization \eqref{A.81C}, prove \eqref{A.81a}.
\end{proof}
%%%%%%%%%%%

Next, we strengthen our assumptions as follows: 

%%%%%%%%%%%
\begin{hypothesis} \lb{hA.10}  
Let $\cH$ and $\cH_j$, $j=1,2$, be complex, separable Hilbert spaces and $-\infty\leq a<b\leq \infty$. 
Suppose that for a.e.\ $x \in (a,b)$, $F_j (x) \in \cB_2(\cH_j,\cH)$ and $G_j(x) \in \cB_2(\cH,\cH_j)$ 
such that $F_j(\cdot)$ and $G_j(\cdot)$ are weakly measurable, and 
\begin{equation}  
\|F_j( \cdot)\|_{\cB_2(\cH_j,\cH)} \in L^2((a,b)), \; 
\|G_j (\cdot)\|_{\cB_2(\cH,\cH_j)} \in L^2((a,b)), \quad j=1,2.      \lb{A.78a}
\end{equation}
\end{hypothesis}
%%%%%%%%%%%

As an immediate consequence of Hypothesis \ref{hA.10} one infers the following facts.

%%%%%%%%%%%
\begin{lemma} \lb{lA.11}
Assume Hypothesis \ref{hA.10} and $\alpha \in \bbC$. Then, for a.e.\ $x \in (a,b)$, 
$\widehat F_j (x; \alpha) \in \cB_2(\cH_j,\cH)$, $\widehat F_j(\cdot\, ; \alpha)$ are 
$\cB_2(\cH_j,\cH)$-measurable, and 
\begin{equation}  
\big\|\widehat F_j( \cdot \, ; \alpha)\big\|_{\cB_2(\cH_j,\cH)} \in L^2((a,b)), \quad j=1,2.     \lb{A.84a}
\end{equation}
Moreover, 
\begin{align}
\begin{split} 
\int_c^d dx \, G_j (x) F_k(x), \, 
\int_c^d dx \, G_j (x) \hatt F_k(x; \alpha) \in \cB_1(\cH_k,\cH_j),&    \\
1 \leq j, k \leq 2, \; 
c, d \in (a,b) \cup \{a, b\},&      \lb{A.85a} 
\end{split} 
\end{align}
and 
\begin{align}
& QR, ST \in \cB_1\big(L^2((a,b);\cH)\big),    \lb{A.86b} \\
& \bsK, \bsH_a, \bsH_b \in \cB_2\big(L^2((a,b);\cH)\big).     \lb{A.86a} 
\end{align} 
\end{lemma}
%%%%%%%%%%%
\begin{proof}
As in the proof of Lemma \ref{lA.9}, one concludes that weak 
measurability of $\widehat F_j(\cdot\, ; \alpha)$, $j=1,2$, implies their $\cB_2(\cH_j,\cH)$-measurability 
by Pettis' theorem. The properties concerning $\widehat F_j( \cdot \, ; \alpha)$, $j=1,2$, then 
follow as in the proof of Lemma \ref{lA.6}, systematically replacing $\| \cdot \|_{\cB(\cH_j,\cH)}$ by 
$\| \cdot \|_{\cB_2(\cH_j,\cH)}$, $j=1,2$. 

Applying Lemma \ref{l4.2}, relations \eqref{A.85a} are now an immediate consequence 
of Hypothesis \ref{hA.10} and the fact that 
\begin{equation}
\big\|G_j (\cdot) \hatt F_k(\cdot \, ; \alpha)\big\|_{\cB_1(\cH_k,\cH_j)} \in L^1((a,b)), \quad 
1 \leq j, k \leq 2.     \lb{A.87a} 
\end{equation}

The proof of Lemma \ref{lA.9} (see \eqref{A.83a}) yields
\begin{align}
\begin{split} 
S \in \cB_2(\cH_1, L^2((a,b);\cH)), \, Q \in \cB_2(\cH_2, L^2((a,b);\cH)), \\  
T \in \cB_2(L^2((a,b);\cH), \cH_1), \, R \in \cB_2(L^2((a,b);\cH), \cH_2),   \lb{A.87A} 
\end{split} 
\end{align}
and \eqref{A.86b} follows.

Next, for any integral operator $\bsT$ in $L^2((a,b);\cH)$, with integral 
kernel satisfying $\|T(\cdot \, , \cdot \,)\|_{\cB_2(\cH)} \in L^2 ((a,b) \times (a,b); d^2x)$, one 
infers (cf.\ \cite[Theorem\ 11.6]{BS87}) that $\bsT \in \cB_2\big(L^2((a,b);\cH)\big)$ and 
\begin{equation}
\|\bsT\|_{\cB_2(L^2((a,b);\cH))} = 
\bigg(\int_a^b dx \int_a^b dx' \, \|T(x,x')\|_{\cB_2(\cH)}^2\bigg)^{1/2}.    \lb{A.82a} 
\end{equation}

Given Lemma \ref{lA.9} and the fact \eqref{A.82a}, one readily concludes \eqref{A.86a}.
\end{proof}
%%%%%%%%%%%

Next, we turn to a discussion of Fredholm determinants in connection with the operators  
$\bsH_a$, $\bsH_b$, and $\bsK$. However, rather than working with modified Fredholm determinants ${\det}_2(\cdot)$ in the following (see, 
however, \cite{GM03} in the special case where $\cH$ and $\cH_j$, $j=1,2$, are assumed to be 
finite-dimensional, and \cite{GN13}), we now make the following additional trace class assumption:

%%%%%%%%%%%
\begin{hypothesis} \lb{hA.12}  
In addition to Hypothesis \ref{hA.10}, suppose that
\begin{equation}
\bsK \in \cB_1\big(L^2((a,b);\cH)\big).     \lb{A.88a} 
\end{equation} 
\end{hypothesis}
%%%%%%%%%%% 

Standard examples satisfying Hypothesis \ref{hA.12} are one-dimensional Schr\"odinger 
and Dirac-type operators (on arbitrary intervals, and with appropriate boundary conditions) with 
suitably bounded operator-valued coefficients. Following \cite{GWZ12} and \cite{GWZ13}, this is 
demonstrated in detail in Section \ref{s3} and Appendix \ref{sA} in the context of 
Schr\"odinger operators on the full real line (cf.\ Hypothesis \ref{hB.1} and Theorem \ref{tB.2}). 
A discrete analog of this semi-separable formalism applies to Jacobi and CMV operators, but this 
is not treated in this paper. More generally, ordinary differential operators of arbitrary order with 
bounded operator-valued coefficients have semi-separable operator-valued Green's functions and 
hence are a natural target for the methods developed in this paper. 

As an immediate consequence of combining Hypothesis \ref{hA.12}, the fact \eqref{A.86b}, 
and the decompositions \eqref{A.14}, \eqref{A.15}, one concludes that  
\begin{equation}
\bsH_a, \bsH_b \in \cB_1\big(L^2((a,b);\cH)\big).     \lb{A.88A} 
\end{equation} 

In the following we use some of the standard properties of determinants and traces, such as,
\begin{align}
& {\det}_{\cK}((I_\cK-A)(I_\cK-B))={\det}_{\cK}(I_\cK-A){\det}_{\cK}(I_\cK-B), \quad A, B
\in\cB_1(\cK), \lb{A.92a} \\   
& {\det}_{\cK}(I_{\cK}-AB)={\det}_{\cK'}(I_{\cK'}-BA), \quad {\tr}_{\cK} (AB) = {\tr}_{\cK'} (BA)   \lb{A.93a} \\
& \, \text{for all $A\in\cB_1(\cK',\cK)$, $B\in\cB(\cK,\cK')$ 
such that $AB\in\cB_1(\cK)$, $BA\in \cB_1(\cK')$,} \no
\intertext{and}
& {\det}_{\cK}(I_\cK-A)={\det}_{\cK_2}(I_{\cK_2}-D) \, \text{ for } \, A=\begin{pmatrix} 
0 & C \\ 0 & D \end{pmatrix}, \;\, D \in \cB_1(\cK_2), \; \cK=\cK_1\dotplus \cK_2, \lb{A.94a} 
\intertext{since}
&I_\cH -A=\begin{pmatrix} I_{\cK_1} & -C \\ 0 & I_{\cK_2}-D \end{pmatrix} =
\begin{pmatrix} I_{\cK_1} & 0 \\ 0 & I_{\cK_2}-D \end{pmatrix} 
\begin{pmatrix} I_{\cK_1} & -C \\ 0 & I_{\cK_2} \end{pmatrix}. \lb{A.95a} 
\end{align}
Here $\cK$, $\cK'$, and $\cK_j$, $j=1,2$, are complex, separable Hilbert spaces,
$\cB(\cK)$ denotes the set of bounded linear operators on $\cK$, $\cB_p(\cK)$,
$p\geq 1$, denote the usual $\ell^p$-based trace ideals of $\cB(\cK)$, and $I_\cK$ denotes
the identity operator in $\cK$. Moreover, ${\det}_{\cK}(I_\cK-A)$, $A\in\cB_1(\cK)$,
denotes the standard Fredholm
determinant, and $\tr_{\cK}(A)$, $A\in\cB_1(\cK)$, the
corresponding trace. Finally, $\dotplus$ in \eqref{A.94a} denotes a
direct, but not necessary orthogonal, sum decomposition of $\cK$ into $\cK_1$
and $\cK_2$. (We refer, e.g., to \cite{GGK96}, \cite{GGK97}, 
\cite[Sect.\ IV.1]{GK69}, \cite[Ch.\ 17]{RS78}, \cite{Si77}, \cite[Ch.\ 3]{Si05} 
for these facts). 
 
The principal result of this section then concerns the following reduction of the Fredholm determinant ${\det}_{L^2((a,b);\cH)}(\bsI-\alpha \bsK)$ to certain 
Fredholm determinants associated with operators in the Hilbert spaces 
$\cH_j$, $j=1,2$, and $\cH_1 \oplus \cH_2$:

%%%%%%%%%%%
\begin{theorem} \lb{tA.13} 
Suppose Hypothesis \ref{hA.12} and let $\alpha\in\bbC$. Then, 
\begin{align} 
&{\tr}_{L^2((a,b);\cH)}(\bsH_a)={\tr}_{L^2((a,b);\cH)}(\bsH_b)=0,     \lb{A.96a} \\ 
&{\det}_{L^2((a,b);\cH)}(\bsI-\alpha \bsH_a)
= {\det}_{L^2((a,b);\cH)}(\bsI-\alpha \bsH_b)=1, 
\lb{A.97a} \\ 
&\tr_{L^2((a,b);\cH)}(\bsK)=\int_a^b dx\,\tr_{\cH_1}(G_1(x)F_1(x))
=\int_a^b dx\,\tr_{\cH}(F_1(x)G_1(x)) \lb{A.98a} \\
& \quad =\int_a^b dx\,\tr_{\cH_2}(G_2(x)F_2(x))
=\int_a^b dx\,\tr_{\cH}(F_2(x)G_2(x)). \lb{A.99a} 
\end{align}
Assume in addition that $U$ is given by \eqref{A.37}. Then,
\begin{align}
&{\det}_{L^2((a,b);\cH)}(\bsI-\alpha \bsK)={\det}_{\cH_1}\big(I_{\cH_1}
-\alpha T(\bsI - \alpha \bsH_b)^{-1}S\big) \lb{A.100a}
\\ 
& \quad ={\det}_{\cH_1}\bigg(I_{\cH_1}-\alpha \int_a^b dx\,
G_1(x)\widehat F_1(x; \alpha )\bigg) \lb{A.101a} \\
& \quad ={\det}_{\cH_1 \oplus \cH_2}(U(a; \alpha)) \lb{A.102a} \\
& \quad ={\det}_{\cH_2}\big(I_{\cH_2}-\alpha R(\bsI - \alpha \bsH_a)^{-1}Q\big)
\lb{A.103a} \\
& \quad ={\det}_{\cH_2}\bigg(I_{\cH_2}-\alpha \int_a^b dx\,
G_2(x)\widehat F_2(x; \alpha )\bigg) \lb{A.104a} \\
& \quad ={\det}_{\cH_1 \oplus \cH_2}(U(b; \alpha)). \lb{A.105a} 
\end{align}
\end{theorem}
%%%%%%%%%%%
\begin{proof}
Relations \eqref{A.96a} and \eqref{A.97a} hold since for any quasi-nilpotent operator $A$ in $\cK$ 
satisfying $A \in \cB_1(\cK)$ one has $\tr_{\cK} (A) = 0$ and $\det_{\cK}(I -A) = 1$ (cf., e.g., 
part $(I)$ of the proof of Theorem\ VII.6.1 in \cite{GGK90}). Thus, \eqref{A.14}, \eqref{A.15}, and 
cyclicity of the trace (i.e., the second equality in \eqref{A.93a}) imply
\begin{align}
\begin{split} 
\tr_{L^2((a,b);\cH)}(\bsK) &= \tr_{L^2((a,b);\cH)}(QR)=\tr_{\cH_2}(RQ)  \\
&= \tr_{L^2((a,b);\cH)}(ST)=\tr_{\cH_1}(TS). \lb{A.106a}
\end{split} 
\end{align}
The equalities throughout \eqref{A.98a} and \eqref{A.99a} follow from computing appropriate traces in \eqref{A.106a}.  For example, taking an orthonormal basis $\{v_n\}_{n\in \bbN}$ in $\cH_2$, one computes
\begin{align}
&\tr_{\cH_2}(RQ)=\sum_{n\in \bbN} (v_n,RQv_n)_{\cH_2} = \sum_{n\in \bbN} \bigg(v_n, \int_a^bdx\, G_2(x)(Qv_n)(x)\bigg)_{\cH_2}\no\\
&\quad= \int_a^bdx\, \sum_{n\in \bbN} (v_n, G_2(x)(Qv_n)(x))_{\cH_2}= \int_a^bdx\, \sum_{n\in \bbN} (v_n,G_2(x)F_2(x)v_n)_{\cH_2}\no\\
&\quad= \int_a^bdx\, \tr_{\cH_2}(G_2(x)F_2(x)),
\end{align}
thereby establishing equality between $\tr_{L^2((a,b);\cH)}(\bsK)$ and the first expression in \eqref{A.99a}.  The remaining traces in \eqref{A.106a} are computed similarly to establish the other equalities in \eqref{A.98a} and \eqref{A.99a}.  Next, one observes
\begin{align}
\bsI-\alpha \bsK &=(\bsI-\alpha \bsH_a)\big[\bsI-\alpha (\bsI-\alpha \bsH_a)^{-1}QR\big] \lb{A.107a} \\ 
&=(\bsI-\alpha \bsH_b)\big[\bsI-\alpha (\bsI-\alpha \bsH_b)^{-1}ST\big] \lb{A.108a}
\end{align}
and hence,
\begin{align}
{\det}_{L^2((a,b);\cH)}(\bsI-\alpha \bsK)&={\det}_{L^2((a,b);\cH)}(\bsI-\alpha \bsH_a)   \no \\
& \quad \times {\det}_{L^2((a,b);\cH)}\big(\bsI-\alpha (\bsI-\alpha \bsH_a)^{-1}QR\big)
\no \\ 
&={\det}_{L^2((a,b);\cH)}\big(\bsI-\alpha (\bsI-\alpha \bsH_a)^{-1}QR\big) \no \\
&={\det}_{\cH_2}\big(I_{\cH_2}-\alpha R(\bsI-\alpha \bsH_a)^{-1}Q\big) \lb{A.109a} \\
&={\det}_{\cH_1 \oplus \cH_2}(U(b; \alpha)). \lb{A.110a} 
\end{align}
Similarly,
\begin{align} 
{\det}_{L^2((a,b);\cH)}(\bsI-\alpha \bsK)&={\det}_{L^2((a,b);\cH)}(\bsI-\alpha \bsH_b)   \no \\
& \quad \times {\det}_{L^2((a,b);\cH)}\big(\bsI-\alpha (\bsI-\alpha \bsH_b)^{-1}ST\big)
\no \\ 
&={\det}_{L^2((a,b);\cH)}\big(\bsI-\alpha (\bsI-\alpha \bsH_b)^{-1}ST\big) \no \\
&={\det}_{\cH_1}\big(I_{\cH_1}-\alpha T(\bsI-\alpha \bsH_b)^{-1}S\big) \lb{A.111a} \\
&={\det}_{\cH_1 \oplus \cH_2}(U(a; \alpha)). \lb{A.112a} 
\end{align}
Relations \eqref{A.110a} and \eqref{A.112a} follow directly from 
taking the limit $x\uparrow b$ and $x\downarrow a$ in \eqref{A.37}. This 
proves \eqref{A.100a}--\eqref{A.105a}.
\end{proof}
%%%%%%%%%%%

The results \eqref{A.96a}--\eqref{A.100a}, \eqref{A.102a}, \eqref{A.103a}, 
\eqref{A.105a} can be found in the finite-dimensional context 
($\dim(\cH) < \infty$ and $\dim(\cH_j) < \infty$, $j=1,2$) in 
Gohberg, Goldberg, and Kaashoek \cite[Theorem 3.2]{GGK90} and in 
Gohberg, Goldberg, and Krupnik \cite[Sects.\ XIII.5, XIII.6]{GGK00} under the 
additional assumptions that $a,b$ are finite. The more general case where 
$(a,b)\subseteq\bbR$ is an arbitrary interval, as well as \eqref{A.101a} and 
\eqref{A.104a}, still in the case where $\cH$ and $\cH_j$, $j=1,2$, are 
finite-dimensional, was derived in \cite{GM03}. To the best of our 
knowledge, the present infinite-dimensional results are new. 

For brevity we restricted ourselves to the case where 
$\bsK \in \cB_1\big(L^2((a,b);\cH)\big)$. However, following the detailed discussion in 
the finite-dimensional case in \cite{GM03}, it is possible to extend all considerations 
in this section to the case where $\bsK \in \cB_2\big(L^2((a,b);\cH)\big)$. This 
extension, will appear in \cite{GN13}.

%%%%%%%%%%%%%%%%%%%%%%%%%%%%%%%%%%%%%%
%%%%%%%%%%%%%%%%%%%%%%%%%%%%%%%%%%%%%%
\section{Schr\"odinger Operators and Associated Fredholm Determinants} 
\lb{s3}
%%%%%%%%%%%%%%%%%%%%%%%%%%%%%%%%%%%%%%
%%%%%%%%%%%%%%%%%%%%%%%%%%%%%%%%%%%%%%

The principal purpose of this section is the application of the abstract Fredholm 
determinant reduction theory in Section \ref{s2}, as summarized in 
Theorem \ref{tA.13}, to the concrete case of Schr\"odinger operators with 
short-range potentials that take values in the trace class. 

We will freely use the notation employed in Appendix \ref{sA}, identifying 
$(a,b)=\bbR$, and throughout this section we make the following assumption: 

%%%%%%%%%%
\begin{hypothesis} \lb{hB.1}
Suppose that $V:\bbR \to \cB_1(\cH)$ is a weakly
measurable operator-valued function with $\|V(\cdot)\|_{\cB_1(\cH)}\in L^1(\bbR)$,
and assume that $V(x) = V(x)^*$ for a.e.\ $x \in \bbR$.
\end{hypothesis}
%%%%%%%%%%

(While it is possible to drop the self-adjointness condition $V(x) = V(x)^*$ for 
a.e.\ $x \in \bbR$, non-self-adjointness of $V$ plays no role in this paper and hence will 
not be considered in the following.)

We introduce the self-adjoint operators in $L^2(\bbR;\cH)$ defined by 
\begin{align}
&\bsH_0 f=-f'', \quad f\in \dom(\bsH_0)=W^{2,2}(\bbR;\cH), \lb{B.1} \\
&\bsH f=\tau f, \lb{B.2} \\
&f\in\dom(\bsH)=\{g\in L^2(\bbR;\cH) \,|\, g,g' \in AC_{\loc}(\bbR;\cH); \, \tau g
\in L^2(\bbR;\cH)\},     \no
\end{align}
where, as in Appendix \ref{sA}, $(\tau f)(x) = -f''(x) + V(x) f(x)$ for a.e.\ $x\in\bbR$. 
In addition, we introduce the self-adjoint operator $\bsV$ in $L^2(\bbR;\cH)$ by 
\begin{align}
& (\bsV f)(x) = V(x) f(x),    \no \\
& f \in \dom(\bsV) = \bigg\{g \in L^2(\bbR;\cH) \, \bigg| \, g(x) \in \dom(V(x)) 
\text{ for a.e.\ $x \in \bbR$,}     \lb{B.2a} \\
& \hspace*{2cm} x \mapsto V(x) g(x) \text{ is (weakly) measurable,} \, 
\int_{\bbR} dx \, \|V(x) g(x)\|^2_{\cH} < \infty\bigg\}.      \no
\end{align}

Next we turn to the $\cB(\cH)$-valued Jost solutions $f_\pm (z,\cdot)$ of
$-\psi''(z)+V\psi(z)=z\psi(z)$, $z\in\bbC$ (in the sense described in Corollary \ref{c2.5}), given by 
\begin{align}
\begin{split}
f_\pm (z,x)=e^{\pm iz^{1/2}x} I_{\cH} - \int_x^{\pm\infty} dx' \,
g_0(z,x,x')V(x')f_\pm (z,x'),& \lb{B.3} \\
z \in \bbC, \; \Im(z^{1/2})\geq 0, \; x\in\bbR,&    
\end{split} 
\end{align}
where $g_0 (z,\cdot,\cdot)$ is the $\cB(\cH)$-valued Volterra Green's function of 
$\bsH_0$ given by 
\begin{equation}
g_0(z,x,x') = z^{-1/2}\sin(z^{1/2}(x-x')) I_{\cH}, \quad z \in \bbC, \; x, x' \in \bbR.       \lb{B.4}
\end{equation} 
One observes that up to normalization (cf.\ \eqref{2.58}), $f_\pm (z,\cdot)$ coincide with 
the Weyl--Titchmarsh solutions $\psi_{\alpha}(z,\cdot,x_0)$ in \eqref{2.59} associated 
with $\bsH$.

We also recall the $\cB(\cH)$-valued Green's function of $\bsH_0$,
\begin{align}
\begin{split} 
G_0(z,x,x')=\big(\bsH_0 - z \bsI\big)^{-1}(x,x')=
\f{i}{2z^{1/2}}e^{iz^{1/2}|x-x'|} I_{\cH},&     \lb{B.5} \\ 
z \in \bbC \backslash [0,\infty), \; \Im(z^{1/2})>0, \; x,x'\in\bbR.&
\end{split} 
\end{align}
For subsequent purposes we also recall the $\cB(\cH)$-valued integral kernel  
of $(\bsH_0 - z \bsI\big)^{-1/2}$,
\begin{align}
\begin{split} 
(\bsH_0 - z \bsI\big)^{-1/2}(z,x,x') = \pi^{-1}K_0(- i z^{1/2} |x-x'|) I_{\cH},&    \\
z \in \bbC \backslash [0,\infty), \; \Im(z^{1/2})>0, \; x, x'\in\bbR, \; x \neq x',&    \lb{B.6}
\end{split} 
\end{align} 
where $K_0(\cdot)$ denotes the modified (irregular) Bessel function of order zero
(cf.\ \cite[Sect.\ 9.6]{AS72}). Formulas such as \eqref{B.5} and \eqref{B.6} follow 
from elementary Fourier transform arguments as detailed in \cite[p.\ 57--59]{RS75}.   
Relation \eqref{B.6} requires in addition the integral representation 
\cite[No.\ 3.7542]{GR80} for $K_0(\cdot)$. 

One recalls that the $\cB(\cH)$-valued Jost function $\cF$ associated with the pair of self-adjoint 
operators $(\bsH, \bsH_0)$ is then given by
\begin{align}
\cF(z)&=\f{1}{2iz^{1/2}} W(f_-(\ol z)^*,f_+(z))  \lb{B.7} \\
&=I_{\cH}-\f{1}{2iz^{1/2}}\int_\bbR dx \, e^{- i z^{1/2}x}V(x)f_+ (z,x),    \lb{B.8} \\
&=I_{\cH}-\f{1}{2iz^{1/2}}\int_\bbR dx \, f_- (\ol z,x)^* V(x) e^{i z^{1/2}x},       \lb{B.8a} \\
& \hspace*{2.6cm} z \in \bbC \backslash \{0\}, \; \Im(z^{1/2})\geq 0,     \no 
\intertext{and similarly,} 
\cF(\ol z)^*&=\f{-1}{2iz^{1/2}} W(f_+(\ol z)^*,f_-(z))  \lb{B.8b} \\
&=I_{\cH}-\f{1}{2iz^{1/2}}\int_\bbR dx \, e^{i z^{1/2}x}V(x)f_- (z,x),    \lb{B.8c} \\
&=I_{\cH}-\f{1}{2iz^{1/2}}\int_\bbR dx \, f_+ (\ol z,x)^* V(x) e^{- i z^{1/2}x},    \lb{B.8d} \\
& \hspace*{2.8cm} z \in \bbC \backslash \{0\}, \; \Im(z^{1/2})\geq 0,     \no 
\end{align}
where $W(F_1,F_2)(x) = F_1(x) F'_2(x) - F'_1(x) F_2(x)$, $x \in (a,b)$, denotes the Wronskian 
of $F_1$ and $F_2$ (cf.\ \eqref{2.33A}). 

Next, we recall the polar decomposition of a self-adjoint operator $S$ in a complex 
separable Hilbert space $\cK$
\begin{equation}
S = |S| U_S = U_S |S|,    \lb{B.11} 
\end{equation}
where $U_S$ is a partial isometry in $\cK$ and $|S| = (S^* S)^{1/2}$,  

Introducing the factorization of $\bsV = \bsu \bsv$, where 
\begin{equation}
\bsu = |\bsV|^{1/2} \bsU_{\bsV} = \bsU_{\bsV} |\bsV|^{1/2}, \quad  \bsv = |\bsV|^{1/2}, 
\quad \bsV = |\bsV| \bsU_{\bsV} = \bsU_{\bsV} |\bsV| = \bsu \bsv = \bsv \bsu,    \lb{B.12} 
\end{equation}
one verifies (see, e.g., \cite{GLMZ05}, \cite{Ka66} and the references cited therein) that 
\begin{align}
& (\bsH - z \bsI)^{-1} - (\bsH_0 - z \bsI)^{-1}    \lb{B.13} \\ 
& \quad = (\bsH_0 - z \bsI)^{-1} \bsv \big[\bsI + \ol{\bsu (\bsH_0 - z \bsI)^{-1} \bsv}\big]^{-1} 
\bsu (\bsH_0 - z \bsI)^{-1}, \quad z\in\bbC\backslash\sigma(\bsH). \no
\end{align}

Next, to make contact with the notation used in Appendix \ref{sA}, we now 
introduce the operator $\bsK(z)$ in $L^2(\bbR;\cH)$ by
\begin{equation}
\bsK(z)=-\ol{\bsu (\bsH_0 - z \bsI)^{-1}\bsv}, \quad
z\in\bbC\backslash [0,\infty),  \lb{B.14}
\end{equation}
with integral kernel
\begin{equation}
K(z,x,x')=-u(x)G_0(z,x,x')v(x'), \quad z \in \bbC \backslash [0,\infty), \; \Im(z^{1/2}) > 0, 
\; x,x' \in\bbR,  \lb{B.15} 
\end{equation}
and the Volterra operators $\bsH_{-\infty} (z)$, $\bsH_\infty (z)$ (cf.\ \eqref{A.4},
\eqref{A.12a}) in $L^2(\bbR;\cH)$, with integral kernel
\begin{equation}
H(z,x,x')=u(x)g_0 (z,x,x')v(x').  \lb{B.16} 
\end{equation}
Here we used the abbreviations,  
\begin{align}
\begin{split} 
& u(x) = |V(x)|^{1/2} U_{V(x)}, \quad v(x) = |V(x)|^{1/2},    \lb{B.17} \\ 
& V(x) = |V(x)| U_{V(x)} = U_{V(x)} |V(x)| = u(x) v(x) \, \text{ for a.e.\ $x\in\bbR$.} 
\end{split} 
\end{align}

Moreover, we introduce for a.e.\ $x\in\bbR$, 
\begin{align}
\begin{split}
f_1(z,x)&=-u(x) e^{iz^{1/2}x}, \hspace*{.68cm} 
g_1(z,x)=(i/2)z^{-1/2}v(x)e^{-iz^{1/2}x},  \\
f_2(z,x)&=-u(x)e^{-iz^{1/2}x}, \quad \; 
g_2(z,x)=(i/2)z^{-1/2}v(x)e^{iz^{1/2}x}. \lb{B.18}
\end{split}
\end{align}
Assuming temporarily that 
\begin{equation}
\supp(\|V(\cdot)\|_{\cB(\cH)}) \text{ is compact} \lb{B.19}
\end{equation}
(employing the notion of support for regular distributions on $\bbR$)  
in addition to Hypothesis \ref{hB.1}, identifying $\cH_1 = \cH_2 = \cH$, and introducing 
$\hat f_j(z,\cdot)$, $j=1,2$, by
\begin{align}
\hat f_1(z,x)&=f_1(z,x)-\int_x^\infty dx'\, H(z,x,x')\hat f_1(z,x'), \lb{B.20}
\\ 
\hat f_2(z,x)&=f_2(z,x)+\int_{-\infty}^x dx'\, H(z,x,x')\hat f_2(z,x'),
\lb{B.21} \\ 
& \hspace*{2mm} z\in\bbC\backslash [0,\infty), \; \Im(z^{1/2}) > 0, \; \text{a.e.\ } \, x \in\bbR, \no
\end{align}  
yields $\hat f_j(z,\cdot) \in L^2(\bbR;\cH)$, $j=1,2$, upon a standard iteration of the Volterra 
integral equations \eqref{B.20}, \eqref{B.21}. In fact, $\hat f_j(z,\cdot) \in L^2(\bbR;\cH)$, $j=1,2$, have 
compact support as long as \eqref{B.19} holds. By
comparison with \eqref{B.3}, one then identifies for all $z\in\bbC\backslash [0,\infty)$, 
$\Im(z^{1/2}) > 0$, and a.e.\ $x \in\bbR$, 
\begin{align}
\hat f_1(z,x)&=-u(x) f_+ (z,x), \lb{B.22} \\
\hat f_2(z,x)&=-u(x) f_- (z,x). \lb{B.23} 
\end{align}
We note that the temporary compact support assumption \eqref{B.19} on $\|V(\cdot)\|_{\cB(\cH)}$ has
only been introduced to guarantee that $f_j(z,\cdot), \hat f_j(z,\cdot) \in
L^2(\bbR;\cH)$, $j=1,2$ for all $z\in\bbC\backslash [0,\infty)$, 
$\Im(z^{1/2}) > 0$. This extra hypothesis will eventually be removed. 

Next we state the following basic result.

%%%%%%%%%%%%
\begin{theorem} \lb{tB.2}
Assume Hypothesis \ref{hB.1} and let $z\in\bbC\backslash [0,\infty)$. Then 
\begin{equation}
\bsK(z)\in \cB_1\big(L^2(\bbR;\cH)\big). \lb{B.24}
\end{equation}
\end{theorem}
%%%%%%%%%%%%
\begin{proof}
The well-known asymptotics of $K_0(\zeta)$ as $\zeta \to 0$ and 
$\zeta \to \infty$, $|\arg(\zeta)| < 3\pi/2$, yields  
\begin{equation}
K_0(\zeta) \underset{\zeta \to 0}{=} - \ln(\zeta) + \Oh(1), \quad 
K_0(\zeta) \underset{|\zeta| \to \infty}{=} \bigg[\f{\pi}{2 \zeta}\bigg]^{1/2} 
e^{-\zeta} \big[1 + \Oh\big(|\zeta|^{-1}\big)\big]   \lb{B.6a}
\end{equation}
(cf.\ \cite[Sect.\ 9.6]{AS72}). Thus, with $\zeta = - i z^{1/2} |x - x'|$, applying \eqref{A.82a} and 
$\|V(\cdot)\|_{\cB_1(\cH)}\in L^1(\bbR)$ by Hypothesis \ref{hB.1}, one concludes that 
\begin{equation}
\bsu (\bsH_0 - z \bsI)^{-1/2}, \ol{(\bsH_0 - z \bsI)^{-1/2}\bsv} \in \cB_2\big(L^2(\bbR;\cH)\big), 
\quad z\in\bbC\backslash [0,\infty).     \lb{B.25} 
\end{equation}
By virtue of \eqref{B.14} this proves \eqref{B.24}. 
\end{proof}
%%%%%%%%%%%%

An application of Lemma \ref{lA.7} and Theorem \ref{tA.13} then yields the
following Fredholm determinant reduction result, identifying the Fredholm determinant of 
$\bsI - \bsK(\cdot)$ and that of the $\cB(\cH)$-valued Jost function $\cF(\cdot)$ (the inverse transmission 
coefficient). The significance of \eqref{B.26} below lies in the fact that a determinant in the Hilbert 
space $L^2(\bbR;\cH)$ is reduced to one in $\cH$. (This reduction is perhaps most spectacular 
in cases where $\cH$ is finite-dimensional.) In particular, the intimate connection between $\cF(\cdot)$ 
and the inverse transmission coefficient establishes the link between $\bsK(\cdot)$ and stationary scattering theory for Schr\"odinger operators with operator-valued potentials. 

%%%%%%%%%%%%
\begin{theorem} \lb{tB.3}
Assume Hypothesis \ref{hB.1} and let $z\in\bbC\backslash [0,\infty)$. Then 
\begin{equation}
{\det}_{L^2(\bbR;\cH)}(\bsI - \bsK(z)) = {\det}_{\cH}(\cF(z)). \lb{B.26}
\end{equation} 
\end{theorem}
%%%%%%%%%%%%
\begin{proof} 
Throughout this proof we fix $z \in \bbC \backslash [0, \infty)$, $\Im(z^{1/2}) > 0$. 
Recalling our temporary assumption \eqref{B.20} that $\supp(\|V(\cdot)\|_{\cB(\cH)})$ is compact, 
Lemma \ref{lA.7} applies and one infers from \eqref{A.37} and \eqref{B.18}--\eqref{B.23} that 
the $2 \times 2$ block operator $U =(U_{j,k})_{1 \leq j, k \leq 2}$, given by 
\begin{align}
& U(z,x)=\begin{pmatrix} I_{\cH_1} - \int_x^\infty dx'\, g_1(z,x')\hat
f_1(z,x') & \int_{-\infty}^x dx'\, g_1(z,x')\hat f_2(z,x') \\ 
\int_x^\infty dx'\, g_2(z,x')\hat f_1(z,x') & I_{\cH_2} - \int_{-\infty}^x
dx'\, g_2(z,x')\hat f_2(z,x') \end{pmatrix},    \no \\
& \hspace*{10cm} x\in\bbR,     \lb{B.27}
\end{align}
becomes
\begin{align}
U_{1,1}(z,x)&= I_{\cH} + \f{i}{2z^{1/2}}\int_x^\infty dx'\,
e^{-iz^{1/2}x'}V(x') f_+ (z,x'), \lb{B.28} \\
U_{1,2}(z,x)&= -\f{i}{2z^{1/2}}\int_{-\infty}^x dx'\,
e^{-iz^{1/2}x'}V(x') f_- (z,x'), \lb{B.29} \\  
U_{2,1}(z,x)&=-\f{i}{2z^{1/2}}\int_x^\infty dx'\, e^{iz^{1/2}x'} V(x') f_+
(z,x'), \lb{B.30} \\ 
U_{2,2}(z,x)&= I_{\cH} + \f{i}{2z^{1/2}}\int_{-\infty}^x dx'\, e^{iz^{1/2}x'} V(x')f_-
(z,x'). \lb{B.31} 
\end{align}
Relations \eqref{A.101a} and \eqref{A.104a} of Theorem \ref{tA.13} then yield 
(with $\cH_1=\cH_2=\cH$)
\begin{align}
{\det}_{L^2(\bbR;\cH)}(\bsI - \bsK(z)) &= {\det}_{\cH} \bigg(I_{\cH} - 
\f{1}{2iz^{1/2}}\int_\bbR dx\, e^{\mp iz^{1/2}x}V(x)f_\pm (z,x)\bigg)    \no \\  
&= {\det}_{\cH} (\cF(z)),       \lb{B.32}
\end{align}
and hence \eqref{B.26} is proved under the additional hypothesis \eqref{B.19}.
Removing the compact support hypothesis on $\|V(\cdot)\|_{\cB(\cH)}$ becomes possible 
along the following approximation argument: first, one replaces 
\begin{equation}
V(\cdot) \, \text{ by } V_n(\cdot) = \chi_{[-n,n]} (\cdot) V(\cdot) \, \text{ a.e.\ on $\bbR$}, \; n \in \bbN,  
\lb{B.33} 
\end{equation}
(with $\chi_M(\cdot)$ the characteristic function of $M \subset \bbR$), and then denotes the quantities 
corresponding to the replacement \eqref{B.33}, in obvious notation, by $u_n$, $v_n$, $\bsV_n$, 
$\bsU_n$, $\bsv_n$, $f_{\pm,n}(z,\cdot)$, $\cF_n(z)$, $\bsK_n(z)$, $n\in\bbN$, etc. In particular, 
what has just been proven can be summarized as 
\begin{equation}
{\det}_{L^2(\bbR;\cH)}(\bsI - \bsK_n(z)) = {\det}_{\cH}(\cF_n(z)), \quad n \in \bbN.    \lb{B.34}
\end{equation} 
Applying \eqref{A.82a} to 
\begin{equation}
\bsu_n (\bsH_0 - z \bsI)^{-1/2}, \ol{(\bsH_0 - z \bsI)^{-1/2}\bsv_n} \in \cB_2\big(L^2(\bbR;\cH)\big), 
\quad n \in \bbN,    \lb{B.35} 
\end{equation}
the dominated convergence theorem readily leads to convergence of the sequences of operators 
in \eqref{B.35} to the corresponding operators in \eqref{B.25} in the  
$\cB_2(L^2(\bbR;\cH))$-norm. Thus, their product converges in trace norm, 
\begin{equation}
\lim_{n\to\infty}\|\bsK_n (z) - \bsK (z) \|_{\cB_1(L^2(\bbR;\cH))} = 0 
\end{equation} 
(recalling that Fredholm determinants are continuous with respect to the trace norm). Consequently,
\begin{equation}
\lim_{n \to \infty}{\det}_{L^2(\bbR;\cH)}(\bsI - \bsK_n(z)) = {\det}_{L^2(\bbR;\cH)}(\bsI - \bsK(z)).    
\lb{B.36}
\end{equation} 
To deal with the limit $n \to \infty$ on the right-hand side of \eqref{B.34} one first notes that by \eqref{B.8}, 
\begin{align}
\cF_n(z) &= I_{\cH}-\f{1}{2iz^{1/2}}\int_\bbR dx \, e^{\mp iz^{1/2}x}V_n(x)f_{\pm,n} (z,x),     \no \\
& = I_{\cH}-\f{1}{2iz^{1/2}}\int_\bbR dx \, V_n(x) \big[e^{\mp iz^{1/2}x} f_{\pm,n} (z,x)\big], 
\quad n \in \bbN.  \lb{B.37} 
\end{align} 
In particular, multiplying \eqref{B.3} by $e^{\mp iz^{1/2}x}$ and rearranging terms yields 
\begin{align} 
\begin{split}
\big[e^{\mp iz^{1/2}x} f_\pm (z,x)\big] = I_{\cH} - \int_x^{\pm\infty} dx' \, \big[e^{\mp iz^{1/2}x} 
g_0(z,x,x') e^{\pm iz^{1/2}x'}\big] V(x')&      \\
\times \big[e^{\mp iz^{1/2}x'} f_\pm (z,x')\big], \quad x\in\bbR,&    \lb{B.38} 
\end{split}
\end{align}
and replacing $V(\cdot)$ by $V_n(\cdot)$ in \eqref{B.38} one obtains for all $n \in \bbN$, 
\begin{align} 
\begin{split}
\big[e^{\mp iz^{1/2}x} f_{\pm,n} (z,x)\big] = I_{\cH} - \int_x^{\pm\infty} dx' \, \big[e^{\mp iz^{1/2}x} 
g_0(z,x,x') e^{\pm iz^{1/2}x'}\big] V_n(x')&      \\
\times \big[e^{\mp iz^{1/2}x'} f_{\pm,n} (z,x')\big], \quad x\in\bbR.&    \lb{B.39} 
\end{split}
\end{align}
Noticing the elementary estimate
\begin{equation}
\big|\big[e^{\mp iz^{1/2}x} g_0(z,x,x') e^{\pm iz^{1/2}x'}\big]\big| \leq |z|^{-1/2} \, 
\text{ for $x' \gtreqless x$},     \lb{B.40} 
\end{equation}
iterating both Volterra integral equations \eqref{B.38} and \eqref{B.39} then yields the uniform bounds,
\begin{equation}
\big\| \big[e^{\mp iz^{1/2}x} f_\pm (z,x)\big] \big\|_{\cB(\cH)} \leq C, \quad 
\big\| \big[e^{\mp iz^{1/2}x} f_{\pm,n} (z,x)\big] \big\|_{\cB(\cH)} \leq C, \quad 
x\in \bbR, \; n \in \bbN,    \lb{B.41} 
\end{equation}
for some fixed $C>0$ (independent of $(x,n) \in \bbR \times \bbN$). In addition, subtracting \eqref{B.38} 
from \eqref{B.39} one obtains
\begin{align} 
& e^{\mp iz^{1/2}x} [f_{\pm,n} (z,x) - f_{\pm} (z,x)]     \no \\ 
& \quad = - \int_x^{\pm\infty} dx' \, \big[e^{\mp iz^{1/2}x} g_0(z,x,x') e^{\pm iz^{1/2}x'}\big] 
[V_n(x') - V(x')] \big[e^{\mp iz^{1/2}x'} f_{\pm,n} (z,x')\big]     \no \\
& \qquad - \int_x^{\pm\infty} dx' \, \big[e^{\mp iz^{1/2}x} g_0(z,x,x') e^{\pm iz^{1/2}x'}\big] V(x')   
 \lb{B.42} \\
& \hspace*{2.7cm}  \times e^{\mp iz^{1/2}x'} \big[f_{\pm,n} (z,x') - f_{\pm} (z,x')\big],      
\quad x\in\bbR.    \no
\end{align}
Using that 
\begin{equation}
\lim_{n \to \infty} \|V_n(x) - V(x)\|_{\cB_1(\cH)} = 0 \, \text{ for a.e.\ $x\in\bbR$},      \lb{B.43} 
\end{equation}
and 
\begin{equation}
\lim_{n \to \infty} \int_{\bbR} dx \, \|V_n(x) - V(x)\|_{\cB_1(\cH)} = 0,     \lb{B.44} 
\end{equation}
an iteration of \eqref{B.42}, employing the estimates \eqref{B.40} and \eqref{B.41}, then proves that 
\begin{equation}
\lim_{n \to \infty} \big\| e^{\mp iz^{1/2}x} [f_{\pm,n} (z,x) - f_{\pm} (z,x)] \big\|_{\cB(\cH)} = 0 
\, \text{ for all $x\in\bbR$.}      \lb{B.45} 
\end{equation}
Finally, using \eqref{B.43}--\eqref{B.45} in \eqref{B.37}, one obtains
\begin{equation}
\lim_{n \to \infty} \|\cF_n (z) - I_{\cH}\|_{\cB_1(\cH)} = \|\cF (z) - I_{\cH}\|_{\cB_1(\cH)},     \lb{B.46} 
\end{equation}
and hence
\begin{equation}
\lim_{n \to \infty} {\det}_{\cH}(\cF_n (z)) = {\det}_{\cH}(\cF (z)).      \lb{B.47} 
\end{equation}
Relations \eqref{B.36} and \eqref{B.47} complete the proof of \eqref{B.26}.
\end{proof} 
%%%%%%%%%%%%

Finally, we briefly discuss the spectral shift function for the pair of self-adjoint 
operators $(\bsH, \bsH_0)$. Still assuming Hypothesis \ref{hB.1}, we note that 
a combination of \eqref{B.8}, \eqref{B.8c}, and \eqref{B.32} yields 
\begin{equation}
{\det}_{\cH} (\cF(z)) = {\det}_{\cH} (\cF(\ol z)^*) = \ol{{\det}_{\cH} (\cF(\ol z))}, 
\quad z \in \bbC \backslash [0, \infty).    \lb{B.52} 
\end{equation}
Of course, \eqref{B.52} is also a consequence of 
\begin{equation}
{\det}_{L^2(\bbR;\cH)}(\bsI - \bsK(z)) = \ol{{\det}_{L^2(\bbR;\cH)}(\bsI - \bsK(\ol z))},  
\quad z \in \bbC \backslash [0, \infty),      \lb{B.53}
\end{equation} 
which in turn follows from \eqref{B.14}, $\ol{{\det}_{L^2(\bbR;\cH)}(\bsI - \bsK(z))} 
= {\det}_{L^2(\bbR;\cH)}(\bsI - \bsK(z)^*)$, $\bsu = \bsU_{\bsV} \bsv$, $\bsv = \bsu \bsU_{\bsV}$, and 
cyclicity of the Fredholm determinant (cf.\ \eqref{A.93a}). Moreover, by the resolvent equation 
\eqref{B.13}, one concludes that
\begin{equation}
{\det}_{L^2(\bbR;\cH)}(\bsI - \bsK(z)) \neq 0, \quad z \in \bbC_+,      \lb{3.1a}
\end{equation}
since $\bbC_+ \subset \rho(\bsH)$, due to self-adjointness of $\bsH$. In addition, 
\begin{equation}
\lim_{y \to + \infty} {\det}_{L^2(\bbR;\cH)}(\bsI - \bsK(iy)) = 0,     \lb{4.55a}
\end{equation}
which can be proved as follows: combining \eqref{A.82a}, \eqref{B.6}, and 
\eqref{B.12}, one obtains
\begin{align}
& \big\|\bsu (\bsH_0 - i|y| \bsI)^{-1/2}\big\|_{\cB_2(L^2(\bbR;\cH))}^2 
= \big\|\bsv (\bsH_0 - i|y| \bsI)^{-1/2}\big\|_{\cB_2(L^2(\bbR;\cH))}^2    \no \\
& \quad = \int_{\bbR} dx \int_{\bbR} dx' \, \big\|v(x) \pi^{-1} 
K_0\big(-i (i |y|)^{1/2} |x - x'| I_{\cH}\big\|_{\cB_2(\cH)}^2    \no \\ 
& \quad \leq \int_{\bbR} dx \, \|v(x)\|_{\cB(\cH)}^2 \, \f{1}{\pi^2} 
\int_{\bbR} dx' \, \big|K_0\big(-i (i |y|)^{1/2} |x - x'|\big)\big|^2     \no \\
& \quad = \f{1}{|y|^{1/2}} \int_{\bbR} dx \, \|v(x)\|_{\cB(\cH)}^2 \, \f{1}{\pi^2} 
\int_{\bbR} dx'' \, \big|K_0\big((-i) i^{1/2} |x''|\big)\big|^2    \no \\ 
& \, \underset{|y| \to \infty}{=} \Oh\big(|y|^{-1/2}\big),    \lb{4.56a}
\end{align}
using the asymptotic behavior of $K_0((1-i) |\zeta|) = \ol{K_0((1+i) |\zeta|)}$ 
as $|\zeta| \to 0$ and $|\zeta| \to + \infty$ recorded in \cite[p.\ 375, 378]{AS72}. 
Thus,
\begin{equation}
\|\bsK(i|y|)\|_{\cB_1(L^2(\bbR;\cH))} \underset{|y| \to \infty}{=} \Oh\big(|y|^{-1/2}\big)
\end{equation}
proves \eqref{4.55a} (recalling again that Fredholm determinants are continuous with respect 
to the trace norm). In addition, one notes that ${\det}_{L^2(\bbR;\cH)}(\bsI - \bsK(\cdot))$ is 
analytic on $\bbC_+$ since $\bsK(\cdot)$ is analytic in the $\cB_1\big(L^2(\bbR;\cH)\big)$-norm 
on $\bbC_+$. An argument 
completely analogous to that leading to \eqref{4.56a} (employing once more the estimates 
for $K_0(\cdot)$ in \cite[p.\ 375, 378]{AS72},  
combined with Lebesgue domination in the integral over $x''$ in \eqref{4.56a}) 
then also shows that ${\det}_{L^2(\bbR;\cH)}(\bsI - \bsK(\cdot))$ extends 
continuously to $\bbR$.

Consequently, the spectral shift function $\xi(\cdot \, ; \bsH, \bsH_0)$ for the pair 
$(\bsH, \bsH_0)$ permits the representation, 
\begin{equation}
\xi(\lambda; \bsH, \bsH_0) = \pi^{-1} \lim_{\varepsilon \downarrow 0} 
\Im\big(\ln\big(\wti D_{\bsH/\bsH_0}(\lambda + i \varepsilon)\big)\big) 
\, \text{ for a.e.\ } \, \lambda \in \bbR,   
\end{equation}
where we used the abbreviation (cf.\ \eqref{B.32})
\begin{align}
& \wti D_{\bsH/\bsH_0} (z) = {\det}_{L^2(\bbR;\cH)} 
\big((\bsH_0 - z \bsI)^{-1/2} (\bsH - z \bsI) (\bsH_0 - z \bsI)^{-1/2}\big)  \no \\
& \quad = {\det}_{L^2(\bbR;\cH)} 
\big(\bsI + (\bsH_0 - z \bsI)^{-1/2} \bsV (\bsH_0 - z \bsI)^{-1/2}\big)   \no \\ 
& \quad = {\det}_{L^2(\bbR;\cH)} (\bsI - \bsK(z))    \no \\
& \quad = {\det}_{\cH} (\cF(z)),    \no \\
& \quad = {\det}_{\cH} \bigg(I_{\cH}-\f{1}{2iz^{1/2}}\int_\bbR dx \, e^{\mp i z^{1/2}x}V(x)f_\pm (z,x)\bigg), 
\quad z \in \rho(\bsH_0).    \lb{4.60} 
\end{align}
The choice of square root branch for $(\bsH_0 - z \bsI)^{-1/2}$ (determined either via the spectral theorem 
for $(\bsH_0 - z \bsI)^{-1/2}$, or better yet, via its integral kernel \eqref{B.6}) is of course immaterial in 
\eqref{4.60}. 

For background on the concept of the spectral shift function we refer to 
\cite[Sects.\ 191.4, 19.1.5]{BW83} and \cite[Ch.\ 8]{Ya92}.

%%%%%%%%%%%%%%%%%%%%%%%%%%%%%%
%%%%%%%%%%%%%%%%%%%%%%%%%%%%%%
\section{Fredholm Indices in terms of Spectral Shift Functions} \lb{s4}
%%%%%%%%%%%%%%%%%%%%%%%%%%%%%%
%%%%%%%%%%%%%%%%%%%%%%%%%%%%%%

In this section we recall some of the principal results on the Fredholm index 
of the operator $\bsD_\bsA^{}$ (cf.\ \eqref{2.9}) obtained in 
\cite{GLMST11} and specialize them to the case of bounded perturbations first considered by 
Pushnitski \cite{Pu08}.  
This is then used to discuss the Fredholm index of the operator $\bsD_\bsA^{}$ in 
terms of scattering theoretic tools. While this has first been established in this 
context by Pushnitski \cite{Pu08}, the principal aim of this section and the next 
is to provide an alternative, Fredholm determinant based approach to the 
fundamental trace relation in \cite{Pu08} (cf.\ eq. \eqref{2.19} below and the next 
Section \ref{s5}). In turn, this Fredholm determinant based approach permits us 
to derive new results, such as \eqref{5.54}. 

Throughout this section we make the following assumptions (equivalent to those made in 
\cite{Pu08}): 

%%%%%%%%%%%%
\begin{hypothesis} \lb{h4.1}
Suppose that $\cH$ is a complex, separable Hilbert space. \\
$(i)$ Assume $A_- \in \cB(\cH)$ is self-adjoint in $\cH$. \\
$(ii)$ Suppose there exists a family of bounded self-adjoint operators $\{B(t)\}_{t\in\bbR}$ 
in $\cH$ with $B(\cdot)$ weakly locally absolutely continuous on $\bbR$, implying the 
existence of a family of bounded self-adjoint operators $\{B'(t)\}_{t\in\bbR}$ in $\cH$ such that for 
a.e.\ $t\in\bbR$,  
\begin{equation} 
\frac{d}{dt} (g,B(t) h)_{\cH} = (g,B'(t) h)_{\cH}, \quad g, h\in\cH.     \lb{2.1}
\end{equation} 
$(iii)$ Assume that $B'(t) \in \cB_1(\cH)$, $t \in \bbR$, and 
\begin{equation}  \lb{2.2}
\int_\bbR dt \, \big\|B'(t)\big\|_{\cB_1(\cH)} < \infty.
\end{equation}
\end{hypothesis}
%%%%%%%%%%%%%

Here we recall that a family $\{B(t)\}_{t\in \bbR}$ of bounded self-adjoint operators in $\cH$ is called 
weakly locally absolutely continuous if and only if for all $g,h\in \cH$, $(g,B(t)h)_{\cH}$ is locally 
absolutely continuous on $\bbR$.

For notational simplicity later on, $B'(t)$ was defined for all $t\in\bbR$ in 
Hypothesis \ref{h4.1}\,$(ii)$; it would have been possible to introduce it for a.e.\ $t\in\bbR$ 
from the outset. 

Next, assuming Hypothesis \ref{h4.1}, we introduce the family of bounded self-adjoint operators 
$\{A(t)\}_{t\in\bbR}$ in $\cH$ by 
\begin{equation}\lb{2.3}
A(t)=A_- + B(t) \in \cB(\cH), \;  t\in\bbR,
\end{equation}
and hence conclude that 
\begin{equation}  \lb{2.3A}
A'(t) \in \cB_1(\cH), \; t \in \bbR, \, \text{ and } \,  
\int_\bbR dt \, \big\|A'(t)\big\|_{\cB_1(\cH)} < \infty.
\end{equation}
As shown later in \eqref{2.7a}, Hypothesis \ref{h4.1} implies that 
$\|B(\cdot)\|_{\cB_1(\cH)} \in L^{\infty}(\bbR; dt)$ and hence 
$\|A(\cdot)\|_{\cB(\cH)} \in L^{\infty}(\bbR; dt)$. For examples satisfying Hypothesis \ref{h4.1} 
we refer to the Schr\"odinger operators discussed in Section \ref{s5}. 

Next, we turn to a closer examination of the family $\{A(t)\}_{t\in\bbR}$ and recall some of the 
results of \cite{GLMST11} for two reasons: first, some results now simplify under 
Hypothesis \ref{h4.1}, and second, presenting them permits us to provide a fairly self-contained treatment of this material. 

We first recall the following result:

%%%%%%%%%
\begin{lemma} [\cite{GLMST11}] \lb{l4.2}
Let $\cH$ be a complex, separable Hilbert space and 
$\bbR \ni t \mapsto F(t) \in \cB_1(\cH)$. Then the following assertions $(i)$ and $(ii)$ 
are equivalent: \\
$(i)$ $\{F(t)\}_{t\in\bbR}$ is a weakly measurable family of operators in $\cB(\cH)$ 
and $\|F(\cdot)\|_{\cB_1(\cH)} \in L^1(\bbR)$.  \\
$(ii)$ $F\in L^1(\bbR; \cB_1(\cH))$. 

Moreover, if either condition $(i)$ or $(ii)$ holds, then
\begin{equation}
\bigg\|\int_{\bbR} dt \, F(t) \bigg\|_{\cB_1(\cH)} \leq 
\int_{\bbR} dt \, \|F(t)\|_{\cB_1(\cH)}   \lb{3.BI}
\end{equation}  
and the $\cB_1(\cH)$-valued function
\begin{equation}
\bbR \ni t \mapsto \int_{t_0}^t ds \, F(s), \quad t_0 \in \bbR\cup \{-\infty\},  
\lb{3.AC}
\end{equation}
is strongly absolutely continuous with respect to the norm in $\cB_1(\cH)$. \\
In addition we recall the following fact: \\
$(iii)$ Suppose that $\bbR \ni t \mapsto G(t) \in \cB_1(\cH)$ is strongly locally absolutely 
continuous in $\cB_1(\cH)$. Then $H(t) = G'(t)$ exists for a.e.\ $t\in\bbR$, $H(\cdot)$ is 
Bochner integrable over any compact interval, and hence
\begin{equation}
G(t) = G(t_0) + \int_{t_0}^t ds \, H(s), \quad t, t_0 \in \bbR.   \lb{3.FTC}
\end{equation}
\end{lemma}
%%%%%%%%%

%%%%%%%%%
\begin{remark} [\cite{GLMST11}] \lb{r4.3}
We temporarily introduce the Bochner integral in $\cB_1(\cH)$, 
\begin{equation}
C(t) = \int_{-\infty}^t ds \, B'(s) \in \cB_1(\cH), \quad t \in \bbR. 
\end{equation}
Applying Lemma \ref{l4.2}\,$(iii)$, one concludes that 
\begin{equation}
C'(t) = B'(t) \, \text{ for a.e.\  $t \in \bbR$,} 
\end{equation}
and hence, in particular, for all $f, g \in \cH$, 
\begin{equation}
(f,C'(t) g)_{\cH} = (f,B'(t) g)_{\cH} \, \text{ for a.e.\  $t \in \bbR$.} 
\end{equation}
Thus, by Hypothesis \ref{h4.1}\,$(iii)$,  
\begin{align}
\f{d}{dt} (f, C(t) g)_{\cH} &= (f,C'(t) g)_{\cH} = (f,B'(t) g)_{\cH}   \no \\
& = \f{d}{dt} (f, B(t) g)_{\cH}  \, \text{ for a.e.\  $t \in \bbR$.} 
\end{align}
Consequently, one arrives at 
\begin{equation}
C(t) = B(t) + C_0 \, \text{ for some $C_0 \in\cB_1(\cH)$.}
\end{equation}
In particular, one infers that 
\begin{equation}
\lim_{t\to -\infty} B(t) = D_- \, \text{ exists in the $\cB_1(\cH)$-norm.}  \lb{3.D-} 
\end{equation}

We now choose the convenient normalization 
\begin{equation}
D_- = 0    \lb{3.D-=0}
\end{equation}
and hence obtain 
\begin{equation} 
B(t) = \int_{-\infty}^t ds \, B'(s) \in \cB_1(\cH), \quad t \in \bbR,  
\lb{2.13jk}
\end{equation} 
implying 
\begin{equation} 
\|B(t)\|_{\cB_1(\cH)}
\leq \int_{-\infty}^t ds \, \|B'(s)\|_{\cB_1(\cH)}, \quad t \in \bbR.   \lb{2.13jl}
\end{equation} 
$\hfill \Diamond$
\end{remark}
%%%%%%%%%

%%%%%%%%%%%%
\begin{theorem} [\cite{GLMST11}, \cite{Pu08}] \lb{t4.4}
Assume Hypothesis \ref{h4.1} and define $A(t)$, $t\in\bbR$, as in \eqref{2.3}. Then there 
exists a self-adjoint operator $A_+ \in \cB(\cH)$ in $\cH$ such that 
\begin{equation}
\nlim_{t \to \infty} A(t) = A_+, \, \text{ similarly, } \, \nlim_{t \to - \infty} A(t) = A_-.      \lb{2.4}
\end{equation}
\end{theorem}
%%%%%%%%%%%%

Given Theorem \ref{t4.4} one can introduce the self-adjoint 
operator $B(+\infty)$ in $\cH$ by
\begin{equation}
B_+ := B(+\infty) = A_+ - A_-.    \lb{2.5}
\end{equation}
In addition, we also set 
\begin{equation}
B(-\infty) = 0,     \lb{2.6}
\end{equation} 
and note that
\begin{equation}  \lb{2.7}
A_+=A_- + B_+.
\end{equation}  

The integrability condition \eqref{2.2} implies that
\begin{equation}
\sup_{t\in\bbR} \|B(t)\|_{\cB_1(\cH)} 
= \sup_{t\in\bbR} \bigg\| \int_{- \infty}^t dt \, B'(t)\bigg\|_{\cB_1(\cH)} 
\leq \int_{\bbR} dt \, \|B'(t)\|_{\cB_1(\cH)} < \infty.    \lb{2.7a}
\end{equation} 
In particular, this yields that 
\begin{equation}
B_+ = [A_+ - A_-] \in \cB_1(\cH),    \lb{2.5a}
\end{equation}
and that $\bsB$, defined in terms of the family $\{B(t)\}_{t\in\bbR}$ in $\cH$ via, 
\begin{equation}
(\bsB f)(t) = B(t) f(t) \, \text{ for a.e.\ $t\in\bbR$,}   \quad f \in L^2(\bbR;\cH),    \lb{2.5aa}
\end{equation}
actually defines a bounded operator in $L^2(\bbR; \cH)$, 
\begin{equation}
\bsB \in \cB\big(L^2(\bbR; \cH)\big), \quad \|\bsB\|_{\cB(L^2(\bbR; \cH))} 
= \sup_{t\in\bbR} \|B(t)\|_{\cB(\cH)} \leq \sup_{t\in\bbR} \|B(t)\|_{\cB_1(\cH)}. 
\end{equation}

Next, let $\bsA \in \cB\big(L^2(\bbR;\cH)\big)$ be associated with the family 
$\{A(t)\}_{t\in\bbR}$ in $\cH$ by
\begin{equation}
(\bsA f)(t) = A(t) f(t) \, \text{ for a.e.\ $t\in\bbR$,}   \quad f \in L^2(\bbR;\cH),    \lb{2.8}
\end{equation}
and analogously, introduce $\bsA^{\prime}$ in $L^2(\bbR;\cH)$ in terms of 
the family $\{A'(t)\}_{t\in\bbR}$ in $\cH$, by
\begin{align}
&(\bsA^{\prime} f)(t) = A'(t) f(t) \, \text{ for a.e.\ $t\in\bbR$,}   \no \\
& f \in \dom(\bsA^{\prime}) = \bigg\{g \in L^2(\bbR;\cH) \,\bigg|\,
t \mapsto A'(t)g(t) \text{ is (weakly) measurable;}     \lb{2.8a} \\
& \hspace*{6.9cm}  
\int_{\bbR} dt \, \|A'(t) g(t)\|_{\cH}^2 < \infty\bigg\}.   \no 
\end{align} 
We note that by \cite[Lemma\ 4.6]{GLMST11} and the fact that $\bsA_- \in \cB\big(L^2(\bbR; \cH)\big)$, 
where $\bsA_-$ is defined as in \eqref{2.8} with $A(t)$ replaced by the constant fiber $A_-$,  
one has 
\begin{equation}
|\bsA'|^{1/2} (\bsH_0 - z \bsI)^{-1/2} \in \cB_2\big(L^2(\bbR; \cH)\big), 
\end{equation} 
with $\bsH_0$ given by \eqref{B.1}. In particular, this implies that $\bsA'$ is relatively form compact 
with respect to $\bsH_0$.
  
To state our results, we start by introducing in $L^2(\bbR;\cH)$ the operator
\begin{equation}
\bsD_\bsA^{} = \f{d}{dt} + \bsA,
\quad \dom(\bsD_\bsA^{})= \dom(d/dt).   \lb{2.9}
\end{equation}
Here the operator $d/dt$ in $L^2(\bbR;\cH)$  is defined by
\begin{align} 
& \bigg(\f{d}{dt}f\bigg)(t) = f'(t) \, \text{ for a.e.\ $t\in\bbR$,} 
\lb{2.10}  \\
& \, f \in \dom(d/dt) = \big\{g \in L^2(\bbR;\cH) \, \big|\,
g \in AC_{\loc}(\bbR; \cH); \, g' \in L^2(\bbR;\cH)\big\} = W^{1,2}(\bbR;\cH).   \no
\end{align}
We recall that 
\begin{align}
& g \in AC_{\loc}(\bbR; \cH) \, \text{ if and only if $g$ is of the form }        \lb{2.11} \\
& \quad  g(t) = g(t_0) + \int_{t_0}^t ds \, h(s), \; t, t_0 \in \bbR, \, \text{ for some } \, 
h \in L^1_{\loc}(\bbR;\cH), \text{ and } \, g' = h\, \text{ a.e.}   \no 
\end{align} 
(The integral in \eqref{2.11} is of course a Bochner integral.) 

It has been shown in \cite{GLMST11} (under assumptions more general than those in 
Hypothesis \ref{h4.1}) 
that $\bsD_\bsA^{}$ is densely defined and closed in $L^2(\bbR;\cH)$. 

Similarly, the adjoint operator
$\bsD_\bsA^*$ of $\bsD_\bsA^{}$ in $L^2(\bbR; \cH)$ is then given by
\begin{equation}
\bsD_\bsA^*=- \f{d}{dt} + \bsA, \quad
\dom(\bsD_\bsA^*) = \dom(d/dt) = \dom(\bsD_\bsA^{}).
\end{equation}

Using these operators, we closely follow \cite{GLMST11} and introduce in $L^2(\bbR;\cH)$ the 
nonnegative self-adjoint operators 
\begin{align} \lb{2.14}
\bsH_1=\bsD_\bsA^* \bsD_\bsA^{} = \bsH_0 \dot + \bsV_1, \quad 
\dom\big(\bsH_1^{1/2}\big) = W^{1,2}(\bbR;\cH), 
\quad \bsV_1 = \bsA^2 - \bsA',    \\
\bsH_2=\bsD_\bsA^{} \bsD_\bsA^* = \bsH_0 \dot + \bsV_2, \quad 
\dom\big(\bsH_2^{1/2}\big) = W^{1,2}(\bbR;\cH), 
\quad \bsV_2 = \bsA^2 + \bsA',
\end{align} 
in particular, $\bsH_2=\bsH_1 \dot + 2 \bsA'$. Here we denoted the form sum of operators with the 
symbol $\dot +$ and we recall that $\bsA^2 \in \cB\big(L^2(\bbR; \cH)\big)$; for more details we 
refer to \cite[Lemmas\ 4.8, 4.9]{GLMST11}.
 
Finally, we introduce the functions
\begin{align} \lb{2.16}
\begin{split}
g_z(x) & = x(x^2-z)^{-1/2}, \quad z\in\C\backslash [0,\infty), \; x\in\bbR,  \\
g(x) & = g_{-1}(x) = x(x^2+1)^{-1/2},  \quad x\in\bbR.
\end{split}
\end{align}

The next fundamental result, originally due to \cite{Pu08}, and extended in \cite{GLMST11} under hypotheses more general than those in Hypothesis \ref{h4.1}, 
will be reproved in detail in Section \ref{s5}, using alternative 
arguments; it relates the trace of the difference of the 
resolvents of $\bsH_1$ and $\bsH_2$ in $L^2(\bbR;\cH)$, and the trace 
of the difference of $g_z(A_+)$ and $g_z(A_-)$ in $\cH$.

%%%%%%%%%%%%%%%%%%%%%%%%%%%%%%%%%%%%%
\begin{theorem} \lb{t4.5}
Assume Hypothesis \ref{h4.1} and define the operators
$\bsH_1$ and $\bsH_2$ as in \eqref{2.14} and the function 
$g_z$ as in \eqref{2.16}. Then 
\begin{align}
& \big[(\bsH_2 - z \, \bsI)^{-1}-(\bsH_1 - z \, \bsI)^{-1}\big] 
\in \cB_1\big(L^2(\bbR;\cH)\big),  
\quad z\in\rho(\bsH_1) \cap \rho(\bsH_2),      \lb{2.17}  \\
& [g_z(A_+)-g_z(A_-)] \in \cB_1(\cH),  \quad 
z\in\rho\big(A_+^2\big) \cap \rho\big(A_-^2\big),     \lb{2.18}
\end{align}
and the following trace formula holds,  
\begin{align}
\begin{split}
   \lb{2.19}
   \tr_{L^2(\bbR;\cH)}\big((\bsH_2 - z \, \bsI)^{-1}-(\bsH_1 - z \, 
\bsI)^{-1}\big) = \frac{1}{2z} \tr_\cH \big(g_z(A_+)-g_z(A_-)\big),&  \\  
z\in\bbC\backslash [0,\infty).&  
\end{split}
\end{align} 
\end{theorem}
%%%%%%%%%%%%%%%%%%%%%%%%%%%%%%%%%%%%%

Equality \eqref{2.19} is fundamental in two ways: first, it reduces a trace in the 
Hilbert space $L^2(\bbR;\cH)$ to one in $\cH$, and second, it is the starting point for 
Fredholm (resp., Witten) index computations in \cite{GLMST11} (resp., \cite{CGPST13}). 
 
Since by \eqref{2.5a}, $[A_+ - A_-] \in \cB_1(\cH)$, the spectral shift function 
$\xi(\cdot \, ; A_+, A_-)$ associated with the pair $(A_+, A_-)$ is well-defined, 
satisfies 
\begin{equation}
\xi(\, \cdot \, ; A_+,A_-) \in L^1(\bbR; d\nu) 
\end{equation} 
(cf.\ \cite[Theorem\ 8.2.1]{Ya92}), and one obtains the following result: 

%%%%%%%%%%%%%
\begin{lemma} [\cite{GLMST11}, \cite{Pu08}] \lb{l4.6}
Assume Hypothesis \ref{h4.1}. Then the following trace formula holds
\begin{equation}
\tr_{\cH}\big(g_{z}(A_+) - g_{z}(A_-)\big)
  = - z \int_{\bbR} \frac{\xi(\nu; A_+, A_-) \, d\nu}{(\nu^2 - z)^{3/2}},
\quad  z\in\bbC\backslash [0,\infty).         \lb{2.24}
\end{equation}
In particular, 
\begin{equation}
\bbC\backslash [0,\infty) \ni z \mapsto 
\tr_{\cH} \big(g_{z}(A_+) - g_{z}(A_-)\big) \, \text{ is analytic.} 
\lb{2.25}
\end{equation}
\end{lemma}
%%%%%%%%%%%%%

Next, one also needs to introduce the spectral shift function 
$\xi(\,\cdot\,; \bsH_2,\bsH_1)$ associated with the pair $(\bsH_2, 
\bsH_1)$. Since $\bsH_2\geq 0$ and
$\bsH_1\geq 0$, and 
\begin{equation}
\big[(\bsH_2 + \bsI)^{-1} - (\bsH_1 + \bsI)^{-1}\big] \in \cB_1 
\big(L^2(\bbR;\cH)\big),
\end{equation}
by Theorem \ref{t4.5}, one uniquely introduces $\xi(\,\cdot\,; \bsH_2,\bsH_1)$ by requiring that
\begin{equation}
\xi(\lambda; \bsH_2,\bsH_1) = 0, \quad \lambda < 0,    \lb{2.27}
\end{equation}
and by 
\begin{align}
\begin{split} 
\tr_{L^2(\bbR;\cH)} \big((\bsH_2 - z \, \bsI)^{-1} - (\bsH_1 - z \, 
\bsI)^{-1}\big)
= - \int_{[0, \infty)}  \frac{\xi(\lambda; \bsH_2, \bsH_1) \, 
d\lambda}{(\lambda -z)^2},& \\  
z\in\bbC\backslash [0,\infty),& 
\lb{2.28}
\end{split} 
\end{align}
following \cite[Sect.\ 8.9]{Ya92}. 

%%%%%%%%%%%%
\begin{lemma} [\cite{GLMST11}] \lb{l4.7}
Assume Hypothesis \ref{h4.1}. Then 
\begin{equation}
\xi(\,\cdot\,; \bsH_2, \bsH_1) \in L^1\big(\bbR; (|\lambda| + 1)^{-1} 
d\lambda\big)   \lb{2.30}
\end{equation}
and 
\begin{equation}
\xi(\lambda; \bsH_2, \bsH_1) = \pi^{-1} \lim_{\varepsilon \downarrow 0} 
\Im\big(\ln\big(\wti D_{\bsH_2/\bsH_1}(\lambda + i \varepsilon)\big)\big) 
\, \text{ for a.e.\ } \, \lambda \in \bbR,   
\end{equation}
where we used the abbreviation 
\begin{align}
\begin{split} 
& \wti D_{\bsH_2/\bsH_1} (z) = {\det}_{L^2(\bbR;\cH)} 
\big((\bsH_1 - z \bsI)^{-1/2} (\bsH_2 - z \bsI) (\bsH_1 - z \bsI)^{-1/2}\big)    \\
& \quad = {\det}_{L^2(\bbR;\cH)} 
\big(\bsI + 2 (\bsH_1 - z \bsI)^{-1/2} \bsA' (\bsH_1 - z \bsI)^{-1/2}\big), 
\quad z \in \rho(\bsH_1).     \lb{2.31} 
\end{split} 
\end{align}
\end{lemma}
%%%%%%%%%%%%

An explicit reduction of the symmetrized perturbation determinant $\wti D_{\bsH_2/\bsH_1} (z)$ in 
the space $L^2(\bbR;\cH)$ to a particular determinant in $\cH$ will be derived in 
Section \ref{s5}, cf.\ formula \eqref{5.54}. 

Next, we detail the precise connection between $\xi$ and Fredholm perturbation  
determinants associated with the pair $(A_-, A_+)$. In particular, this will justify the perturbation 
determinants formula \eqref{4.66} in the index computation in Theorem \ref{t4.11}. 

Let 
\begin{align} 
& D_{T/S}(z) = {\det}_{\cH} ((T-z I_{\cH})(S- z I_{\cH})^{-1}) 
= {\det}_{\cH}(I_{\cH} + (T-S)(S-z I_{\cH})^{-1}),    \no \\ 
& \hspace*{9.5cm} z \in \rho(S),  
\end{align}
denote the perturbation determinant for the pair of operators $(S,T)$ in $\cH$, 
assuming $(T-S)(S - z_0 I_{\cH})^{-1} \in \cB_1(\cH)$ for some (and hence for all) 
$z_0 \in \rho(S)$. 

%%%%%%%%%%%
\begin{theorem} [\cite{GLMST11}] \lb{t4.8} 
Assume Hypothesis \ref{h4.1}, then
\begin{equation} 
\xi(\lambda; A_+,A_-) = \pi^{-1}\lim_{\e\downarrow 0} 
\Im(\ln(D_{A_+/A_-}(\lambda+i\e))) \, 
\text{ for a.e.\ } \, \lambda\in\bbR,    \lb{2.34}
\end{equation} 
where we make the choice of branch of $\ln(D_{A_+/A_-}(\cdot))$ on $\bbC_+$ such that 
\begin{equation}
\lim_{\Im(z) \to +\infty}\ln(D_{A_+/A_-}(z))=0 
\end{equation} 
(cf.\ \cite[Lemma\ 8.1.2]{Ya92}). Assuming in addition that $0\in\rho(A_-)\cap\rho(A_+)$, then for a continuous representative of  $\xi(\,\cdot\,; A_+,A_-)$ in a neighborhood of $\lambda = 0$ the equality 
\begin{equation} 
\xi(0; A_+,A_-)= \pi^{-1} \lim_{\varepsilon \downarrow 0}
\Im(\ln(D_{A_+/A_-}(i\e)))   \lb{2.35}
\end{equation} 
holds.
\end{theorem}
%%%%%%%%%%%

We note that the representation \eqref{2.34} for $\xi(\,\cdot\,; A_+,A_-)$ was proved in \cite[Theorem\ 7.6]{GLMST11} only under the additional hypothesis that $0\in\rho(A_-)\cap\rho(A_+)$, by appealing to the decomposition 
\cite[eq.\ (7.50)]{GLMST11}, involving $A_-^{-1}$. But this additional assumption is clearly unnecessary by using the corresponding decomposition
\begin{equation}
\begin{split}
&(A_+ - A_-) (A_- - z I_{\cH})^{-1} \\
&\quad = [(A_+ - A_-) (|A_-| + I_{\cH})^{-1}]
[(|A_-| + I_{\cH})(A_- - z I_{\cH})^{-1}],\quad
z \in \bbC\backslash\bbR,
\end{split}
\end{equation} 
instead.

In addition (cf., e.g., \cite[Sect.\ 8.2]{Ya92}), we recall that   
\begin{align} 
\begin{split} 
\frac{d}{dz}\ln(D_{A_+/A_-}(z)) &= - {\tr}_{\cH}\big((A_+ - z I_{\cH})^{-1}-(A_- - z I_{\cH})^{-1}\big)   \\
& = \int_{\bbR}\frac{\xi(\lambda; A_+, A_-)\,d\lambda}{(\lambda-z)^2},    \quad 
z \in \bbC\backslash\bbR.      \lb{2.36}  
\end{split} 
\end{align} 

Given these preparations, one obtains Pushnitski's formula \cite{Pu08} relating the spectral shift 
functions $\xi(\cdot \, ; \bsH_2, \bsH_1)$ and $\xi(\cdot \, ; A_+,A_-)$: 

%%%%%%%%%%%%
\begin{theorem} [\cite{GLMST11}, \cite{Pu08}] \lb{t4.9}
Assume Hypothesis \ref{h4.1} and define $\xi(\,\cdot\,; \bsH_2,\bsH_1)$ according to \eqref{2.27} and \eqref{2.28}.
Then one has for a.e.\ $\lambda>0$,
\begin{equation} 
\xi(\lambda; \bsH_2, \bsH_1)=\frac{1}{\pi}\int_{-\lambda^{1/2}}^{\lambda^{1/2}}
\frac{\xi(\nu; A_+,A_-) \, d\nu}{(\lambda-\nu^2)^{1/2}},   \lb{2.37a}
\end{equation} 
with a convergent Lebesgue integral on the right-hand side of \eqref{2.37a}. 
\end{theorem}
%%%%%%%%%%%%

Finally, we turn to the connection between the spectral shift function and the Fredholm 
index of $\bsD_\bsA^{}$. (We recall that $\bsD_\bsA^{}$ is densely defined and closed in 
$L^2(\bbR;\cH)$, cf.\ \cite{GLMST11}.) First, we state a characterization of the Fredholm 
property of $\bsD_\bsA^{}$. For this purpose, one recalls that if $T$ is a densely defined 
and closed operator in $\cH$, then 
\begin{align}
& \text{$T$ is said to be upper-semi-Fredholm if $\ran(T)$ is closed}   \no \\
& \quad \text{and $\dim(\ker(T)) < \infty$,} \\
& \text{$T$ is said to be lower-semi-Fredholm if $\ran(T)$ is closed}   \no \\
& \quad \text{and $\dim(\ker(T^*)) < \infty$,} \\
& \text{$T$ is said to be Fredholm if $\ran(T)$ is closed}   \no \\
& \quad \text{and $\dim(\ker(T)) + \dim(\ker(T^*)) < \infty$.}
\end{align}
Equivalently, $T$ is Fredholm if $\dim(\ker(T)) + {\rm codim}(\ran(T)) < \infty$. 

The following result is a generalization of 
\cite[Propositions\ 7.11, 7.12, Corollary\ 7.13]{LT05} and 
\cite[Theorem\ 1.2]{LP08} (in the context where $\bsA$ is self-adjoint). 

%%%%%%%%%%%%
\begin{theorem} \lb{t4.10}
Assume Hypothesis \ref{h4.1}. Then $\bsD_\bsA^{}$ Fredholm if and only if 
$0 \in \rho(A_+)\cap\rho(A_-)$. More precisely, if $0 \in \sigma(A_+)$ $($or $0 \in \sigma(A_-)$$)$, 
then $\bsD_\bsA^{}$ is neither upper-semi-Fredholm, nor lower-semi-Fredholm.
\end{theorem} 
%%%%%%%%%%%%
\begin{proof}
That $\bsD_\bsA^{}$ is closed and densely defined in $L^2(\bbR; \cH)$, as well as the 
sufficiency of the condition $0 \in \rho(A_+)\cap\rho(A_-)$ for the Fredholm property of 
$\bsD_\bsA^{}$ have been proven in \cite{GLMST11}. It remains to show that the condition 
$0 \in \rho(A_+)\cap\rho(A_-)$ is also necessary for the Fredholm property of $\bsD_\bsA^{}$. 

Arguing by contradiction, let $0 \in \sigma (A_+)$. Then there exists a sequence 
\begin{equation} 
\{x_n\}_{n \in \bbN} \subset \cH, \text{ with $\|x_n\|_{\cH} = 1$, $n \in \bbN$, such that 
$\|A_+ x_n\|_{\cH} \underset{n\to\infty}{\longrightarrow} 0$.} 
\end{equation}
For each $n\in\bbN$, consider a smooth function $\chi_n:\bbR\to[0,1]$
such that
\begin{equation}
\chi_n(t)= \begin{cases} 1, & |t|\leq n, \\ 0, & |t|\geq n+1, \end{cases}
\quad  |\chi^{\prime}_n(t)|\leq 2, \; t\in\bbR.
\end{equation}
Next, we define $\{{\bf x}_n\}_{n \in \bbN} \subset \dom(\bsD_{\bsA})$ by 
\begin{equation}
{\bf x}_n(t)= \|\chi_n\|^{-1}_{L^2(\bbR; dt)} \chi_n (t-2n-2) x_n, \quad t\in \bbR \; n \in \bbN,
\end{equation}
and note that $\|{\bf x}_n \|_{L^2(\bbR;\cH)}=1$, $n \in \bbN$, 
and
\begin{equation}\label{4.57}
{\bf x}_n \underset{n\to\infty}{\longrightarrow} 0 \, \text{ weakly,} 
\end{equation}
since for each $n \in \bbN$ the support of ${\bf x}_n$ is contained in $[n+1,3n+3]$. 
Moreover,
\begin{align}
& \|\bsD_{\bsA}^{}{\bf x}_n\|^2_{L^2(\bbR;\cH)}   \no \\
& \quad = \|\chi_n\|^{-2}_{L^2(\bbR; dt)} \left \|\chi^{\prime}_n(\cdot-2n-2)x_n 
+ \chi_n(\cdot-2n-2) A(\,\cdot\,)x_n \right \|^2_{L^2(\bbR;\cH)}    \no \\
& \quad \leq 2\|\chi_n\|^{-2}_{L^2(\bbR; dt)} \|\chi^{\prime}_n(\cdot-2n-2)x_n \|^2_{L^2(\bbR;\cH)}    
\no \\
& \qquad + 2\|\chi_n\|^{-2}_{L^2(\bbR; dt)} 
\left \| \chi_n(\cdot-2n-2) A(\,\cdot\,)x_n \right \|^2_{L^2(\bbR;\cH)}  \no \\
& \quad \leq 2\|\chi_n\|^{-2}_{L^2(\bbR; dt)} \|\chi^{\prime}_n \|^2_{L^2(\bbR; dt)} 
+ 2 \sup_{t \in [n+1, 3n+3]} \|A(t)x_n\|_{\cH}^2     \no \\
& \quad \leq  2\|\chi_n\|^{-2}_{L^2(\bbR; dt)} \|\chi^{\prime}_n \|^2_{L^2(\bbR; dt)} 
+ 4 \|A_+ x_n\|_{\cH}^2 + 4\sup_{t \in [n+1, 3n+3]} \|B(t)x_n\|_{\cH}^2   \no \\
& \quad \leq  2\|\chi_n\|^{-2}_{L^2(\bbR; dt)} \|\chi^{\prime}_n \|^2_{L^2(\bbR; dt)} 
+ 4 \|A_+ x_n\|_{\cH}^2 + 4 \sup_{t \in [n+1, 3n+3]}\|B(t)\|_{\cB(\cH)}^2.
\end{align}
Since 
\begin{equation}
\|\chi_n\|^{-2}_{L^2(\bbR; dt)} \|\chi^{\prime}_n \|^2_{L^2(\bbR; dt)} 
\underset{n \to \infty}{\longrightarrow} 0,
\end{equation}
by construction,
and
\begin{equation}
\sup_{t \in [n+1, 3n+3]}\|B(t)\|_{\cB(\cH)} \underset{n \to \infty}{\longrightarrow} 0, 
\end{equation}
by hypothesis (cf.\ \eqref{2.7a}), one obtains 
\begin{equation}
\|\bsD_{\bsA}^{}{\bf x}_n\|^2_{L^2(\bbR;\cH)} \underset{n \to \infty}{\longrightarrow} 0.  \lb{4.61} 
\end{equation}

At this point we mention the following sufficient condition for $T$ to be an upper-semi-Fredholm 
operator, to be found in \cite{Wo59} (where properties of $T^*T$ are exploited). 
For convenience of the reader we briefly state it together with its short proof:  
\begin{align}
& \text{Let $T$ be densely defined and closed in $\cK$. Then $T$ is {\bf not} upper-semi-Fredholm} 
\no \\  
& \quad \text{if there exists a sequence $\{y_n\}_{n\in\bbN} \subset \cK$, 
$\|y_n\|_{\cK} = 1$, $n \in\bbN$, such that}   \no \\
& \quad \text{$y_n \underset{n\to\infty}{\longrightarrow} 0$ weakly and 
$\|T y_n\|_{\cK} \underset{n\to\infty}{\longrightarrow} 0$.}    \lb{4.62}
\end{align}
For the proof of \eqref{4.62} it suffices to note that if $T$ is upper semi-Fredholm, then there exists 
$S \in \cB(\cK)$ and $K \in \cB_{\infty}(\cK)$ such that  $ST = I_{\cK} + K$ on $\dom(T)$ 
(cf.\ \cite[Ch.\ 7]{Sc02}). Then one has $\|STy_n \| \underset{n \to \infty}{\longrightarrow} 0$, and,  
 in view of compactness of $K$, 
$\|ST y_n\|_{\cK} = \|(I_{\cK} + K) y_n\|_{\cK} \underset{n \to \infty}{\longrightarrow} 1$, 
a contradiction. 

In view of \eqref{4.57} and \eqref{4.61}, the latter property implies that $\bsD_{\bsA}^{}$ is not 
an upper-semi-Fredholm, let alone a Fredholm operator. 
Similarly, passing to adjoint operators, $\bsD_{\bsA}^{}$ is not a lower-semi-Fredholm 
operator (using the fact that the adjoint of a closed, densely defined, upper-semi-Fredholm 
operator is lower-semi-Fredholm).

Finally, considering
$\{{\bf \widetilde x}_n\}_{n \in \bbN} \subset \dom(\bsD_{\bsA}^{})$ defined by
\begin{equation}
{\bf \widetilde x}_n(t)= \|\chi_n\|^{-1}_{L^2(\bbR; dt)} \chi_n (t+2n+2) x_n, \quad n \in \bbN,
\end{equation}
and noting $\|{\bf \widetilde x}_n \|_{L^2(\bbR;\cH)}=1$, $n \in \bbN$, and
\begin{equation}
{\bf \widetilde x}_n \underset{n \to \infty}{\longrightarrow} 0 \, \text{ weakly,} 
\end{equation}
entirely analogous arguments show that $0 \in \sigma (A_-)$ implies once more that 
$\bsD_{\bsA}^{}$ is not upper-semi-Fredholm (and again not lower-semi-Fredholm 
by passing to adjoints).
\end{proof} 
%%%%%%%%%%%%

To the best of our knowledge, the ``only if'' characterization of the upper/lower Fredholm 
property of $\bsD_\bsA^{}$ under Hypothesis \ref{h4.1} is new. In this context we also refer to 
\cite[Propositions 7.11, 7.12,  Corollary 7.13 ]{LT05}, \cite[Theorem\ 1.2]{LP08}, and 
\cite{SS08} for closely related results (which do not seem to imply Theorem \ref{t4.10}). 
Thus, if $0$ belongs to either of 
$\sigma(A_+)$ or $\sigma(A_-)$ then $\bsD_\bsA^{}$ ceases to be Fredholm and hence 
the spectral assumptions in Theorem \ref{t4.11} below do not restrict the generality of the 
result. For spectral results concerning $\bsD_\bsA^{}$, we refer, for instance, to 
\cite{Ba91}, \cite{CGPST13}, and \cite[Section~2.2.2]{CL99}, and references cited therein.
The literature concerned with model operators of the type $\bsD_\bsA^{}$ is so enormous that 
no comprehensive account can possibly be provided here. Instead, we refer to the monograph 
\cite{CL99} and the extensive list of references cited therein. 

We conclude this section with the following result on the Fredholm index of $\bsD_\bsA^{}$.

%%%%%%%%%%%%%%%
\begin{theorem} [\cite{GLMST11}, \cite{Pu08}] \lb{t4.11}
Assume Hypothesis \ref{h4.1} and suppose that $0 \in \rho(A_+)\cap\rho(A_-)$. Then 
$\bsD_\bsA^{}$ is a Fredholm operator in $L^2(\bbR; \cH)$ and the following equalities hold:
\begin{align}
\ind (\bsD_\bsA^{}) & = \xi(0_+; \bsH_2, \bsH_1)    \lb{4.64} \\  
& = \xi(0;A_+,A_-)     \lb{4.65} \\ 
& = \pi^{-1} \lim_{\varepsilon \downarrow 0}\Im\big(\ln\big({\det}_{\cH}
\big((A_+ - i \varepsilon I_{\cH})(A_- - i \varepsilon I_{\cH})^{-1}\big)\big)\big),       \lb{4.66}
\end{align}
with a choice of branch of $\ln({\det}_{\cH}(\cdot))$ on $\bbC_+$ analogous to that in Theorem \ref{t4.8}. 
\end{theorem}
%%%%%%%%%%%%%%%

%%%%%%%%%%%%%%%%%%%%%%%%%%%%%%
%%%%%%%%%%%%%%%%%%%%%%%%%%%%%%
\section{A New Proof of the Trace Relation \eqref{2.19}}    \lb{s5}
%%%%%%%%%%%%%%%%%%%%%%%%%%%%%%
%%%%%%%%%%%%%%%%%%%%%%%%%%%%%%

In our final section we now provide the promised new proof of the fundamental trace relation 
\eqref{2.19}, employing basic notions in scattering theory and the theory of 
Fredholm determinants. Along the way we also derive the new result \eqref{5.54}, which  
provides an explicit computation of the symmetrized perturbation determinant 
$\wti D_{\bsH_2/\bsH_1} (z)$ in \eqref{2.31} in terms of a particular determinant in $\cH$. 
Together with the new proof of the trace relation \eqref{2.19}, the determinant formula \eqref{5.54} represents the principal new result of this section.  

Our proof differs from the one by Pushnitski \cite{Pu08} in a variety of ways: Foremost, we 
explicitly exploit the factorizations $\bsH_1=\bsD_\bsA^* \bsD_\bsA^{}$ and 
$\bsH_2=\bsD_\bsA^{} \bsD_\bsA^*$ (also dubbed ``supersymmetry''), in addition, we apply 
the reduction theory for determinants outlined in Theorem \ref{tA.13} which then permits us to 
go a step further and also derive the new formula \eqref{5.54}. 

Throughout this section we assume Hypothesis \ref{h4.1}. In addition, we will temporarily assume that
\begin{equation}
\supp(\|A'(\cdot)\|_{\cB(\cH}) \, \text{ is compact.}     \lb{3.0}
\end{equation}

%%%%%%%%%%
\begin{proof}[Proof of The Trace Relation \eqref{2.19}] 
Introducing 
\begin{align}
\begin{split} 
k_{\pm}(z) = \big[z I_{\cH} - A_{\pm}^2\big]^{1/2} = \int_{\sigma(A_{\pm})} dE_{A_{\pm}}(\alpha) \, [z - \alpha^2]^{1/2},& \\ 
\Im \big([z - \alpha^2]^{1/2}\big) \geq 0, \; z \in \bbC, \; \alpha \in \bbR,&    \lb{3.1}
\end{split}  
\end{align}
with $\{E_{A_\pm}(\alpha)\}_{\alpha \in \bbR}$ denoting the strongly right-continuous family of 
spectral projections for $A_{\pm}$, the $\cB(\cH)$-valued Jost solutions associated with 
$\bsH_j$, $j=1,2$, are given by
\begin{align}
&f_{j,\pm} (z,t) = e^{\pm i k_{\pm}(z) t} I_{\cH}       \no \\ 
& \quad - \int_t^{\pm \infty} dt' \, k_{\pm}(z)^{-1} \sin[k_{\pm}(z) (t - t')] \big[A(t')^2 - A_{\pm}^2
+ (-1)^j A'(t')\big] f_{j,\pm} (z,t'),    \no \\
& \hspace*{7.7cm} j = 1,2, \; z \in \bbC, \; t \in \bbR.   \lb{3.2}
\end{align}
In this context one notes that 
\begin{equation}
[\pm i k_{\pm} (z)]^* = \pm i k_{\pm} (z), \quad z < 0.    \lb{3.3}
\end{equation}

Then, for $z \in \bbC$, $f_{j,\pm}(z,\cdot)$ are $\cB(\cH)$-valued solutions of 
\begin{align}
\bsH_1 f_{1,\pm} = \bsD_\bsA^* \bsD_\bsA^{} f_{1,\pm} = z f_{1,\pm},    \lb{3.4} \\
\bsH_2 f_{2,\pm} = \bsD_\bsA^{} \bsD_\bsA^* f_{2,\pm} = z f_{2,\pm},    \lb{3.5} 
\end{align}
in the sense described in Corollary \ref{c2.5}. Since asymptotically, as $t \to \pm \infty$, 
\begin{equation}
\big\|f_{j,\pm} (z,t) - e^{\pm i k_{\pm}(z) t} I_{\cH} \big\|_{\cB(\cH)} \underset{t \to \pm \infty}{=}  \oh(1), \quad z \in \bbC,    \lb{3.6} 
\end{equation}
and hence 
\begin{equation}
\|(\bsD_\bsA^{} f_{1,\pm})(z,t) - [\pm i k_{\pm} (z) + A_{\pm}] e^{\pm i k_{\pm}(z) t} I_{\cH} \|_{\cB(\cH)} 
\underset{t \to \pm \infty}{=} \oh(1),  \quad z \in \bbC,    \lb{3.7} 
\end{equation}
the supersymmetric structure of $\bsH_j$, $j=1,2$, and \eqref{3.4} and \eqref{3.5} imply
\begin{equation}
\bsD_\bsA^{} \bsD_\bsA^* (\bsD_\bsA^{} f_{1,\pm}) = z (\bsD_\bsA^{} f_{1,\pm}),    \quad z \in \bbC, 
\end{equation}
and thus,
\begin{equation}
f_{2,\pm} (z,t) = (\bsD_\bsA^{} f_{1,\pm})(z,t) [\pm i k_{\pm} (z) + A_{\pm}]^{-1},   \quad 
 z \in \bbC, \; t \in \bbR.    \lb{3.9}
\end{equation}
(Actually, the temporary hypothesis \eqref{3.0} permits one to replace $\oh(1)$ on the right-hand 
sides of \eqref{3.6} and \eqref{3.7} by $0$, but we don't need to use this fact.) 

Next, we recall that the Wronskian of $\cB(\cH)$-valued functions $\phi$ and $\psi$ (which together 
with their derivatives are assumed to be continuous with respect to $t \in \bbR$) is given by 
$W(\phi, \psi)(t) = \phi(t) \psi'(t) - \phi'(t) \psi(t)$, $t \in \bbR$, 
where $\prime$ denotes differentiation with respect to $t \in \bbR$ (as usual throughout this 
paper). Then a somewhat lengthy, but straightforward computation, employing \eqref{3.3}--\eqref{3.5} 
(and hence the crucial factorization of $\bsH_j$, $j=1,2$, into products of $\bsD_\bsA^{}$ and 
$\bsD_\bsA^*$) yields
\begin{align}
\begin{split} 
W(\bsD_\bsA^{} f_{1,-}(\ol z,\cdot)^*, \bsD_\bsA^{} f_{1,+}(z,\cdot))(t) 
= z W(f_{1,-}(\ol z,\cdot)^*, f_{1,+}(z,\cdot))(t),&    \\  
z \in \bbC, \; t \in \bbR.&     \lb{3.12} 
\end{split}
\end{align}
Clearly, the identity \eqref{3.12} lies at the heart of the supersymmetric approach employed. 

In this context we recall that \eqref{3.4}, \eqref{3.5} explicitly read,
\begin{align}
- f_{j,\pm}''(z,t) + V_j(t) f_{j,\pm} (z,t) = z f_{j,\pm} (z,t), \quad j=1,2, \; 
z \in \bbC, \; t \in \bbR, 
\end{align} 
and hence,
\begin{align}
- f_{j,\pm}''(\ol z,t)^* + f_{j,\pm} (\ol z,t)^* V_j(t) = z f_{j,\pm} (\ol z,t)^*, \quad j=1,2, \; 
z \in \bbC, \; t \in \bbR, 
\end{align} 
where
\begin{equation}
V_j(t) = A(t)^2 + (-1)^j A'(t), \quad j=1,2, \; t \in \bbR. 
\end{equation}
In addition, one recalls that if $\bsH_j y_{j,k}(z,t) = z y_{j,k}(z,t)$ for $j, k \in \{1,2\}$, then the 
Wronskian 
\begin{equation}
W(y_{j,1}(\ol z,\cdot)^*, y_{j,2}(z,\cdot)), \quad j = 1,2, 
\end{equation}
is independent of $t \in \bbR$. 

The fact that $V_2(t) - V_1(t) = 2 A'(t)$ for a.e.\ $t\in\bbR$, yields the following Volterra-type 
integral equations relating $f_{2,\pm} (z,\cdot)$ and $f_{1,\pm} (z,\cdot)$,
\begin{equation}
f_{2\pm} (z,t) = f_{1,\pm} (z,t) - \int_x^{\pm \infty} dt' \, H_1(z,t,t') [2 A'(t)] f_{2,\pm}(z,t'), \quad 
z \in \rho(\bsH_1), \; t \in \bbR,      \lb{3.19}
\end{equation}
where
\begin{align}
H_1(z,t,t') & = 
f_{1,+}(z,t) W(f_{1,-}(\ol z,\cdot)^*, f_{1,+}(z,\cdot))^{-1} f_{1,-}(\ol z,t')^*      \no \\
& \quad + f_{1,-}(z,t) W(f_{1,+}(\ol z,\cdot)^*, f_{1,-}(z,\cdot))^{-1} f_{1,+}(\ol z,t')^*,    \lb{3.20} \\
& \hspace*{4.78cm}   z \in \rho(\bsH_1), \; t, t' \in \bbR,   \no 
\end{align}
denotes the $\cB(\cH)$-valued Volterra Green's function associated with $\bsH_1$ (to be 
distinguished from the $\cB(\cH)$-valued (Fredholm) Green's function of $\bsH_1$, the integral kernel of the 
resolvent of $\bsH_1$, cf.\ \eqref{3.37}).
 
While \eqref{3.19} is well-known in the special scalar case where $\dim(\cH) = 1$ 
(cf.\ \cite{BGGSS87}), 
it requires some effort to derive its present analog in the infinite-dimensional context.  Introducing the 
$\cB(\cH)$-valued half-line Weyl--Titchmarsh functions (cf.\ Theorem \ref{t2.15}, 
\eqref{2.59})
\begin{equation}
m_{1,\pm}(z,t) = f_{1,\pm}'(z,t) f_{1,\pm}(z,t)^{-1}, \quad  
z \in \bbC\backslash\bbR, \; t \in \bbR,   \lb{3.21}
\end{equation}
one has the (Nevanlinna--Herglotz function related, cf.\ \cite{GWZ12}) 
property (cf.\ \eqref{3.59A})
\begin{equation} 
m_{1,\pm} (\ol z,t)^* = m_{1,\pm} (z,t), \quad  z \in \bbC\backslash\bbR, \; 
t \in \bbR, 
\end{equation}
implying (cf.\ \eqref{2.63}, \eqref{2.65}) 
\begin{align}
G_1(z,t,t) & = - f_{1,+}(z,t) W(f_{1,-}(\ol z,\cdot)^*, f_{1,+}(z,\cdot))^{-1} f_{1,-}(\ol z,t)^*   \no \\
& = f_{1,-}(z,t) W(f_{1,+}(\ol z,\cdot)^*, f_{1,-}(z,\cdot))^{-1} f_{1,+}(\ol z,t)^*     \no \\
& = [m_{1,-}(z,t) - m_{1,+}(z,t) ]^{-1}, \quad z \in \bbC\backslash\bbR, \; t \in \bbR.     
\lb{3.23} 
\end{align}
This yields 
\begin{align}
& f_{1,+}'(z,t) W(f_{1,-}(\ol z,\cdot)^*, f_{1,+}(z,\cdot))^{-1} f_{1,-}(\ol z,t)^* 
= \big[I_{\cH} - m_{1,-}(z,t) m_{1,+}(z,t)^{-1}\big]^{-1},    \\
& f_{1,-}'(z,t) W(f_{1,+}(\ol z,\cdot)^*, f_{1,-}(z,\cdot))^{-1} f_{1,+}(\ol z,t)^* 
= \big[I_{\cH} - m_{1,+}(z,t) m_{1,-}(z,t)^{-1}\big]^{-1}, 
\end{align} 
and hence 
\begin{align}
\begin{split} 
& f_{1,+}'(z,t) W(f_{1,-}(\ol z,\cdot)^*, f_{1,+}(z,\cdot))^{-1} f_{1,-}(\ol z,t)^*     \\
& \quad + f_{1,-}'(z,t) W(f_{1,+}(\ol z,\cdot)^*, f_{1,-}(z,\cdot))^{-1} f_{1,+}(\ol z,t)^* = I,   
\end{split} 
\end{align}
and similarly,
\begin{align} 
& f_{1,+}'(z,t) W(f_{1,-}(\ol z,\cdot)^*, f_{1,+}(z,\cdot))^{-1} f_{1,-}' (\ol z,t)^*    \no \\
& \qquad + f_{1,-}'(z,t) W(f_{1,+}(\ol z,\cdot)^*, f_{1,-}(z,\cdot))^{-1} f_{1,+}'(\ol z,t)^*  \no \\
& \quad = \big[(m_{1,-}(\ol z,t)^*)^{-1} - m_{1,+}(z,t)^{-1}\big]^{-1} 
+ \big[(m_{1,+}(\ol z,t)^*)^{-1} - m_{1,-}(z,t)^{-1}\big]^{-1}    \no \\ 
& \quad = 0.       \lb{3.27}
\end{align}
At this point one can show that $f_{2,\pm} (z,\cdot)$ as defined by the Volterra-type integral equations \eqref{3.19} 
indeed satisfy $\bsH_2 f_{2,\pm} (z,\cdot) = z f_{2,\pm} (z,\cdot)$ 
in the sense described in Corollary \ref{c2.5}, upon differentiating \eqref{3.19} twice with respect to $t \in \bbR$, employing 
\eqref{3.21}--\eqref{3.27}. 

Next, we will use \eqref{3.19} to relate the Wronskians $W(f_{1,-}(\ol z,\cdot)^*, f_{1,+}(z,\cdot))$ and 
$W(f_{2,-}(\ol z,\cdot)^*, f_{2,+}(z,\cdot))$. For this purpose we note that \eqref{3.19} and \eqref{3.23} imply 
\begin{align}
& f_{2,\pm}(z,t) = f_{1,\pm}(z,t) - \int_x^{\pm \infty} dt' \, \big[f_{1,+}(z,t) W(f_{1,-}(\ol z,\cdot)^*, f_{1,+}(z,\cdot))^{-1} 
f_{1,-}(\ol z,t')^*    \no \\
& \quad + f_{1,-}(z,t) W(f_{1,+}(\ol z,\cdot)^*, f_{1,-}(z,\cdot))^{-1}  f_{1,+}(\ol z,t')^*\big] [2 A'(t')] f_{2,\pm}(z,t'),      \\ 
& f_{2,\pm}'(z,t) = f_{1,\pm}'(z,t) - \int_x^{\pm \infty} dt' \, \big[f_{1,+}'(z,t) W(f_{1,-}(\ol z,\cdot)^*, f_{1,+}(z,\cdot))^{-1} 
f_{1,-}(\ol z,t')^*    \no \\
& \quad + f_{1,-}'(z,t) W(f_{1,+}(\ol z,\cdot)^*, f_{1,-}(z,\cdot))^{-1}  f_{1,+}(\ol z,t')^*\big] [2 A'(t')] f_{2,\pm}(z,t'),      \\ 
& f_{2,\pm}(\ol z,t)^* = f_{1,\pm}(\ol z,t)^* + \int_x^{\pm \infty} dt' \, f_{2,\pm}(\ol z,t')^* [2 A'(t')] \big[f_{1,-} (z,t')     \no \\
& \hspace*{5.3cm} \times W(f_{1,+}(\ol z,\cdot)^*, f_{1,-}(z,\cdot))^{-1} f_{1,+} (\ol z,t)^*   \no \\
& \quad + f_{1,+} (z,t') W(f_{1,-}(\ol z,\cdot)^*, f_{1,+}(z,\cdot))^{-1} f_{1,-} (\ol z,t)^*\big],     \\
& f_{2,\pm}'(\ol z,t)^* = f_{1,\pm}'(\ol z,t)^* + \int_x^{\pm \infty} dt' \, f_{2,\pm}(\ol z,t')^* [2 A'(t')] \big[f_{1,-} (z,t')     \no \\
& \hspace*{5.3cm} \times W(f_{1,+}(\ol z,\cdot)^*, f_{1,-}(z,\cdot))^{-1} f_{1,+}' (\ol z,t)^*   \no \\
& \quad + f_{1,+} (z,t') W(f_{1,-}(\ol z,\cdot)^*, f_{1,+}(z,\cdot))^{-1} f_{1,-}' (\ol z,t)^*\big],      \\
& \hspace*{5.3cm} z \in \rho(\bsH_1), \; t \in \bbR.     \no   
\end{align}
Thus, one computes for $t>0$ sufficiently large and beyond the support of $A'(\cdot)$, 
\begin{align}
& W(f_{2,-}(\ol z,\cdot)^*, f_{2,+}(z,\cdot)) = f_{2,-} (\ol z,t)^* f_{2,+}' (z,t) - f_{2,-}' (\ol z,t)^* f_{2,+} (z,t) \no \\
& \quad \underset{t \uparrow \infty}{=} \bigg(I_{\cH} - \int_{\bbR} dt' \, f_{2,-} (\ol z,t')^* [2 A'(t')] f_{1,+}(z,t') 
W(f_{1,-}(\ol z,\cdot)^*, f_{1,+}(z,\cdot))^{-1} \bigg)    \no \\
& \qquad \quad \times \big[ f_{1,-} (\ol z,t)^* f_{1,+}' (z,t) - f_{1,-}' (\ol z,t)^* f_{1,+} (z,t)\big] 
+ \oh(1)   \no \\
& \quad = W(f_{1,-}(\ol z,\cdot)^*, f_{1,+}(z,\cdot)) - \int_{\bbR} dt' \, f_{2,-} (\ol z,t')^* [2 A'(t')] f_{1,+}(z,t').   \lb{3.32}
\end{align}
Taking adjoints in \eqref{3.32} and replacing $z$ by $\ol z$ then yields 
\begin{align}
\begin{split}
& W(f_{2,+}(\ol z,\cdot)^*, f_{2,-}(z,\cdot)) = W(f_{1,+}(\ol z,\cdot)^*, f_{1,-}(z,\cdot))     \\
& \quad + \int_{\bbR} dt \, f_{1,+}(\ol z, t)^* [2 A'(t)] f_{2,-}(z,t),  \quad z \in \bbC.  \lb{3.33} 
\end{split}
\end{align}
Similarly, assuming $- t>0$ sufficiently large and again beyond the support of $A'(\cdot)$, one computes 
\begin{align}
& W(f_{2,-}(\ol z,\cdot)^*, f_{2,+}(z,\cdot)) = f_{2,-} (\ol z,t)^* f_{2,+}' (z,t) - f_{2,-}' (\ol z,t)^* f_{2,+} (z,t) \no \\
& \quad \underset{t \downarrow - \infty}{=} \big[ f_{1,-} (\ol z,t)^* f_{1,+}' (z,t) - f_{1,-}' (\ol z,t)^* f_{1,+} (z,t)\big]   \no \\
& \qquad \quad \;\; \times \bigg(I_{\cH} - W(f_{1,-}(\ol z,\cdot)^*, f_{1,+}(z,\cdot))^{-1} \int_{\bbR} dt' \, 
f_{1,-} (\ol z,t')^* [2 A'(t')] f_{2,+} (z,t')\bigg)    \no \\
& \hspace*{1cm} + \oh(1),   
\end{align}  
and hence,
\begin{align}
\begin{split}
& W(f_{2,-}(\ol z,\cdot)^*, f_{2,+}(z,\cdot)) = W(f_{1,-}(\ol z,\cdot)^*, f_{1,+}(z,\cdot))    \\
& \quad - \int_{\bbR} dt \, f_{1,-} (\ol z,t)^* [2 A'(t)] f_{2,+} (z,t), \quad z \in\bbC.     \lb{3.35} 
\end{split} 
\end{align}

We note as an aside that the Volterra-type integral equations \eqref{3.19} together with the Wronski relations 
\eqref{3.33} and \eqref{3.35} imply the Fredholm-type integral equations
 \begin{align}
& f_{2\pm} (z,t) = f_{1,\pm} (z,t) W(f_{1,\mp}(\ol z,\cdot)^*, f_{1,\pm}(z,\cdot))^{-1} 
W(f_{2,\mp}(\ol z,\cdot)^*, f_{2,\pm}(z,\cdot))   \no \\
& \quad - \int_{\bbR} dt' \, G_1(z,t,t') [2 A'(t)] f_{2,\pm}(z,t'), \quad 
z \in \rho(\bsH_1), \; t \in \bbR,      \lb{3.36}
\end{align}
were (cf.\ again \eqref{2.63}, \eqref{2.65}) 
\begin{align}
G_1(z,t,t') & = (\bsH_1 -z \bsI)^{-1}(t,t')     \no \\ 
& = \begin{cases} 
- f_{1,+}(z,t) W(f_{1,-}(\ol z,\cdot)^*, f_{1,+}(z,\cdot))^{-1} f_{1,-}(\ol z,t')^*, & t \geq t', \\
\;\;\, f_{1,-}(z,t) W(f_{1,+}(\ol z,\cdot)^*, f_{1,-}(z,\cdot))^{-1} f_{1,+}(\ol z,t')^*, & t \leq t',
\end{cases}     \lb{3.37} \\
& \hspace*{6.26cm}   z \in \rho(\bsH_2), \; t, t' \in \bbR,   \no 
\end{align}
denotes the $\cB(\cH)$-valued Green's function of $\bsH_1$. 

Next, recalling the polar decomposition for $A'(\cdot)$,  
\begin{equation} 
A'(t) = U_{A'(t)} |A'(t)|,  \quad t\in\bbR,    \lb{3.38a} 
\end{equation}   
with $U_{A'(t)}$ a partial isometry in $\cH$, we introduce the Birman--Schwinger-type integral operator 
$\bsK(z) \in \cB_1\big(L^2(\bbR; \cH)\big)$ (cf.\ Theorem \ref{tB.2}, particularly, \eqref{B.25}), defined by 
\begin{equation}
\bsK(z) = -2 \bsU_{\bsA'} |\bsA'|^{1/2} (\bsH_1 - z \bsI)^{-1} |\bsA'|^{1/2}, \quad z \in \rho(\bsH_1), 
\lb{3.41} 
\end{equation}
with integral kernel $K(z,t,t')$ defined by 
\begin{align}
\begin{split}
K(z,t,t') &= - 2 U_{A'(t)} |A'(t)|^{1/2} G_1(z,t,t') |A'(t')|^{1/2}     \lb{3.42} \\
&= \begin{cases} f_1 (z,t) g_1 (z,t'), & t \geq t', \\
f_2 (z,t) g_2 (z,t'), & t \leq t', 
\end{cases}  \quad z \in \rho(\bsH_1), \; t, t' \in \bbR, 
\end{split} 
\end{align}
where 
\begin{align}
f_1 (z,t) &= - 2 U_{A'(t)} |A'(t)|^{1/2} f_{1,+} (z,t),   \\
g_1 (z,t) &= - W(f_{1,-}(\ol z,\cdot)^*, f_{1,+}(z,\cdot))^{-1} f_{1,-}(\ol z,t)^* |A'(t)|^{1/2},     \\
f_2 (z,t) &= - 2 U_{A'(t)} |A'(t)|^{1/2} f_{1,-} (z,t),   \\
g_2 (z,t) &= W(f_{1,+}(\ol z,\cdot)^*, f_{1,-}(z,\cdot))^{-1} f_{1,+}(\ol z,t)^* |A'(t)|^{1/2}.   
\end{align}
Due to the compact support property of $A'(\cdot)$, and since $|A'(t)|^{1/2} \in \cB_2(\cH)$, 
$t\in\bbR$, with $\int_{\bbR} dt \, \|A'(t)\|_{\cB_1(\cH)} < \infty$ (cf.\ \eqref{2.3A}),  
one concludes that 
\begin{equation}  
\|f_j( \cdot)\|_{\cB_2(\cH)} \in L^2(\bbR), \; 
\|g_j (\cdot)\|_{\cB_2(\cH)} \in L^2(\bbR), \quad j=1,2. \lb{3.42a}
\end{equation}
Here 
\begin{equation}
(\bsU_{\bsA'} f)(t) = U_{A'(t)} f(t) \, \text{ for a.e.\ $t\in\bbR$,}  
\quad f \in L^2(\bbR;\cH).    \lb{3.42A}
\end{equation}

This yields (cf.\ \eqref{B.12}--\eqref{B.15} with $\bsH_0$, $\bsH$, $\bsV$ replaced by $\bsH_1$, 
$\bsH_2$, $2 \bsA'$, respectively), 
\begin{align}
\begin{split} 
& (\bsH_2 - z \bsI)^{-1} - (\bsH_1 - z \bsI)^{-1} 
= -2 \big[|\bsA^{\prime}|^{1/2} (\bsH_1 - {\ol z} \bsI)^{-1}\big]^*      \\
& \quad \times [\bsI - \bsK(z)]^{-1} \bsU_{\bsA'} |\bsA'|^{1/2} (\bsH_1 - z \bsI)^{-1} 
\in \cB_1\big(L^2(\bbR; \cH)\big), 
\quad  z \in \rho(\bsH_2).     \lb{3.43}
\end{split} 
\end{align}

Associated to the (Fredholm) integral kernel $K(z,\cdot,\cdot)$ one also introduces the corresponding 
Volterra integral kernel $H(z,\cdot,\cdot)$ (cf.\ \cite{GM03}),  
\begin{align}
& H(z,t,t') = f_1 (z,t) g_1 (z,t') - f_2 (z,t) g_2 (z,t')    \no \\
& \quad =  2 U_{A'(t)} |A'(t)|^{1/2} f_{1,+} (z,t) W(f_{1,-}(\ol z,\cdot)^*, f_{1,+}(z,\cdot))^{-1} f_{1,-}(\ol z,t')^* |A'(t')|^{1/2} 
\no \\
& \qquad + 2 U_{A'(t)} |A'(t)|^{1/2} f_{1,-} (z,t) W(f_{1,+}(\ol z,\cdot)^*, f_{1,-}(z,\cdot))^{-1} f_{1,+}(\ol z,t')^* |A'(t')|^{1/2},  
\no \\
& \hspace*{8.6cm}  z\in \bbC, \; t, t' \in \bbR, 
\end{align}  
and the associated Volterra integral equations 
\begin{align}
\widehat f_1 (z,t) &= f_1 (z,t) - \int_t^{\infty} dt' \, H(z,t,t') \widehat f_1 (z,t'),     \lb{3.48} \\
\widehat f_2 (z,t) &= f_2 (z,t) - \int_{-\infty}^t dt' \, H(z,t,t') \widehat f_2 (z,t'),    \lb{3.49} \\
& \hspace*{3.53cm}  z \in \bbC, \; t, t' \in \bbR.    \no 
\end{align}
A comparison of \eqref{3.48}, \eqref{3.49} with \eqref{3.19} yields 
\begin{align}
\widehat f_1 (z,t) &= - 2 U_{A'(t)} |A'(t)|^{1/2} f_{2,+} (z,t),   \\
\widehat f_2 (z,t) &= - 2 U_{A'(t)} |A'(t)|^{1/2} f_{2,-} (z,t),    \\
& \hspace*{2.28cm}  z \in \bbC, \; t, t' \in \bbR,    \no 
\end{align}
and as in \eqref{3.42a} one concludes that
\begin{equation}  
\|\widehat f_j( \cdot)\|_{\cB_2(\cH)} \in L^2(\bbR), \quad j=1,2. \lb{3.50}
\end{equation}

A direct application of Theorem \ref{tA.13}, particularly, \eqref{A.101a} and \eqref{A.104a} 
(see also Theorem \ref{tB.3}) then yields
\begin{align}
& {\det}_{L^2(\bbR;\cH)} (\bsI - \bsK(z)) = {\det}_{\cH} \bigg(I_{\cH} - \int_{\bbR} dt \, g_1(t) \widehat f_1(t)\bigg)\no\\
&\quad = {\det}_{\cH} \bigg(I_{\cH} - \int_{\bbR} dt \, g_2(t) \widehat f_2(t)\bigg)  \no \\
& \quad =  {\det}_{\cH} \bigg(I_{\cH} - W(f_{1,-}(\ol z,\cdot)^*, f_{1,+}(z,\cdot))^{-1} \int_{\bbR} dt \, 
f_{1,-} (\ol z,t)^* [2 A'(t)] f_{2,+} (z,t)\bigg)   \no \\ 
& \quad = {\det}_{\cH} \bigg(I_{\cH} + W(f_{1,+}(\ol z,\cdot)^*, f_{1,-}(z,\cdot))^{-1} \int_{\bbR} dt \, 
f_{1,+} (\ol z,t)^* [2 A'(t)] f_{2,-} (z,t)\bigg)   \no \\ 
& \quad =  {\det}_{\cH} \big(W(f_{1,-}(\ol z,\cdot)^*, f_{1,+}(z,\cdot))^{-1} 
W(f_{2,-}(\ol z,\cdot)^*, f_{2,+}(z,\cdot))\big)    \lb{3.52} \\
& \quad =  {\det}_{\cH} \big(W(f_{1,+}(\ol z,\cdot)^*, f_{1,-}(z,\cdot))^{-1} 
W(f_{2,+}(\ol z,\cdot)^*, f_{2,-}(z,\cdot))\big),    \lb{3.53} \\
& \hspace*{7.95cm}  z \in \rho(\bsH_2).    \no  
\end{align}
Here the final two equalities are an immediate consequence of the Wronski relations \eqref{3.33}, 
\eqref{3.35}. 

Next, an application of the first relation in \eqref{A.93a} to the perturbation determinant 
$\wti D_{\bsH_2/\bsH_1} (z)$ in \eqref{2.31}, taking into account the definition \eqref{3.41} of $\bsK(z)$,  
and employing the polar decomposition of $\bsA'$ as well as the factorization 
$\bsA' = \bsU_{\bsA'} |\bsA'| = \big[\bsU_{\bsA'} |\bsA'|^{1/2}\big] |\bsA'|^{1/2}$, one observes that actually, 
\begin{align}
\begin{split} 
\wti D_{\bsH_2/\bsH_1} (z) &= {\det}_{L^2(\bbR;\cH)} 
\big(\bsI + 2 (\bsH_1 - z \bsI)^{-1/2} \bsA' (\bsH_1 - z \bsI)^{-1/2}\big)   \\
&= {\det}_{L^2(\bbR;\cH)} (\bsI - \bsK(z)), \quad z \in \rho(\bsH_1).     \lb{2.53a} 
\end{split} 
\end{align} 

Combining \eqref{3.9}, \eqref{3.12}, and \eqref{3.52} finally yields 
\begin{align}
& {\det}_{L^2(\bbR;\cH)} (\bsI - \bsK(z)) = \wti D_{\bsH_2/\bsH_1} (z)    \lb{5.54} \\
& \quad = - {\det}_{\cH} \Big(z^{-1} \big[(A_+^2 - z I_{\cH})^{1/2} + A_+\big] 
\big[-(A_-^2 - z I_{\cH})^{1/2} + A_-\big]\Big), \quad z \in \rho(\bsH_1),      \no 
\end{align}
one of the principle new results of this section. 

Consequently, 
\begin{align}
& {\tr}_{L^2(\bbR;\cH)} \big((\bsH_2 - z \bsI)^{-1} - (\bsH_1 - z \bsI)^{-1}\big) 
= - \f{d}{dz} \ln\big({\det}_{L^2(\bbR;\cH)} (\bsI - \bsK(z))\big)     \no \\
& \quad = \f{d}{dz} \ln \Big({\det}_{\cH} \Big(z^{-1} \big[(A_+^2 - z I_{\cH})^{1/2} + A_+\big] 
\big[-(A_-^2 - z I_{\cH})^{1/2} + A_-\big]\Big)\Big)      \no \\
& \quad = \f{1}{2 z} {\tr}_{\cH} \big(g_z(A_+) - g_z(A_-)\big),  \quad  z \in \bbC\backslash [0,\infty).  \lb{3.55} 
\end{align}
The last equality in \eqref{3.55}, although, straightforward, requires a bit of calculation, repeatedly utilizing 
cyclicity of the trace. This proves \eqref{2.19}, subject to the constraint \eqref{3.0}. However, the latter 
is now removed by an approximation argument precisely along the lines detailed in \eqref{B.33}--\eqref{B.47}.
\end{proof}
%%%%%%%%%%

%%%%%%%%%%
\begin{remark} \lb{r3.2} (Reduction and Topological invariance.)
We emphasize the remarkable reduction of the Fredholm determinant 
${\det}_{L^2(\bbR;\cH)} (\bsI - \bsK(z))$ associated with the Hilbert space 
$L^2(\bbR;\cH)$, to a Fredholm determinant in the Hilbert space $\cH$ 
in \eqref{3.52}, \eqref{3.53}, invoking Wronski determinants of Jost solutions 
associated with $\bsH_1$ and $\bsH_2$, and in \eqref{5.54}, involving only the asymptotes $A_\pm$ of $A(t)$ as $t \to \pm \infty$.

More precisely, one notes the remarkable independence of 
\begin{align}
\begin{split} 
& {\det}_{L^2(\bbR;\cH)} (\bsI - \bsK(z)), \quad 
\tr_{L^2(\bbR;\cH)}\big((\bsH_2 - z \, \bsI)^{-1} 
- (\bsH_1 - z \, \bsI)^{-1}\big), \\
& \xi(\lambda; \bsH_2, \bsH_1), 
\, \text{ and } \, \ind (\bsD_\bsA^{})    \lb{2.41}
\end{split} 
\end{align} 
of the details of the operator path $\{A(t)\}_{t\in\bbR}$. Indeed, \eqref{5.54} settles 
the claim for ${\det}_{L^2(\bbR;\cH)} (\bsI - \bsK(z))$. In addition, since only the asymptotes 
$A_{\pm} = \nlim_{t \to \pm \infty} A(t)$ enter $\xi(\,\cdot\,; A_+,A_-)$, equations \eqref{2.19}, \eqref{2.24}, \eqref{2.37a}, and \eqref{4.65} also settle the claim for the remaining quantities 
in \eqref{2.41}. This is sometimes dubbed topological invariance in the pertinent literature (see, e.g., \cite{BGGSS87}, \cite{Ca78}, \cite{GLMST11}, \cite{GS88}, 
and the references therein). $\Diamond$
\end{remark}
%%%%%%%%%%

%%%%%%%%%%%%%%%%%%%%%%%%%%%%%%%%%%%%%%
%%%%%%%%%%%%%%% appendices %%%%%%%%%%%%%%%%
\appendix
%%%%%%%%%%%% Appendix A %%%%%%%%%%%%%%
\section{Basic Weyl--Titchmarsh Theory for Schr\"odinger Operators with Bounded 
Operator-Valued Potentials} \lb{sA}
\renewcommand{\theequation}{A.\arabic{equation}}
\renewcommand{\thetheorem}{A.\arabic{theorem}}
\setcounter{theorem}{0} \setcounter{equation}{0}
%%%%%%%%%%%%%%%%%%%%%%%%%%%%%%%%%%%%%%
%%%%%%%%%%%%%%%%%%%%%%%%%%%%%%%%%%%%%%

In this appendix we summarize some of the basic results on spectral theory for 
Schr\"odinger operators with bounded operator-valued potentials as recently developed in \cite{GWZ12}, \cite{GWZ13}. (The latter references, in particular, contain all 
proofs and the pertinent background literature on the material displayed in this preparatory section.)
The results represented below provide the necessary background for the Schr\"odinger operators 
discussed in Sections \ref{s3} and \ref{s5}. 

We start with some necessary preliminaries:
Let $(a,b) \subseteq \bbR$ be a finite or infinite interval and $\cH$ a complex, separable 
Hilbert space.
Unless explicitly stated otherwise (such as in the context of operator-valued measures in
Nevanlinna--Herglotz representations), integration of $\cH$-valued functions on $(a,b)$ will
always be understood in the sense of Bochner (cf., e.g., \cite[p.\ 6--21]{ABHN01},
\cite[p.\ 44--50]{DU77}, \cite[p.\ 71--86]{HP85}, \cite[Ch.\ III]{Mi78}, \cite[Sect.\ V.5]{Yo80} for 
details). In particular, if $p\geq 1$, the Banach space $L^p((a,b);dx;\cH)$ denotes the set of equivalence classes of strongly measurable $\cH$-valued functions which differ at most on sets of Lebesgue measure zero, such that $\|f(\cdot)\|_{\cH}^p \in L^1((a,b);dx)$. The
corresponding norm in $L^p((a,b);dx;\cH)$ is given by
\begin{equation}
\|f\|_{L^p((a,b);dx;\cH)} = \bigg(\int_{(a,b)} dx\, \|f(x)\|_{\cH}^p \bigg)^{1/p}. 
\end{equation}
In the case $p=2$, $L^2((a,b);dx;\cH)$ is a separable 
Hilbert space. By a result of Pettis \cite{Pe38}, weak measurability of $\cH$-valued functions 
implies their strong measurability.

If $g \in L^1((a,b);dx;\cH)$, $f(x)= \int_{x_0}^x dx' \, g(x')$, $x_0, x \in (a,b)$, then $f$ is
strongly differentiable a.e.\ on $(a,b)$ and
\begin{equation}
f'(x) = g(x) \, \text{ for a.e.\ $x \in (a,b)$}.
\end{equation}
In addition,
\begin{equation}
\lim_{t\downarrow 0} \f{1}{t} \int_x^{x+t} dx' \, \|g(x') - g(x)\|_{\cH} = 0 \,
\text{ for a.e.\ $x \in (a,b)$,}
\end{equation}
in particular,
\begin{equation}
\slim_{t\downarrow 0}\f{1}{t} \int_x^{x+t} dx' \, g(x') = g(x)
\, \text{ for a.e.\ $x \in (a,b)$.}
\end{equation}

Sobolev spaces $W^{n,p}((a,b); dx; \cH)$ for $n\in\bbN$ and $p\geq 1$ are defined as follows: 
$W^{1,p}((a,b);dx;\cH)$ is the set of all
$f\in L^p((a,b);dx;\cH)$ such that there exists a $g\in L^p((a,b);dx;\cH)$ and an
$x_0\in(a,b)$ such that
\begin{equation}
f(x)=f(x_0)+\int_{x_0}^x dx' \, g(x') \, \text{ for a.e.\ $x \in (a,b)$.}
\end{equation}
In this case $g$ is the strong derivative of $f$, $g=f'$. Similarly,
$W^{n,p}((a,b);dx;\cH)$ is the set of all $f\in L^p((a,b);dx;\cH)$ so that the first $n$ strong
derivatives of $f$ are in $L^p((a,b);dx;\cH)$. For simplicity of notation one also introduces
$W^{0,p}((a,b);dx;\cH)=L^p((a,b);dx;\cH)$. Finally, $W^{n,p}_{\rm loc}((a,b);dx;\cH)$ is
the set of $\cH$-valued functions defined on $(a,b)$ for which the restrictions to any
compact interval $[\alpha,\beta]\subset(a,b)$ are in $W^{n,p}((\alpha,\beta);dx;\cH)$.
In particular, this applies to the case $n=0$ and thus defines $L^p_{\rm loc}((a,b);dx;\cH)$.
If $a$ is finite we may allow $[\alpha,\beta]$ to be a subset of $[a,b)$ and denote the
resulting space by $W^{n,p}_{\rm loc}([a,b);dx;\cH)$ (and again this applies to the case
$n=0$).

Following a frequent practice (cf., e.g., the discussion in \cite[Sect.\ III.1.2]{Am95}), we
will call elements of $W^{1,1} ([c,d];dx;\cH)$, $[c,d] \subset (a,b)$ (resp.,
$W^{1,1}_{\rm loc}((a,b);dx;\cH)$), strongly absolutely continuous $\cH$-valued functions
on $[c,d]$ (resp., strongly locally absolutely continuous $\cH$-valued functions
on $(a,b)$). 

For simplicity of notation, from this point on we will omit the Lebesgue measure whenever 
no confusion can occur and henceforth simply write $L^p_{(\loc)}((a,b);\cH)$ for 
$L^p_{(\loc)}((a,b);dx;\cH)$. Moreover, in the special case where $\cH = \bbC$, we omit $\cH$ 
and typically (but not always) the Lebesgue measure and just write $L^p_{(\loc)}((a,b))$.

For brevity of notation, $\bsI$ denotes the identity operator in $L^2((a,b);\cH)$.

In addition, and also for reasons of brevity, for operator-valued functions that are measurable with 
respect to the uniform operator topology, we typically use the short cut uniformly measurable.  

Finally, one more remark on a notational convention:

%%%%%%%%%%%%%
\begin{remark} \lb{r2.1} 
To avoid possible confusion later on between two standard
notions of strongly continuous operator-valued functions $F(x)$,
$x \in (a,b)$, that is, strong continuity of $F(\cdot) h$ in $\cH$ for all $h \in\cH$ (i.e.,
pointwise continuity of $F(\cdot)$), versus strong continuity of $F(\cdot)$ in the norm
of $\cB(\cH)$ (i.e., uniform continuity of $F(\cdot)$), we will always mean pointwise continuity of 
$F(\cdot)$ in $\cH$. The same pointwise conventions will apply to the notions of strongly 
differentiable and strongly measurable operator-valued functions throughout this paper.
In particular, and unless explicitly stated otherwise, for operator-valued functions $Y$, the symbol 
$Y'$ will be understood in the strong sense; similarly,  $y'$ will denote the strong derivative for 
vector-valued functions $y$. $\Diamond$
\end{remark}
%%%%%%%%%%%%%

%%%%%%%%%%%%%
\begin{definition} \lb{d2.2}
Let $(a,b)\subseteq\bbR$ be a finite or infinite interval and
$Q:(a,b)\to\cB(\cH)$ a weakly measurable operator-valued function with
$\|Q(\cdot)\|_{\cB(\cH)}\in L^1_\loc((a,b))$, and suppose that
$f\in L^1_{\loc}((a,b);\cH)$. Then the $\cH$-valued function
$y: (a,b)\to \cH$ is called a $($strong\,$)$ solution of
\begin{equation}
- y'' + Q y = f   \lb{2.15A}
\end{equation}
if $y \in W^{2,1}_\loc((a,b);\cH)$ and \eqref{2.15A} holds a.e.\ on $(a,b)$.
\end{definition}
%%%%%%%%%%%%%

We recall our notational convention that vector-valued solutions of \eqref{2.15A} will always be viewed as strong solutions.

One verifies that $Q:(a,b)\to\cB(\cH)$ satisfies the conditions in
Definition \ref{d2.2} if and only if $Q^*$ does (a fact that will play a role later on, cf.\
the paragraph following \eqref{2.33A}).

%%%%%%%%%%%%%
\begin{theorem} \lb{t2.3}
Let $(a,b)\subseteq\bbR$ be a finite or infinite interval and
$V:(a,b)\to\cB(\cH)$ a weakly measurable operator-valued function with
$\|V(\cdot)\|_{\cB(\cH)}\in L^1_\loc((a,b))$. Suppose that
$x_0\in(a,b)$, $z\in\bbC$, $h_0,h_1\in\cH$, and $f\in
L^1_{\loc}((a,b);\cH)$. Then there is a unique $\cH$-valued
solution $y(z,\cdot,x_0)\in W^{2,1}_\loc((a,b);\cH)$ of the initial value problem
\begin{equation}
\begin{cases}
- y'' + (V - z) y = f \, \text{ on } \, (a,b)\bs E,  \\
\, y(x_0) = h_0, \; y'(x_0) = h_1,
\end{cases}     \lb{2.1a}
\end{equation}
where the exceptional set $E$ is of Lebesgue measure zero and independent
of $z$.

Moreover, the following properties hold:
\begin{enumerate}[$(i)$]
\item For fixed $x_0,x\in(a,b)$ and $z\in\bbC$, $y(z,x,x_0)$ depends jointly continuously on $h_0,h_1\in\cH$, and $f\in L^1_{\loc}((a,b);\cH)$ in the sense that
\begin{align}
\begin{split}
& \big\|y\big(z,x,x_0;h_0,h_1,f\big) - y\big(z,x,x_0;\wti h_0,\wti h_1,\wti f\big)\big\|_{\cH}    \\
& \quad \leq C(z,V)
\big[\big\|h_0 - \wti h_0\big\|_{\cH} + \big\|h_1 - \wti h_1\big\|_{\cH}
+ \big\|f - \wti f\big\|_{L^1([x_0,x];\cH)}\big],    \lb{2.1A}
\end{split}
\end{align}
where $C(z,V)>0$ is a constant, and the dependence of
$y$ on the initial data $h_0, h_1$ and the inhomogeneity $f$ is displayed 
in \eqref{2.1A}. 
\item For fixed $x_0\in(a,b)$ and $z\in\bbC$, $y(z,x,x_0)$ is strongly continuously differentiable with respect to $x$ on $(a,b)$.
\item For fixed $x_0\in(a,b)$ and $z\in\bbC$, $y'(z,x,x_0)$ is strongly differentiable with respect to $x$ on $(a,b)\bs E$.
\item For fixed $x_0,x \in (a,b)$, $y(z,x,x_0)$ and $y'(z,x,x_0)$
are entire with respect to $z$.
\end{enumerate}
\end{theorem}
%%%%%%%%%%%%%
  
%%%%%%%%%%%%%
\begin{definition} \lb{d2.4}
Let $(a,b)\subseteq\bbR$ be a finite or infinite interval and assume that
$F,\,Q:(a,b)\to\cB(\cH)$ are two weakly measurable operator-valued functions such
that $\|F(\cdot)\|_{\cB(\cH)},\,\|Q(\cdot)\|_{\cB(\cH)}\in L^1_\loc((a,b))$. Then the
$\cB(\cH)$-valued function $Y:(a,b)\to\cB(\cH)$ is called a solution of
\begin{equation}
- Y'' + Q Y = F   \lb{2.26A}
\end{equation}
if $Y(\cdot)h\in W^{2,1}_\loc((a,b);\cH)$ for every $h\in\cH$ and $-Y''h+QYh=Fh$ holds
a.e.\ on $(a,b)$.
\end{definition}
%%%%%%%%%%%%%

%%%%%%%%%%%%%
\begin{corollary} \lb{c2.5}
Let $(a,b)\subseteq\bbR$ be a finite or infinite interval, $x_0\in(a,b)$, $z\in\bbC$, $Y_0,\,Y_1\in\cB(\cH)$, and suppose $F,\,V:(a,b)\to\cB(\cH)$ are two weakly measurable operator-valued functions with
$\|V(\cdot)\|_{\cB(\cH)},\,\|F(\cdot)\|_{\cB(\cH)}\in L^1_\loc((a,b))$. Then there is a
unique $\cB(\cH)$-valued solution $Y(z,\cdot,x_0):(a,b)\to\cB(\cH)$ of the initial value
problem
\begin{equation}
\begin{cases}
- Y'' + (V - z)Y = F \, \text{ on } \, (a,b)\bs E,  \\
\, Y(x_0) = Y_0, \; Y'(x_0) = Y_1.
\end{cases} \lb{2.3a}
\end{equation}
where the exceptional set $E$ is of Lebesgue measure zero and independent
of $z$. Moreover, the following properties hold:
\begin{enumerate}[$(i)$]
\item For fixed $x_0 \in (a,b)$ and $z \in \bbC$, $Y(z,x,x_0)$ is continuously
differentiable with respect to $x$ on $(a,b)$ in the $\cB(\cH)$-norm.
\item For fixed $x_0 \in (a,b)$ and $z \in \bbC$, $Y'(z,x,x_0)$ is strongly differentiable with respect to $x$ on $(a,b)\bs E$.
\item For fixed $x_0, x \in (a,b)$, $Y(z,x,x_0)$ and $Y'(z,x,x_0)$ are entire in $z$ in
the $\cB(\cH)$-norm.
\end{enumerate}
\end{corollary}
%%%%%%%%%%%%%

%%%%%%%%%%%%
\begin{definition} \lb{d2.6}
Pick $c \in (a,b)$.
The endpoint $a$ (resp., $b$) of the interval $(a,b)$ is called {\it regular} for the operator-valued differential expression $- (d^2/dx^2) + Q(\cdot)$ if it is finite and if $Q$ is weakly measurable and $\|Q(\cdot)\|_{\cB(\cH)}\in  L^1_{\loc}([a,c])$ (resp.,
$\|Q(\cdot)\|_{\cB(\cH)}\in  L^1_{\loc}([c,b])$) for some $c\in (a,b)$. Similarly,
$- (d^2/dx^2) + Q(\cdot)$ is called {\it regular at $a$} $($resp., {\it regular at $b$}$)$ if
$a$ $($resp., $b$$)$ is a regular endpoint for $- (d^2/dx^2) + Q(\cdot)$.
\end{definition}
%%%%%%%%%%%%

We note that if $a$ (resp., $b$) is regular for $- (d^2/dx^2) + Q(x)$, one may allow for
$x_0$ to be equal to $a$ (resp., $b$) in the existence and uniqueness Theorem \ref{t2.3}.

If $f_1, f_2$ are strongly continuously differentiable $\cH$-valued functions, we define the Wronskian of $f_1$ and $f_2$ by
\begin{equation}
W_{*}(f_1,f_2)(x)=(f_1(x),f'_2(x))_\cH - (f'_1(x),f_2(x))_\cH,    \lb{2.31A}
\quad x \in (a,b).
\end{equation}
If $f_2$ is an $\cH$-valued solution of $-y''+Qy=0$ and $f_1$ is an $\cH$-valued
solution of $-y''+Q^*y=0$, their Wronskian $W_{*}(f_1,f_2)(x)$ is $x$-independent, that is,
\begin{equation}
\f{d}{dx} W_{*}(f_1,f_2)(x) = 0, \, \text{ for a.e.\ $x \in (a,b)$.}   \lb{2.32A}
\end{equation}
Equation \eqref{2.52A} will show that the right-hand side of \eqref{2.32A} actually
vanishes for all $x \in (a,b)$.

We decided to use the symbol $W_{*}(\cdot,\cdot)$ in \eqref{2.31A} to indicate its
conjugate linear behavior with respect to its first entry.

Similarly, if $F_1,F_2$ are strongly continuously differentiable $\cB(\cH)$-valued
functions, their Wronskian is defined by
\begin{equation}
W(F_1,F_2)(x) = F_1(x) F'_2(x) - F'_1(x) F_2(x), \quad x \in (a,b).    \lb{2.33A}
\end{equation}
Again, if $F_2$ is a $\cB(\cH)$-valued solution of  $-Y''+QY = 0$ and $F_1$ is a
$\cB(\cH)$-valued solution of $-Y'' + Y Q = 0$ (the latter is equivalent to
$- {(Y^{*})}^{\prime\prime} + Q^* Y^* = 0$ and hence can be handled in complete analogy
via Theorem \ref{t2.3} and Corollary \ref{c2.5}, replacing $Q$ by $Q^*$) their Wronskian will be $x$-independent,
\begin{equation}
\f{d}{dx} W(F_1,F_2)(x) = 0 \, \text{ for a.e.\ $x \in (a,b)$.}
\end{equation}

Our main interest is in the case where $V(\cdot)=V(\cdot)^* \in \cB(\cH)$ is self-adjoint,
that is, in the differential equation $\tau \eta=z \eta$, where $\eta$ represents an $\cH$-valued, respectively, $\cB(\cH)$-valued solution (in the sense of Definitions \ref{d2.2},
resp., \ref{d2.4}), and where $\tau$ abbreviates the operator-valued differential expression
\begin{equation} \label{2.4a}
\tau =-(d^2/dx^2) + V(\cdot).
\end{equation}
To this end, we now introduce the following basic assumption:

%%%%%%%%%%
\begin{hypothesis} \lb{h2.7}
Let $(a,b)\subseteq\bbR$, suppose that $V:(a,b)\to\cB(\cH)$ is a weakly
measurable operator-valued function with $\|V(\cdot)\|_{\cB(\cH)}\in L^1_\loc((a,b))$,
and assume that $V(x) = V(x)^*$ for a.e.\ $x \in (a,b)$.
\end{hypothesis}
%%%%%%%%%%

Moreover, for the remainder of this section we assume that $\alpha \in \cB(\cH)$ is a
self-adjoint operator,
\begin{equation}
\alpha = \alpha^* \in \cB(\cH).      \lb{2.4A}
\end{equation}

Assuming Hypothesis \ref{h2.7} and \eqref{2.4A}, we introduce the standard fundamental systems of operator-valued solutions of $\tau y=zy$ as follows: Since $\alpha$ is a bounded self-adjoint operator, one may define the self-adjoint operators $A=\sin(\alpha)$ and $B=\cos(\alpha)$ via the spectral theorem. One then concludes that
$\sin^2(\alpha) + \cos^2(\alpha) = I_\cH$ and $[\sin\alpha,\cos\alpha]=0$ (here
$[\cdot,\cdot]$ represents the commutator symbol). The spectral theorem implies also
that the spectra of $\sin(\alpha)$ and $\cos(\alpha)$ are contained in $[-1,1]$ and that the spectra of $\sin^2(\alpha)$ and $\cos^2(\alpha)$ are contained in $[0,1]$. Given such an operator $\alpha$ and a point $x_0\in(a,b)$ or a regular endpoint for $\tau$, we now
define $\theta_\alpha(z,\cdot, x_0,), \phi_\alpha(z,\cdot,x_0)$ as those $\cB(\cH)$-valued
solutions of $\tau Y=z Y$ (in the sense of Definition \ref{d2.4}) which satisfy the initial
conditions
\begin{equation}
\theta_\alpha(z,x_0,x_0)=\phi'_\alpha(z,x_0,x_0)=\cos(\alpha), \quad
-\phi_\alpha(z,x_0,x_0)=\theta'_\alpha(z,x_0,x_0)=\sin(\alpha).    \lb{2.5A}
\end{equation}

By Corollary 2.5\,$(iii)$, for any fixed $x, x_0\in(a,b)$, the functions
$\theta_{\alpha}(z,x,x_0)$ and $\phi_{\alpha}(z,x,x_0)$ as well as their strong $x$-derivatives are entire with respect to $z$ in the $\cB(\cH)$-norm. The same is true for the functions $z\mapsto\theta_{\alpha}(\ol{z},x,x_0)^*$ and $z\mapsto\phi_{\alpha}(\ol{z},x,x_0)^*$.

We also recall two versions of Green's formula (also called Lagrange's identity).

%%%%%%%%%%%%
\begin{lemma} \label{l2.8}
Let $(a,b)\subseteq\bbR$ be a finite or infinite interval and $[x_1,x_2]\subset(a,b)$. \\
$(i)$ Assume that $f,g\in W^{2,1}_{\rm loc}((a,b);\cH)$. Then
\begin{equation}
\int_{x_1}^{x_2} dx \, [((\tau f)(x),g(x))_\cH-(f(x),(\tau g)(x))_\cH]
= W_{*}(f,g)(x_2)-W_{*}(f,g)(x_1).     \lb{2.52A}
\end{equation}
$(ii)$ Assume that $F,\,G:(a,b)\to\cB(\cH)$ are absolutely continuous operator-valued functions such that $F',\,G'$ are again differentiable and that $F''$, $G''$ are weakly measurable. In addition, suppose that $\|F''\|_\cH,\, \|G''\|_\cH \in L^1_\loc((a,b);dx)$. Then
\begin{equation}
\int_{x_1}^{x_2} dx \, [(\tau F^*)(x)^*G(x) - F(x) (\tau G)(x)] = (FG'-F'G)(x_2)-(FG'-F'G)(x_1).
\lb{2.53A}
\end{equation}
\end{lemma}
%%%%%%%%%%%%

Next, following \cite{GWZ12} (see also \cite{GWZ13}), we briefly recall basic 
elements of the Weyl--Titchmarsh theory for self-adjoint Schr\"odinger 
operators $\bsH_{\alpha}$ in $L^2((a,b);\cH)$ associated with the operator-valued differential expression 
$\tau =-(d^2/dx^2)+V(\cdot)$, assuming regularity of the 
left endpoint $a$ and the limit point case at the right endpoint $b$ (see 
Definition \ref{d2.10}). 

As before, $\cH$ denotes a separable Hilbert space and $(a,b)$ denotes a finite or infinite interval. 
One recalls that $L^2((a,b);\cH)$ is separable (since $\cH$ is) 
and that
\begin{equation}
(f,g)_{L^2((a,b);\cH)} =\int_a^b dx \, (f(x),g(x))_\cH, \quad f,g\in L^2((a,b);\cH).
\end{equation}

Still assuming Hypothesis \ref{h2.7}, we are interested in
studying certain self-adjoint operators in $L^2((a,b);\cH)$ associated with the 
operator-valued differential expression $\tau =-(d^2/dx^2)+V(\cdot)$. These will be suitable restrictions of the {\it maximal} operator $\bsH_{\max}$ in $L^2((a,b);\cH)$ defined by
\begin{align}
& \bsH_{\max} f = \tau f,    \\
& f\in \dom(\bsH_{\max})=\big\{g\in L^2((a,b);\cH) \,\big|\, g\in W^{2,1}_{\rm loc}((a,b);\cH); \, 
 \tau g\in L^2((a,b);\cH)\big\}.     \no 
\end{align}
We also introduce the operator $\dot \bsH_{\min}$ in $L^2((a,b);\cH)$ as the restriction of $\bsH_{\max}$ to the domain
\begin{equation}
\dom\big(\dot \bsH_{\min}\big)=\{g\in\dom(\bsH_{\max})\,|\,\supp (u) \, \text{is compact in} \, (a,b)\}.
\end{equation}
Finally, the {\it minimal} operator $\bsH_{\min}$ in $L^2((a,b);\cH)$ associated with $\tau$ is then defined as the closure of $\dot \bsH_{\min}$,
\begin{equation}
\bsH_{\min} = \ol{\dot \bsH_{\min}}.
\end{equation}

%%%%%%%%%%%%%%
\begin{theorem} \label {t2.9}
Assume Hypothesis \ref{h2.7}. Then the operator $\dot \bsH_{\min}$ is densely defined. Moreover, $\bsH_{\max}$ is the adjoint of $\dot \bsH_{\min}$,
\begin{equation}
\bsH_{\max} = (\dot \bsH_{\min})^*.   \lb{3.12a}
\end{equation}
In particular, $\bsH_{\max}$ is closed. In addition, $\dot \bsH_{\min}$ is symmetric and
$\bsH_{\max}^*$ is the closure of $\dot \bsH_{\min}$, that is,
\begin{equation}
\bsH_{\max}^* = \ol{\dot \bsH_{\min}} = \bsH_{\min}.    \lb{3.13a}
\end{equation}
\end{theorem}
%%%%%%%%%%%%%%

Using the dominated convergence theorem and Green's formula \eqref{2.52A} one can show that $\lim_{x\to a}W_*(u,v)(x)$ and $\lim_{x\to b}W_*(u,v)(x)$ both exist whenever
$u,v\in\dom(\bsH_{\max})$. We will denote these limits by $W_*(u,v)(a)$ and $W_*(u,v)(b)$, respectively. Thus Green's formula also holds for $x_1=a$ and $x_2=b$ if $u$ and $v$ are in $\dom(\bsH_{\max})$, that is,
\begin{equation}\label{3.17A}
(\bsH_{\max}u,v)_{L^2((a,b);\cH)}-(u,\bsH_{\max}v)_{L^2((a,b);\cH)}
= W_*(u,v)(b) - W_*(u,v)(a).
\end{equation}
This relation and the fact that $\bsH_{\min}=\bsH_{\max}^*$ is a restriction of $\bsH_{\max}$ show that
\begin{align}
\begin{split}
& \dom(\bsH_{\min})=\{u\in\dom(\bsH_{\max}) \,|\, W_*(u,v)(b)=W_*(u,v)(a)=0   \\
& \hspace*{6.2cm} \text{ for all } v\in \dom(\bsH_{\max})\}.     \label{3.18A}
\end{split}
\end{align}

%%%%%%%%%%%%%%%
\begin{definition} \lb{d2.10}
Assume Hypothesis \ref{h2.7}.
Then the endpoint $a$ $($resp., $b$$)$ is said to be of {\it limit-point type for $\tau$} if
$W_*(u,v)(a)=0$ $($resp., $W_*(u,v)(b)=0$$)$ for all $u,v\in\dom(\bsH_{\max})$.
\end{definition}
%%%%%%%%%%%%%%%

Next, we introduce the subspaces
\begin{equation}
\cD_{z}=\{u\in\dom(\bsH_{\max}) \,|\, \bsH_{\max}u=z u\}, \quad z \in \bbC.
\end{equation}
For $z\in\bbC\backslash\bbR$, $\cD_{z}$ represent the deficiency subspaces of
$\bsH_{\min}$. Von Neumann's theory of extensions of symmetric operators implies that
\begin{equation} \label{3.20A}
\dom(\bsH_{\max})=\dom(\bsH_{\min}) \dotplus \cD_i \dotplus \cD_{-i}
\end{equation}
where $\dotplus$ indicates the direct (but not necessarily orthogonal direct) sum.

Next, we determine the self-adjoint restrictions of $\bsH_{\max}$ assuming
that $a$ is a regular endpoint for $\tau$ and $b$ is of limit-point type for $\tau$ 
(e.g., having in mind half-line situations where $a\in \bbR$, $b = \infty$). 

%%%%%%%%%%%%%
\begin{hypothesis} \lb{h2.11}
In addition to Hypothesis \ref{h2.7} suppose that $a$ is a regular
endpoint for $\tau$ and $b$ is of limit-point type for $\tau$.
\end{hypothesis}
%%%%%%%%%%%%%

%%%%%%%%%%%%%
\begin{theorem} \lb{t2.12}
Assume Hypothesis \ref{h2.11}. If $H$ is a self-adjoint restriction of
$\bsH_{\max}$, then there is a bounded and self-adjoint operator $\alpha\in\cB(\cH)$
such that
\begin{equation}
\dom(H)=\{u\in\dom(\bsH_{\max}) \,|\, \sin(\alpha)u'(a) + \cos(\alpha)u(a)=0\}.  \lb{3.29A}
\end{equation}
Conversely, for every $\alpha \in \cB(\cH)$, \eqref{3.29A} gives rise to a self-adjoint restriction of $\bsH_{\max}$ in $L^2((a,b);\cH)$.
\end{theorem}
%%%%%%%%%%%%%

Henceforth, under the assumptions of Theorem \ref{t2.12}, we denote the operator $H$ in $L^2((a,b);\cH)$ associated with the boundary condition induced by $\alpha = \alpha^* \in \cB(\cH)$, that is, the restriction of $\bsH_{\max}$ to the set
\begin{equation}
\dom(\bsH_{\alpha})=\{u\in\dom(\bsH_{\max}) \,|\, \sin(\alpha)u'(a)+\cos(\alpha)u(a)=0\}
\end{equation}
by $\bsH_{\alpha}$. 

Our next goal is to construct the square integrable solutions $Y(z,\cdot) \in\cB(\cH)$
of $\tau Y=zY$, $z\in\bbC\backslash\bbR$, the $\cB(\cH)$-valued Weyl--Titchmarsh solutions, under the assumptions that $a$ is a regular endpoint for $\tau$ and $b$ is of limit-point type for $\tau$.

Fix $c\in(a,b)$ and $z \in \rho(\bsH_{\alpha})$. For any $f_0\in\cH$ let 
$f=f_0\chi_{[a,c]}\in L^2((a,b);\cH)$ and 
$u(f_0,z,\cdot)=(\bsH_{\alpha} - z \bsI)^{-1} f \in\dom(\bsH_{\alpha})$. (We recall 
our convention that $\bsI$ denotes the identity operator in $L^2((a,b);\cH)$.) 
By the variation of constants formula,
\begin{align}
\begin{split}
u(f_0,z,x) &=\theta_{\alpha}(z,x,a)\bigg(g(z)+\int_x^c dx' \, \phi_{\alpha}(\ol z,x',a)^* f_0\bigg)\\
& \quad  +\phi_{\alpha}(z,x,a)\bigg(h(z)-\int_x^c dx' \, \theta_{\alpha}(\ol z,x',a)^* f_0\bigg)
 \end{split}
\end{align}
for suitable vectors $g(z) \in \cH$, $h(z) \in \cH$. Since
$u(f_0,z,\cdot)\in\dom(\bsH_{\alpha})$, one infers that
\begin{equation} \label{3.40A}
g(z)=-\int_a^c dx' \, \phi_{\alpha}(\ol z,x',a)^* f_0, \quad z \in \rho(\bsH_{\alpha}),
\end{equation}
and that
\begin{equation} \lb{3.40B}
h(z)=\cos(\alpha)u'(f_0,z,a) - \sin(\alpha)u(f_0,z,a)
+ \int_a^c dx' \, \theta_{\alpha}(\ol z,x',a)^* f_0, \quad z \in \rho(\bsH_{\alpha}).
\end{equation}

%%%%%%%%%%%%%%%
\begin{lemma} \lb{l2.13}
Assume Hypothesis \ref{h2.11} and suppose that $\alpha \in \cB(\cH)$ is self-adjoint. In addition, choose $c \in (a,b)$ and introduce $g(\cdot)$ and $h(\cdot)$ as in \eqref{3.40A} and \eqref{3.40B}. Then the maps
\begin{equation}
C_{1,\alpha}(c,z):\begin{cases} \cH\to\cH, \\
f_0\mapsto g(z), \end{cases} \quad
C_{2,\alpha}(c,z): \begin{cases} \cH\to\cH, \\
f_0\mapsto h(z), \end{cases}  \quad z \in \rho(\bsH_{\alpha}),
\end{equation}
are linear and bounded. Moreover, $C_{1,\alpha}(c,\cdot)$ is entire and
$C_{2,\alpha}(c,\cdot)$ is analytic on $\rho(\bsH_{\alpha})$. In addition,
$C_{1,\alpha}(c,z)$ is boundedly invertible if $z\in \bbC\backslash\bbR$ and $c$
is chosen appropriately.
\end{lemma}
%%%%%%%%%%%%%%%

Using the bounded invertibility of $C_{1,\alpha}(c,z)$ one now defines 
\begin{equation}
\psi_{\alpha}(z,x)=\theta_{\alpha}(z,x,a)
+ \phi_{\alpha}(z,x,a)C_{2,\alpha}(c,z)C_{1,\alpha}(c,z)^{-1},
\quad z \in \bbC\backslash\bbR, \; x \in [a,b),     \lb{3.49A}
\end{equation}
still assuming Hypothesis \ref{h2.11} and $\alpha = \alpha^* \in \cB(\cH)$. By Lemma \ref{l2.13}, $\psi_{\alpha}(\cdot,x)$ is analytic on
$z \in \bbC\backslash\bbR$ for fixed $x \in [a,b]$.

Since $\psi_{\alpha}(z,\cdot) f_0$ is the solution of the initial value problem
\begin{equation}
\tau y =z y, \quad y(c)=u(f_0,z,c), \; y'(c)=u'(f_0,z,c), \quad z \in \bbC\backslash\bbR,
\end{equation}
the function $\psi_{\alpha}(z,x)C_{1,\alpha}(z,c)f_0$ equals
$u(f_0,z,x)$ for $x\geq c$, and thus is square integrable for every choice of $f_0\in\cH$.
In particular, choosing $c \in (a,b)$ such that $C_{1,\alpha}(z,c)^{-1} \in \cB(\cH)$, one infers that
\begin{equation}
\int_a^b dx \, \|\psi_{\alpha}(z,x) f\|_{\cH}^2 < \infty, \quad
f \in \cH, \; z \in \bbC\backslash\bbR.
\end{equation}

Every $\cH$-valued solution of $\tau y=z y$ may be written as
\begin{equation}
y=\theta_\alpha(z,\cdot,a)f_{\alpha,a} + \phi_\alpha(z,\cdot,a)g_{\alpha,a},
\end{equation}
with
\begin{equation}
f_{\alpha,a}=(\cos\alpha)y(a)+(\sin\alpha)y'(a), \quad
g_{\alpha,a}=-(\sin\alpha)y(a)+(\cos\alpha)y'(a).
\end{equation}
Hence one can define the maps
\begin{align}
& \cC_{1,\alpha,z}:\begin{cases} \cD_z\to\cH, \\
\theta_\alpha(z,\cdot,a) f_{\alpha,a}
+ \phi_\alpha(z,\cdot,a) g_{\alpha,a} \mapsto f_{\alpha,a}, \end{cases} \\
& \cC_{2,\alpha,z}: \begin{cases} \cD_z\to\cH, \\
\theta_\alpha(z,\cdot,a) f_{\alpha,a} + \phi_\alpha(z,\cdot,a) g_{\alpha,a} \mapsto g_{\alpha,a}. \end{cases}
\end{align}

%%%%%%%%%%%%%%%%
\begin{lemma} \label{l2.14}
Assume Hypothesis \ref{h2.11}, suppose that $\alpha \in \cB(\cH)$ is self-adjoint, and
let $z\in\bbC\backslash\bbR$. Then the operators
$\cC_{1,\alpha,z}$ and $\cC_{2,\alpha,z}$ are linear bijections and hence 
\begin{equation}
\cC_{1,\alpha,z}, \, \cC_{1,\alpha,z}^{-1}, \, \cC_{2,\alpha,z}, \, 
\cC_{2,\alpha,z}^{-1} \in \cB(\cH).   \lb{3.57a}
\end{equation}
\end{lemma}
%%%%%%%%%%%%%%%%

At this point one is finally in a position to define the 
Weyl--Titchmarsh $m$-function for $z\in\bbC\backslash\bbR$ by setting
\begin{equation} \label{3.57A}
m_{\alpha}(z)=\cC_{2,\alpha,z}\cC_{1,\alpha,z}^{-1}, \quad
z\in\bbC\backslash\bbR.
\end{equation}

%%%%%%%%%%%%%%%%
\begin{theorem} \label{t2.15}
Assume Hypothesis \ref{h2.11} and that $\alpha \in \cB(\cH)$ is self-adjoint. Then 
\begin{equation}
m_{\alpha}(z) \in \cB(\cH), \quad z\in\bbC\backslash\bbR,   \lb{3.57B}
\end{equation}
and $m_{\alpha}(\cdot)$ is analytic on $\bbC\backslash\bbR$. Moreover,
\begin{equation}
m_{\alpha}(z)=m_{\alpha}(\ol z)^*, \quad z \in \bbC\backslash\bbR.    \lb{3.59A}
\end{equation}
\end{theorem}
%%%%%%%%%%%%%%%%

Thus, the $\cB(\cH)$-valued function $\psi_{\alpha}(z,\cdot)$
in \eqref{3.49A} can be rewritten in the form
\begin{equation} \label{3.58A}
\psi_{\alpha}(z,x)=\theta_{\alpha}(z,x,a)+\phi_{\alpha}(z,x,a)m_{\alpha}(z),
\quad z \in \bbC\backslash\bbR, \; x \in [a,b).
\end{equation}
In particular, this implies that $\psi_{\alpha}(z,\cdot)$ is independent of
the choice of the parameter $c \in (a,b)$ in \eqref{3.49A}.
Following the tradition in the scalar case ($\dim(\cH) = 1$), one calls 
$\psi_{\alpha}(z,\cdot)$ the {\it Weyl--Titchmarsh} solution associated
with $\tau Y = z Y$.

With the aid of the Weyl--Titchmarsh solutions one can now give a detailed description of the resolvent
$\bsR_{z,\alpha} = (\bsH_{\alpha} - z \bsI)^{-1}$ of $\bsH_{\alpha}$.

%%%%%%%%%%%%%%%%
\begin{theorem}\label{t2.16}
Assume Hypothesis \ref{h2.11} and that $\alpha \in \cB(\cH)$ is self-adjoint. Then the 
resolvent of $\bsH_{\alpha}$ is an integral operator of the type
\begin{align}
\begin{split}
\big((\bsH_{\alpha} - z \bsI)^{-1} u\big)(x)
= \int_a^b dx' \, G_{\alpha}(z,x,x')u(x'),& \\
u \in L^2((a,b);\cH), \; z \in \rho(\bsH_{\alpha}), \; x \in [a,b),&
\end{split}
\end{align}
with the $\cB(\cH)$-valued Green's function $G_{\alpha}(z,\cdot,\cdot) $ given by
\begin{equation} \label{3.63A}
G_{\alpha}(z,x,x') = \begin{cases}
\phi_{\alpha}(z,x,a) \psi_{\alpha}(\ol{z},x')^*, & a\leq x \leq x'<b, \\
\psi_{\alpha}(z,x) \phi_{\alpha}(\ol{z},x',a)^*, & a\leq x' \leq x<b,
\end{cases}  \quad z\in\bbC\backslash\bbR.
\end{equation}
\end{theorem}
%%%%%%%%%%%%%%%%

In particular, $G_{\alpha}(z,\cdot,\cdot)$ exhibits the semi-separable structure 
of integral kernels discussed in Section \ref{s2}. 

In the last part of this section we now treat Schr\"odinger operators with bounded operator-valued 
potentials on the entire real line $\bbR$. Hence we make the following basic assumption throughout 
the remainder of this section.

%%%%%%%%%%%%%%%%%%%%%%%%%%%%%%%%%%%%%%%%
\begin{hypothesis} \lb{h2.17}
$(i)$ Suppose $V:\bbR \to \cB(\cH)$ to be a weakly measurable operator-valued function 
satisfying 
\begin{equation}
\|V(\cdot)\|_{\cB(\cH)}\in L^1_\loc((a,b)), \quad V(x)=V(x)^* \, \text{ for a.e. } x\in\bbR. 
\lb{2.51}
\end{equation}
$(ii)$ Introducing the differential expression $\tau$ given by
\begin{equation}
\tau=-\f{d^2}{dx^2} + V(x), \quad x\in\bbR, \lb{2.52}
\end{equation}
we assume $\tau$ to be in the limit point case at $+\infty$ and at
$-\infty$.
\end{hypothesis}
%%%%%%%%%%%%%%%%%%%%%%%%%%%%%%%%%%%%%%%%

Associated with the differential expression $\tau$ one introduces the self-adjoint Schr\"odinger 
operator $\bsH$ in $L^2(\bbR;\cH)$ by
\begin{align}
&\bsH f=\tau f,   \lb{2.53}
\\ \no
&f\in \dom(\bsH) = \big\{g\in L^2(\bbR;\cH) \,\big|\, g, g' \in
W^{2,1}_{\loc}(\bbR;\cH); \, \tau g\in L^2(\bbR;\cH)\big\}.
\end{align}

As in the half-line context we introduce the $\cB(\cH)$-valued fundamental
system of solutions $\phi_\alpha(z,\cdot,x_0)$ and
$\theta_\alpha(z,\cdot,x_0)$, $z\in\bbC$, of
\begin{equation}
(\tau \psi)(z,x) = z \psi(z,x), \quad x\in \bbR \lb{2.54}
\end{equation}
with respect to a fixed reference point $x_0\in\bbR$, satisfying the
initial conditions at the point $x=x_0$,
\begin{align}
\begin{split}
\phi_\alpha(z,x_0,x_0)&=-\theta'_\alpha(z,x_0,x_0)=-\sin(\alpha), \\
\phi'_\alpha(z,x_0,x_0)&=\theta_\alpha(z,x_0,x_0)=\cos(\alpha), \quad
\alpha=\alpha^*\in\cB(\cH). \lb{2.55}
\end{split}
\end{align}
Again we note that by Corollary 2.5\,$(iii)$, for any fixed $x, x_0\in\bbR$, the functions $\theta_{\alpha}(z,x,x_0)$ and $\phi_{\alpha}(z,x,x_0)$ as well as their strong $x$-derivatives are entire with respect to $z$ in the $\cB(\cH)$-norm. The same is true for the functions $z\mapsto\theta_{\alpha}(\ol{z},x,x_0)^*$ and
$z\mapsto\phi_{\alpha}(\ol{z},x,x_0)^*$. In particular,
\begin{equation}
W(\theta_\alpha(\ol{z},\cdot,x_0)^*,\phi_\alpha(z,\cdot,x_0))(x)=I_\cH, \quad
z\in\bbC.  \lb{2.56}
\end{equation}

Particularly important solutions of \eqref{2.54} are the
{\it Weyl--Titchmarsh solutions} $\psi_{\pm,\alpha}(z,\cdot,x_0)$,
$z\in\bbC\backslash\bbR$, uniquely characterized by
\begin{align}
\begin{split}
&\psi_{\pm,\alpha}(z,\cdot,x_0)f\in L^2([x_0,\pm\infty);\cH), \quad f\in\cH,
\\
&\sin(\alpha)\psi'_{\pm,\alpha}(z,x_0,x_0)
+\cos(\alpha)\psi_{\pm,\alpha}(z,x_0,x_0)=I_\cH, \quad
z\in\bbC\backslash\bbR. \lb{2.57}
\end{split}
\end{align}
The crucial condition in \eqref{2.57} is again the $L^2$-property which
uniquely determines $\psi_{\pm,\alpha}(z,\cdot,x_0)$ up to constant
multiples by the limit point hypothesis of $\tau$ at $\pm\infty$. In
particular, for
$\alpha = \alpha^*, \beta = \beta^* \in \cB(\cH)$,
\begin{align}
\psi_{\pm,\alpha}(z,\cdot,x_0) = \psi_{\pm,\beta}(z,\cdot,x_0)C_\pm(z,\alpha,\beta,x_0)
\lb{2.58}
\end{align}
for some coefficients $C_\pm (z,\alpha,\beta,x_0)\in\cB(\cH)$. The normalization in \eqref{2.57} shows that
$\psi_{\pm,\alpha}(z,\cdot,x_0)$ are of the type
\begin{equation}
\psi_{\pm,\alpha}(z,x,x_0)=\theta_{\alpha}(z,x,x_0)
+ \phi_{\alpha}(z,x,x_0) m_{\pm,\alpha}(z,x_0),
\quad  z\in\bbC\backslash\bbR, \; x\in\bbR, \lb{2.59}
\end{equation}
for some coefficients $m_{\pm,\alpha}(z,x_0)\in\cB(\cH)$, the
{\it Weyl--Titchmarsh $m$-functions} associated with $\tau$, $\alpha$,
and $x_0$ (cf.\ Theorem \ref{t2.15}).

Next, we recall that $\pm m_{\pm,\alpha}(\cdot,x_0)$ are operator-valued Nevanlinna--Herglotz 
functions. By \eqref{2.54} and \eqref{2.55}, the Wronskian of $\psi_{\pm,\alpha}(\ol{z_1},x,x_0)^*$ 
and $\psi_{\pm,\alpha}(z_2,x,x_0)$ satisfies
\begin{align}
W(\psi_{\pm,\alpha}(\ol{z_1},x_0,x_0)^*,\psi_{\pm,\alpha}(z_2,x_0,x_0)) &= m_{\pm,\alpha}(z_2,x_0)- m_{\pm,\alpha}(\ol{z_1},x_0)^*,
\\
\f{d}{dx}W(\psi_{\pm,\alpha}(\ol{z_1},x,x_0)^*,\psi_{\pm,\alpha}(z_2,x,x_0)) &= (z_1-z_2)\psi_{\pm,\alpha}(\ol{z_1},x,x_0)^*\psi_{\pm,\alpha}(z_2,x,x_0), \no
\\
&\hspace{3cm} z_1,z_2\in\bbC\bs\bbR.
\end{align}
Hence, using the limit point hypothesis of $\tau$ at $\pm\infty$ and the $L^2$-property in \eqref{2.57} one obtains
\begin{align} \lb{2.60}
& (z_2-z_1)\int_{x_0}^{\pm\infty} dx\, \big(\psi_{\pm,\alpha}(\ol{z_{1}},x,x_0)f,\psi_{\pm,\alpha}(z_{2},x,x_0)g\big)_\cH
\no \\
&\quad = \big(f,[m_{\pm,\alpha}(z_2,x_0)- m_{\pm,\alpha}(\ol{z_1},x_0)^*]g\big)_\cH,
\quad  f,g\in\cH, \; z_1,z_2 \in\bbC\backslash\bbR.
\end{align}
Setting $z_1=z_2=z$ in \eqref{2.60}, one concludes
\begin{equation}
m_{\pm,\alpha}(z,x_0) = m_{\pm,\alpha}(\ol z,x_0)^*, \quad
z\in\bbC\backslash\bbR.  \lb{2.61}
\end{equation}
Choosing $f=g$ and $z_2=z$, $z_1=\ol z$ in \eqref{2.60}, one also infers
\begin{equation}
\Im(z)\int_{x_0}^{\pm\infty} dx\,\|\psi_{\pm,\alpha}(z,x,x_0)f\|_{\cH}^2
= \big(f,\Im[m_{\pm,\alpha}(z,x_0)]f\big)_\cH, \quad f\in\cH, \; z\in\bbC\backslash\bbR. \lb{2.62}
\end{equation}
Since $m_{\pm,\alpha}(\cdot,x_0)$ are analytic on $\bbC\backslash\bbR$, \eqref{2.62} yields that 
$\pm m_{\pm,\alpha}(\cdot,x_0)$ are operator-valued Nevanlinna--Herglotz functions.

In the following we abbreviate the Wronskian of $\psi_{+,\alpha}(\ol{z},x,x_0)^*$ and 
$\psi_{-,\alpha}(z,x,x_0)$ by $W(z)$. It then readily follows from the properties of the $\cB(\cH)$-valued
solutions $\theta_\alpha(z,\cdot, x_0,), \phi_\alpha(z,\cdot,x_0)$ of $\tau Y=z Y$ in \eqref{2.5A} 
that 
\begin{align}
\begin{split} 
W(z) &= \pm W(\psi_{\pm,\alpha}(\ol{z},x,x_0)^*,\psi_{\mp,\alpha}(z,x,x_0))    \\
&= m_{-,\alpha}(z,x_0) - m_{+,\alpha}(z,x_0), \quad z\in\bbC\bs\bbR.     \lb{2.64}
\end{split} 
\end{align}
The Green's function $G(z,x,x')$ of the Schr\"odinger operator $\bsH$ then reads
\begin{align}
G(z,x,x') = \psi_{\mp,\alpha}(z,x,x_0) W(z)^{-1} \psi_{\pm,\alpha}(\ol{z},x',x_0)^*,
\quad x \lesseqgtr x', \; z\in\bbC\backslash\bbR.
\lb{2.63}
\end{align}
Again, $G(z,\cdot,\cdot)$ exhibits the semi-separable structure 
of integral kernels discussed in Section \ref{s2}.
Thus, the resolvent of $\bsH$ takes on the form, 
\begin{align}
\begin{split} 
\big((\bsH-z\bsI)^{-1}f\big)(x)
=\int_{\bbR} dx' \, G(z,x,x')f(x'),&   \\
z\in\bbC\backslash\bbR, \;
x\in\bbR, \; f\in L^2(\bbR;\cH).& \lb{2.65}
\end{split} 
\end{align}
Of course, \eqref{2.57}--\eqref{2.65} extend to 
$z \in \bbC \backslash \sigma(H)$. 

\medskip

%%%%%%%%%%%%%%%%%%%%%%%%%%%%%%%%%%%%%
\noindent 
{\bf Acknowledgments.} We are indebted to Yuri Latushkin and Konstantin Makarov 
for helpful discussions, and particularly to Roger Nichols for a careful reading of the entire 
manuscript and for suggesting numerous improvements. 

F.G.\ gratefully acknowledges the extraordinary hospitality and stimulating atmosphere during 
his five-week visit to the Australian National University (ANU), Canberra, and to the University 
of New South Wales (UNSW), Sydney, in July/August of 2012, and the funding of his visit by 
the Australian Research Council. 

A.C., D.P., and F.S.\ gratefully acknowledge financial support from the Australian Research Council.
A.C.\ also thanks the Alexander von Humboldt Stiftung and colleagues at the University of M\"unster. 

Y.T.\ was partially supported by the NCN grant DEC-2011/03/B/ST1/00407 and by the EU Marie Curie IRSES program, project ``AOS'', No.\ 318910.

Finally, A.C., F.G., F.S., and Y.T.\ thank the Erwin Schr\"odinger International Institute for 
Mathematical Physics (ESI), Vienna, Austria, for funding support for this collaboration.
%%%%%%%%%%%%%%%%%%%%%%%%%%%%%%%%%%%%%

%%%%%%%%%%%%%%%%%%%%%%%%%%%%%%%%
%%%%%%%%%%%%%%%%%%%%%%%%%%%%%%%%
 
\end{document}